\documentclass{article}
\usepackage{graphicx}
\usepackage[letterpaper, margin=0.7in]{geometry}
\usepackage{amsmath}
\usepackage{mathtools,subcaption}
\usepackage{amsthm}
\usepackage{float}\usepackage{cite}
\usepackage{url}
\newtheorem{theorem}{Theorem}
\newtheorem{lemma}{Lemma} 
\newtheorem{prop}[theorem]{Proposition} 
\usepackage{tikz}
\usetikzlibrary{positioning}
\usepackage{comment}
\usepackage{supertabular}
\usepackage{multicol}
\usepackage{subcaption}
\usepackage{hyperref}
\usepackage[ruled,linesnumbered]{algorithm2e}

\usepackage{tabularx}
\usepackage{adjustbox}
\usepackage{multirow}
\usepackage{booktabs}
\usepackage{siunitx}
\usepackage{authblk}
\usepackage{array}
\usepackage{longtable}
\usepackage{ upgreek }
 \usepackage{algorithm2e}
\usepackage{mathrsfs}
\usepackage{mathtools}
\usepackage{cuted}
\usepackage{caption}
\usepackage{amssymb}
\usepackage{textcomp}
\usepackage[active]{srcltx}

\theoremstyle{definition}
\newtheorem{definition}{Definition}
\newtheorem{rem}{Remark}

\newcommand{\comments}[1]{}
\newcommand{\R}{\mathbb{R}}
\newcommand{\mbf}{\mathbf}
\newcommand{\fk}{\mathfrak{K}}

\newcommand{\norm}[1]{\left\lVert #1\right\rVert}

\title{A Data-Assimilation-Augmented Optimization Framework for Parameter Estimation in Dynamical Systems} 

\author[1]{Muhammad Jalil Ahmad\thanks{Corresponding author: Muhammad Jalil Ahmad, Department of Mathematics and Statistics, University of Maryland, Baltimore County, 1000 Hilltop Circle, Baltimore, MD 21250, USA. Email: \texttt{LS47576@umbc.edu}.}} 

\author[1]{Animikh Biswas} \author[1]{Kathleen Hoffman} \affil[1]{Department of Mathematics and Statistics, University of Maryland, Baltimore County, USA}

\date{}

\begin{document}
\maketitle

%===================================  Abstract  ================================

\begin{abstract}
Parameter estimation in nonlinear dynamical systems from observational data is a fundamental inverse problem with applications in many disciplines such as epidemiology, systems biology, climate science, and related fields. In practice, this is further complicated by the fact that observational data are often noisy, sparse, and available only for a subset of the state variables. Furthermore,   the initial condition may be unknown or inaccurate, causing further complications for chaotic systems with sensitive dependence on initial conditions. In this work, we develop a data-assimilation-augmented optimization framework for parameter estimation in ordinary differential equation systems using partial state observations. The method introduces a nudged system driven by the available observed component and estimates the unknown parameters by minimizing a cost functional, defined as a time-delayed mismatch between the observations and the corresponding observed component of the nudged solution over the admissible parameter space. Since the nudged system can be arbitrarily initialized, this approach eliminates the dependence on accurate initial-state information.

Using the Lorenz--63 system as a test case, we establish theoretical results showing synchronization of the nudged solution under parameter agreement, stability under parameter mismatch, and well-posedness of the data-to-parameter inverse map under suitable nondegeneracy conditions. Structural identifiability, practical identifiability, and Sobol sensitivity analyses are incorporated to assess which parameters can be reliably estimated from the available observations. Numerical experiments in both chaotic and non-chaotic regimes show that the proposed framework accurately recovers parameters from noisy partial observations. Comparisons with an existing on-the-fly parameter learning method and with Bayesian MCMC estimation demonstrate that the proposed method remains accurate under partial observations and higher noise levels while requiring substantially lower computational cost.

\end{abstract}



%===================================  Section 1: Introduction  ================================

\section{Introduction}
Differential equations (DEs) are the mathematical backbone of models used to describe physical, chemical, biological, and even financial systems. The effectiveness of these models is determined by their ability to replicate real-world behavior. A primary challenge when using such models is the need for accurate initial conditions and parameter estimates. However, in many applications, these quantities are unknown and must be inferred from observational data. This inference problem is often challenging because real-world data may be sparse, noisy, or available only for a subset of the state variables. These limitations can lead to inverse problems that are ill-posed, nonconvex, and computationally demanding, particularly for nonlinear dynamical systems.

These challenges involve several interconnected issues: whether the available data contain enough information to determine the unknown parameters, how strongly different parameters influence the observed dynamics, and how observational data can be incorporated into the model in a stable and computationally efficient way. We therefore organize the related work around parameter identifiability, parameter sensitivity, parameter estimation methods for dynamical systems, limitations of these methods and motivation for data-assimilation.

%==============================   Related Work  =============================================

\subsection{Related Work and Motivation}

%============= Parameter Identifiability  =============

   \paragraph{Parameter Identifiability:} The question of whether it is possible to uniquely estimate parameters for a given model and data set is addressed through the concept of \textit{parameter identifiability} \cite{Mio, Ljung1994, walter1997, ANSTETTCOLLIN2020139}. This notion is typically examined in two distinct forms. \textit{Structural identifiability} concerns the uniqueness of parameter estimates under ideal conditions---that is, assuming noise-free, continuous, and complete measurements \cite{Mio,walter1997,ANSTETTCOLLIN2020139, meshkat2025structural}. In this sense, structural identifiability characterizes the maximum amount of parameter information that can be extracted from a specified observation scheme and serves as a necessary prerequisite for reliable estimation from real-world data. 
    
    \textit{Practical identifiability}, on the other hand, accounts for realistic limitations such as measurement noise, finite sampling, and imperfect observability, and investigates whether these factors hinder the ability to reliably estimate parameters in practice \cite{Mio, Saucedo2024, tuncer2018, RodriguezFernandez2006, Banks2014, Venzon1988, Jacquez1990, Eisenberg2014}. As discussed in \cite{Saucedo2024}, such limitations may render a structurally identifiable model practically non-identifiable. Conversely, in the idealized limit of noise-free and sufficiently informative data, structural identifiability represents the best-case scenario for practical parameter recovery \cite{Wieland2021}. A more detailed discussion of parameter identifiability is provided in Sections~\ref{sec:sidentifiability1} and~\ref{sec:sidentifiability2}.

%============= Parameter Sensitivity  =============
    
    \paragraph{Parameter Sensitivity:} A significant challenge that can arise in the parameter estimation process is an imbalance between the number of unknown parameters and the number of observable states \cite{Villaverde2016, Eisenberg2014, Jacquez1990}. However, it is crucial to recognize that not all parameters significantly influence the observables in a dynamical system. This issue is addressed through the concept of \textit{parameter sensitivity}, which refers to how changes in a model\textquotesingle s parameters influence its outputs \cite{marino2008, zhang2015}. Here, the term ``output" may refer to individual state variables, linear combinations thereof, or other derived quantities, including those that are not themselves state variables. For instance, in epidemiological modeling, sensitivity analysis is frequently conducted with respect to the basic reproduction number, \( R_0 \), which represents the average number of secondary cases generated by a single infectious individual. Numerous methods are available in the literature to assess which parameters are sensitive in a model (see, e.g., \cite{bidah2020, sobol2001, hoare2008sasat, saltelli1990nonparametric, saltelli2005chemical, kleijnen1999scatterplots}). A more detailed discussion on parameter sensitivity is provided in Section~\ref{sensitivity}. Since non-sensitive parameters do not significantly impact the dynamics of the system's observed output, it may not be necessary to estimate them with high precision, and they may therefore be excluded from the parameter estimation process for that particular observation setting.

    %============= Parameter Estimation  =============

    \paragraph{Parameter Estimation:}
    The literature presents a variety of methods for parameter estimation in dynamical systems. Some approaches leverage neural networks \cite{YANG2025112649, PSOparamestimate, Ahmad2024, param1, alavanifractional}, which often require large datasets---an assumption that is impractical in many real-world scenarios. Classical approaches often formulate parameter estimation as a nonlinear least-squares or inverse problem, in which the parameters are chosen to minimize the discrepancy between observed data and model predictions \cite{Bard1974, Aster2005, Liang01122008, param2}. These methods are widely used and mathematically well established, but they may become computationally expensive for nonlinear dynamical systems because each evaluation of the cost function typically requires solving the forward model.

    In the Bayesian framework, the mismatch between model predictions and observed data is modeled probabilistically through a likelihood function, which facilitates parameter estimation through either maximum likelihood estimation or fully Bayesian inference. In the latter case, prior distributions are introduced to incorporate prior information and regularize the inverse problem. In practice, Bayesian inference is often implemented using Markov chain Monte Carlo (MCMC) methods \cite{BoerschSupan2017, Ghasemi2011, MCMCLorenz, huang2006hierarchical}, which enable sampling from the posterior distribution of the parameters. However, these methods can be computationally expensive, as they require numerous forward simulations of the model.

%============= Motivation for Data Assimilation  =============
    
    \paragraph{Motivation for Data Assimilation in Parameter Estimation:}
Despite the variety of available parameter estimation methods, several challenges remain. Neural-network-based approaches may require large training data sets, while classical least-squares, Bayesian, and MCMC-based methods can become computationally expensive because they often require many forward simulations of the model. In addition, many parameter estimation methods are sensitive to the choice of initial conditions, which is especially problematic in chaotic systems where small initial errors can lead to large deviations in the trajectory. Another major difficulty is that full-state observations are rarely available in practice; data are often sparse, noisy, indirect, or restricted to only a subset of the state variables.

Data assimilation (DA) provides a natural framework for addressing these issues because it combines observational data with the governing dynamics to stabilize state reconstruction and prediction \cite{LawStuartZygalakis2015,Kalnay2003,Asch2016,carrassi2018}. In this work, we focus on nudging, a continuous data assimilation method originally used in weather prediction and geophysical models \cite{stauffer1990,Anthes1974,Blayo1994}, and later developed into a rigorous mathematical framework by Azouani, Olson, and Titi \cite{Azouani2014} for geophysical systems. Nudging introduces a feedback control term into the model equations, driving the assimilated solution toward the observed components of the reference solution.

The use of nudging for parameter estimation has been studied in several recent works. Carlson et al.~\cite{Carlson2022} proposed an on-the-fly parameter learning algorithm in which unknown parameters are adjusted using the persistent synchronization error between the true and nudged systems. Martinez et al.~\cite{Martinez2024} developed relaxation-based schemes, including relaxation least-squares and relaxation Newton iteration methods, for simultaneous state and parameter recovery in dissipative systems. More recently, Newey et al.~\cite{NEWEY2025114121} interpreted related CDA parameter-learning algorithms through finite-dimensional root-finding and optimization methods, including Newton, Gauss--Newton and Levenberg--Marquardt-type approaches.

The work of Newey et al.~\cite{NEWEY2025114121} is particularly relevant to our work because it also considers an observable-error cost functional arising from a nudged system. However, 
their approach involves finding the zero of the gradient of this cost functional with respect to the parameter pointwise in time and uses successive temporal updates of the parameter estimates thus obtained. This formulation is primarily designed for on-the-fly parameter updates based on a single large-time observable error, which necessitates observing the variables the equations of which contain the parameter. By contrast,
 the present work formulates parameter recovery as a direct minimization problem for a  time-averaged cost functional, where this restriction is not present.  Furthermore,  our simulation study indicates that the synchronization error may oscillate in time, especially away from the true parameter value, and particularly  in the presence of noise in the observed data.  Additionally, our computations indicate that the cost functional is nonconvex; thus solving for the gradient may converge to a local minimum, particularly in case of observational error in the data.  

While the comparison with Newey et al.~\cite{NEWEY2025114121} clarifies the distinction between our approach and theirs, the inverse-problem viewpoint adopted here is closely related to the determining-map framework for parameter recovery developed in the PDE setting by Biswas and Hudson~\cite{biswas2023determining}. In the context of the two-dimensional incompressible Navier--Stokes equations, they formulated viscosity recovery from finitely many modal observations as an inverse problem based on the determining map. Their work used a nudging-generated cost functional and established rigorous conditions for uniqueness, well-posedness, control of the parameter error by the cost functional, and convergence of an associated algorithm. The present work is inspired by this viewpoint, but extends it from a single-parameter PDE setting to a multi-parameter ODE setting with partial state observations.

We develop and analyze the method for the Lorenz--63 system as a canonical nonlinear test problem with both chaotic and non-chaotic regimes. Although the unknown parameters in Lorenz--63 enter the equations through linear parameter factors, the proposed framework does not require parameters to enter the model linearly. The same nudging-based optimization formulation can be applied to ODE models in which the parameters appear nonlinearly.

For the Lorenz--63 system, we establish synchronization of the nudged system, stability under parameter mismatch, well-posedness of the data-to-parameter inverse map away from explicitly identified degenerate trajectories, and estimates showing that small values of the time-delayed cost functional imply small parameter error under suitable nondegeneracy conditions.  These theoretical results provide the foundation for the numerical experiments in both chaotic and non-chaotic regimes with noisy partial observations. In the numerical part, we examine the performance of the proposed method under different observation settings and noise levels, and compare the recovered parameters with those obtained from an existing data-assimilation-based parameter estimation method. We also compare our results with Bayesian MCMC parameter estimation.

\comments{
We provide an analytical proof of the well-posedness of the inverse problem as well as stability estimates and error bounds  based on the value of the minimized cost function. }

%============= Contributions  =============
 
\subsection{Contributions}

The main contribution of this work is a data-assimilation-augmented optimization framework for estimating unknown parameters in nonlinear ODE systems from partial state observations. The approach uses a nudged system driven by the available observed component and estimates the unknown parameters by minimizing a time-delayed cost functional, defined by the mismatch between the observations and the corresponding observed component of the nudged solution.

The proposed framework is supported by theoretical analysis, including well-posedness and stability estimates. For the Lorenz--63 system, we study synchronization, stability under parameter mismatch, and the well-posedness of the inverse map from observations to parameters. We also identify dynamical situations in which parameter recovery may fail, thereby clarifying when the proposed inverse problem is well-posed and when it is degenerate.

The main contributions of this work are as follows:
\begin{itemize}
    \item We formulate a nudging-based optimization framework for parameter estimation in finite-dimensional nonlinear ODE systems from partial state observations.

    \item We extend the determining-map-based philosophy of Biswas and Hudson~\cite{biswas2023determining}, developed for viscosity recovery in the Navier--Stokes equations, from a single-parameter PDE setting to a multi-parameter ODE setting.

    \item We introduce a time-delayed, time-averaged cost functional based on the mismatch between the observed data and the corresponding observed component of the nudged solution.

    \item We prove synchronization results for the Lorenz--63 nudged system under correct parameter values, showing convergence of the nudged trajectory to the true trajectory from arbitrary initial conditions.

    \item We establish stability estimates under parameter mismatch, showing that the nudged trajectory converges to a neighborhood of the true trajectory whose size is controlled by the parameter error.

    \item We characterize the well-posedness of the data-to-parameter inverse map in terms of the underlying Lorenz dynamics and identify degenerate trajectories and equilibrium-related cases where parameter recovery may fail or become unstable.

    \item We derive local error estimates showing that, under appropriate nondegeneracy assumptions, small values of the time-delayed cost functional imply small error in the recovered parameter vector.

    \item We extend the analysis from full observation to partial observation by showing that, near the true parameter, the nudged trajectory remains close enough to preserve the relevant well-posedness and stability mechanisms.

    \item We incorporate structural identifiability, practical identifiability, and Sobol sensitivity analysis into the parameter estimation workflow to assess which parameters can be reliably estimated from the available observations.

    \item We validate the proposed framework on the Lorenz--63 system in both chaotic and non-chaotic regimes, using noisy partial observations, and compare the resulting estimates with existing on-the-fly CDA parameter-learning methods and Bayesian MCMC estimation.
\end{itemize}

The remainder of this paper is organized as follows. Section~\ref{sec:identifiability} reviews key aspects of parameter estimation, including structural identifiability, practical identifiability, well-posedness of the data-to-parameter inverse map, and parameter sensitivity. Section~\ref{sec3} introduces the data-assimilation framework and formulates the nudging-augmented parameter estimation approach. Section~\ref{lorenzsystem} presents the Lorenz--63 system and analyzes the corresponding nudged system, including synchronization, stability under parameter mismatch, and well-posedness properties under full and partial observations. Section~\ref{sec:results} presents the numerical results for the Lorenz--63 system, including forward prediction, structural and practical identifiability, Sobol sensitivity analysis, and parameter estimation in both chaotic and non-chaotic regimes with noisy observations. Finally, Section~\ref{discuss} summarizes the main findings, compares the proposed approach with an existing data-assimilation-based parameter estimation method and Bayesian MCMC estimation, and outlines directions for future work.

%============================= Section 2: Aspects of Param Est  ===================================

\section{Aspects of Parameter Estimation} 
\label{sec:identifiability}

Parameter estimation for dynamical systems involves several related questions. Model calibration refers to the process of selecting parameter values so that model trajectories agree with observed data \cite{Villaverde2022}. However, before carrying out calibration, one must understand whether the available observations contain sufficient information to determine the unknown parameters. This leads to the notions of structural identifiability, and practical identifiability. In addition, parameter sensitivity analysis provides information about how strongly each parameter influences the model output. 

A related but conceptually distinct issue is the stability of the inverse map from observed data to parameters \cite{liyanage2026tutorial}. In other words, even when identifiability guarantees uniqueness, one must still ask whether small changes in the observed data lead only to controlled changes in the recovered parameters. We refer to this stability requirement as the well-posedness of the data-to-parameter inverse map. Although this viewpoint is closely connected to practical identifiability and inverse-problem stability, it is often not stated explicitly in the dynamical-systems parameter-estimation literature. We include it here to clarify the role it plays in the proposed framework and to set the stage for the numerical and theoretical analyses below. This section reviews these concepts and explains how they are used in the numerical investigations below.

\subsection{Structural Parameter Identifiability} \label{sec:sidentifiability1}

Structural identifiability analysis investigates whether the parameters of a model can be uniquely determined from the observed output, based only on the model structure and the specified observation scheme, without accounting for limitations in data quantity or measurement noise. More precisely, it asks whether two parameter values that generate the same observed output must necessarily be identical. Thus, structural identifiability is a property of the parameter-to-observation map under idealized conditions, namely continuous and noise-free observations, and provides a theoretical prerequisite for reliable parameter estimation from data.

Structural identifiability is further classified into global and local structural identifiability \cite{Ljung1994,walter1997, ANSTETTCOLLIN2020139}. Consider a general model
\begin{equation}
    \begin{cases}
        \dot{{\bf x}}_{\mbf \theta}(t) = f({\bf x}_{\mbf \theta}(t), {\mbf \theta}), \\
        {\bf x}_{\mbf \theta}(0) = {\bf x}_0, \\
        {\cal O}_{\mbf \theta}(t) = g({\bf x}_{\mbf \theta}(t), {\mbf \theta}),
    \end{cases}
    \label{eq15}
\end{equation}
where the state vector ${\bf x}_{\mbf \theta}(t) \in \mathbb{R}^n$ represents the state of the system at time $t$, ${\mbf \theta} \in {\Theta} \subseteq \mathbb{R}^p$ is the vector of parameters, ${\bf x}_0 \in \mathbb{R}^n$ is the initial condition, and  \({\cal O}_{\mbf \theta}(t) \in \mathbb{R}^q\) denotes the observable outputs, with $q \leq n$.  The functions $f$ and $g$ are taken to be defined on an open set and are locally Lipschitz. In particular, this implies that ${\cal O}$ is a continuous function of the state $\bf x$.
The initial condition is assumed to be fixed throughout the identifiability analysis. When the initial condition is known, it is suppressed from the notation. When it is unknown and must be estimated, it is treated as part of an extended parameter vector \([{\mbf \theta}^\top,{\bf x}(0)^\top]^\top\). When the initial condition is included, the observation is denoted by ${\cal O}_{\mbf \theta}(t; {\bf x}(0))$. For simplicity, we omit this dependence unless it is explicitly needed. 

Definitions~\ref{def1}--\ref{def3} are stated over a time interval $I \subset [0,+\infty)$ \cite{ANSTETTCOLLIN2020139, walter1997}. For instance, one may take $I = [0,+\infty)$ or a finite interval $I = [a,b] \subset [0,+\infty)$. Definitions~\ref{def1} and~\ref{def2} characterize the structural identifiability of a single parameter component ${\theta}^{(i)}$, for $i \in \{1,\ldots,p\}$.

\begin{definition}\label{def1}
Following \normalfont \cite{walter1997}, the parameter $\theta^{(i)}$ is said to be \emph{structurally locally identifiable} if, for almost all $\theta_1 \in \Theta$, there exists a neighborhood $B(\theta_1)$ such that
\[
\theta_2 \in B(\theta_1), \quad 
{\cal O}_{\theta_1}(t)={\cal O}_{\theta_2}(t) \ \text{for all } t \in I
\;\Rightarrow\; 
\theta_2^{(i)} = \theta_1^{(i)}.
\]
The model \eqref{eq15} is said to be \emph{structurally locally identifiable} if all parameters $\theta^{(i)}$, $i=1,\ldots,p$, are structurally locally identifiable.
\end{definition}

\begin{definition}\label{def2}
As formulated in \normalfont \cite{walter1997}, the parameter $\theta^{(i)}$ is said to be \emph{structurally globally identifiable} if, for almost all $\theta_1,\theta_2 \in \Theta$,
\[
{\cal O}_{\theta_1}(t)={\cal O}_{\theta_2}(t) \ \text{for all } t \in I
\;\Rightarrow\; 
\theta_2^{(i)} = \theta_1^{(i)}.
\]
The model \eqref{eq15} is said to be \emph{structurally globally identifiable} if all parameters $\theta^{(i)}$, $i=1,\ldots,p$, are structurally globally identifiable.
\end{definition}

\noindent It follows immediately that structural global identifiability implies structural local identifiability.

\begin{definition}\label{def3}
The model \eqref{eq15} is said to be \emph{structurally non-identifiable} if there exists at least one parameter $\theta^{(i)}$ that is not structurally identifiable \normalfont \cite{walter1997}.
\end{definition}

We note that in the original formulation \cite{walter1997}, the time domain over which the equality ${\cal O}_{\theta_1}(t)={\cal O}_{\theta_2}(t)$ is required is not explicitly specified. In the context of dynamical systems, this equality is naturally interpreted over a time interval. To avoid any potential ambiguity and to emphasize that identifiability is not a pointwise notion in time, we explicitly require the equality ${\cal O}_{\theta_1}(t)={\cal O}_{\theta_2}(t)$ to hold for all $t \in I$, rather than at a single time instance. We also refer the reader to the more recent review \cite{ANSTETTCOLLIN2020139}, which provides a comprehensive overview of various notions of structural identifiability and discusses the relationships among these notions.

The differential algebra approach to structural identifiability and the associated computational algorithms have been developed in \cite{audoly2001global,lunel2001parameter,Ljung1994,ollivier1990probleme}; see also the review \cite{Mio}. The main idea is to eliminate the unobserved state variables and obtain one or more input--output differential equations involving only the measured output, its derivatives, known inputs when present, and the unknown parameters. In the notation of Eq.~\eqref{eq15}, suppose that such an input--output relation has the form
\[
h\bigl(
{\cal O}_{\theta}(t),
{\cal O}_{\theta}^{(1)}(t),
\ldots,
{\cal O}_{\theta}^{(\ell)}(t),
\theta
\bigr)=0,
\]
where \({\cal O}_{\theta}^{(j)}\) denotes the \(j\)-th derivative of the observed output and \(h\) is an algebraic, or polynomial, expression in the observed output, its derivatives, and the unknown parameters. The existence of such an input--output equation is an important first step, since it shows that the unknown parameters enter an algebraic-differential relation determined by the observed output. However, the existence of this equation alone does not automatically imply structural identifiability. One must still verify that the resulting algebraic equations determine the parameter vector uniquely.

In many models, a single input--output equation does not provide enough independent equations to determine all unknown parameters. One possible approach is to differentiate the input--output equation further in time until a sufficient number of equations is obtained. This is the basis of the implicit-function-theorem approach of Xia and Moog \cite{XiaMoog2003}, in which an identification function is constructed from the input--output relation and its derivatives. If the Jacobian of this identification function with respect to the unknown parameters has full rank at the true parameter value, then local structural identifiability follows from the implicit function theorem. However, this procedure may require high-order derivatives of the measured output. Such derivatives may be difficult to justify analytically unless sufficient smoothness is known, and they are also difficult to evaluate reliably from finite or noisy data.

An alternative, proposed by Wu et al.~\cite{Wu2008}, is to evaluate the same input--output relation at several distinct observation times rather than repeatedly differentiating it. To describe this idea, let \(\theta^*\) denote the true parameter vector and define
\[
Y_j^*
=
\bigl(
{\cal O}_{\theta^*}(t_j),
{\cal O}_{\theta^*}^{(1)}(t_j),
\ldots,
{\cal O}_{\theta^*}^{(\ell)}(t_j)
\bigr),
\qquad j=1,\ldots,J,
\]
where \(\{t_j\}_{j=1}^J\) is a finite collection of observation times. Evaluating the input--output equation at these times gives
\[
h(Y_j^*,\theta)=0,
\qquad j=1,\ldots,J.
\]
Equivalently, define the identification map
\[
\Phi(\theta;Y^*)
=
\begin{pmatrix}
h(Y_1^*,\theta)\\
\vdots\\
h(Y_J^*,\theta)
\end{pmatrix}.
\]
Then local structural identifiability can be studied through the rank of the Jacobian of this map with respect to the parameter vector. In particular, if
\[
\operatorname{rank}D_{\theta}\Phi(\theta^*;Y^*)=p,
\]
where \(p\) is the dimension of the parameter vector, then the identification equations have full rank with respect to the unknown parameters at \(\theta^*\). In the square case, this condition reduces to
\[
\det D_{\theta}\Phi(\theta^*;Y^*)\neq 0.
\]
By the implicit function theorem, this full-rank condition implies local uniqueness of the parameter vector. More precisely, there exists a neighborhood \(U\) of \(\theta^*\) such that
\[
\Phi(\theta;Y^*)=\Phi(\theta^*;Y^*),
\qquad \theta\in U,
\]
implies
\[
\theta=\theta^*.
\]

Thus, after deriving an input--output equation, one must still verify an appropriate nondegeneracy condition. The multiple-time-point formulation is useful because it can reduce the need for repeatedly differentiating the output equation, while still producing enough algebraic equations to test local structural identifiability. This point is especially relevant when only limited smoothness of the observed trajectory is available, or when higher-order derivatives of noisy data would be numerically unstable.

Symbolic tools such as {\em DAISY} and {\em SIAN} provide algorithmic implementations of differential-algebraic methods for structural identifiability analysis. In particular, they can be used to eliminate unobserved state variables and derive input--output relations involving only the observed variables, their derivatives, known inputs when present, and the unknown parameters. DAISY was among the first software tools to employ differential algebra techniques for structural identifiability analysis \cite{Bellu2007}. SIAN is another symbolic tool for structural identifiability analysis that combines differential algebra ideas with Taylor-series-based computations \cite{Hong2019}. These tools are useful because they automate part of the symbolic analysis that would otherwise be difficult to carry out by hand. At the same time, the outcome of any such analysis should be interpreted in terms of the underlying identifiability question: whether the available output relations determine the unknown parameters uniquely, either globally or locally. Thus, input--output equations, differential-algebraic reductions, and software-generated identifiability conclusions should be viewed as part of the same structural-identifiability framework, with the final conclusion depending on the corresponding uniqueness or rank conditions for the model under consideration.

\subsection{Well-posedness of the Data-to-Parameter Inverse Map}
\comments{Consider the dynamical system $x_p(\cdot)$ with observation $y_p(\cdot)$ given in \eqref{eq15}, where $p$ denotes the parameter.}

While identifiability addresses the issue of injectivity of the map from parameter to data, 
another important question in parameter estimation is the well-posedness of the inverse map from the observed data to the parameters. In other words, how reliably parameters of a dynamical system can be recovered from observed data can be quantified by an estimate of the form
\[
\| \theta -  \theta'\| \le \omega (\|{\cal O}_{\theta}(\cdot)-{\cal O}_{\theta'}(\cdot)\|_{\cal B}),
\]
where $\omega$ is a suitable modulus of continuity and ${\cal B}$ denotes a suitable Banach space containing the observations \({\cal O}_\alpha(\cdot)\) for all parameters \(\alpha\in\Theta\). This concept is crucial for the stability of any parameter estimation algorithm. While identifiability has been widely addressed, the issue of well-posedness of the inverse map has received much less attention. Practical parameter identifiability analysis, described in Section~\ref{sec:sidentifiability2}, can be regarded as a tool for addressing the issue of well-posedness of the data-to-parameter inverse map numerically \cite{PracIDentShun}.

%============= Section 2.3: Practical Parameter Identifiability  =============

\subsection{Practical Parameter Identifiability}\label{sec:sidentifiability2}

While structural identifiability characterizes the theoretical ability to recover parameters from ideal data, practical identifiability assesses whether parameters can be reliably estimated from real, noisy, and finite data. To analyze practical identifiability, we employ a Monte Carlo (MC) approach \cite{Mio, tuncer2018}. The data for this analysis are synthetically generated by numerically solving the system given in Eq.~\eqref{eq15}, using true parameter values $\theta^*$. For this analysis, the initial condition $x_0$ is assumed to be known. The Monte Carlo (MC) procedure used for practical identifiability analysis consists of the following steps:

\begin{enumerate}
    \item Integrate the system in Eq.~\eqref{eq15} using the true parameter vector $ \theta^*$ and some fixed initial data to generate the reference output $g({\bf x}(t), \theta^*)$ at discrete observation times $t_i$, for $i=1,\dots,n$.

    \item Construct $M=1,000$ synthetic data sets by perturbing the reference output with measurement noise. Assuming Gaussian noise with zero mean and standard deviation $\epsilon$, the noisy observations are generated as
    \begin{equation}
    \label{eq4}
    {\cal O}_{i,j}=g({\bf x}(t_i), \theta^*)\bigl(1+\xi_{i,j}\bigr),
    \qquad
    \xi_{i,j}\sim \mathcal{N}(0,\epsilon).
    \end{equation}
    The standard deviation $\epsilon$ determines the noise magnitude; for example, $\epsilon=0.02$ corresponds to $2\%$ noise. Depending on the application, observational noise may be modeled additively or multiplicatively. In this work, a multiplicative noise model is adopted throughout for consistency.

    \item For each synthetic data set, estimate a parameter vector $ \theta_j$ by fitting the model to the perturbed observations through least-squares minimization:
    \begin{equation}
         \theta_j \approx \arg\min_ \theta \sum_{i=1}^n \left({\cal O}_{i,j}-g({\bf x}(t_i), \theta)\right)^2.
    \end{equation}
    The resulting optimization problem is solved using a numerical optimization algorithm, namely the Nelder--Mead algorithm \cite{nelder1965simplex}.

    \item Quantify estimation accuracy for each parameter using the average relative error (ARE), defined by
\begin{equation}
    \mathrm{ARE}( \theta^{(k)})=
    \frac{100\%}{M}
    \sum_{j=1}^{M}
    \frac{|\theta^{*(k)}-\theta^{(k)}_{j}|}{|\theta^{*(k)}|},
\end{equation}
    where $ \theta^{(k)}$ represents the $k$-th parameter in the set $ \theta$, $ \theta^{*(k)}$ is the $k$-th parameter in the true parameter vector $ \theta^*$, and $ \theta_j^{(k)}$ denotes the $k$-th element of $ \theta_j$.
\end{enumerate}
This process is repeated across varying noise levels. Following \cite{tuncer2018, Saucedo2024}, a parameter is considered practically identifiable if its Average Relative Error (ARE) is less than or equal to the measurement error, \(\epsilon\).

%============= Section 2.4: Parameter Sensitivity  =============

\subsection{Parameter Sensitivity}
\label{sensitivity}

Parameter sensitivity analysis is used to quantify how strongly variations in the model parameters affect a chosen model output \cite{marino2008, zhang2015}. In the present work, the output may be an observed state component, a collection of observed components, or another quantity derived from the state variables. Sensitivity information is useful for parameter estimation because parameters that have little influence on the available output are typically more difficult to estimate reliably from that observation scheme. Thus, sensitivity analysis provides a complementary diagnostic to identifiability analysis: it helps indicate which parameters are expected to be influential in the data and which parameters may be weakly informed by the observations.

A widely used tool for determining parameter sensitivity is the Sobol sensitivity analysis. Briefly, the Sobol methodology decomposes the total variance of the model output into contributions associated with individual parameters and combinations of parameters \cite{sobol2001}. The resulting \emph{first-order sensitivity index} measures the contribution of a single parameter acting alone, while the \emph{total-order sensitivity index} accounts for both the direct effect of that parameter and all interaction effects involving it. Parameters with small sensitivity indices have a limited influence on the output, whereas parameters with large indices play a dominant role in shaping the observed dynamics. We refer the reader to \cite{sobol2001,zhang2015} for a complete mathematical formulation of the variance decomposition and index definitions.

%============= Section 3  =============

\section{Parameter Estimation Augmented by Data Assimilation}
\label{sec3}

This section introduces the data assimilation framework used for parameter estimation in this work. The main idea is that full-state observation represents the ideal setting for trajectory-based parameter estimation, since all components of the model state are directly available for comparison with the model output. In many applications, however, only partial observations are available. Data assimilation provides a mechanism for using these partial observations, together with the model dynamics, to reconstruct the unobserved components of the state. Thus, if the available observations are sufficiently informative, a data-assimilation-based estimator may be expected to approach the performance of a full-observation estimator.

We first review nudging as a continuous data assimilation method for reconstructing the state of a dynamical system from partial observations. We then describe how this framework is incorporated into a parameter estimation procedure by defining a nudging-based cost functional and minimizing it over the admissible parameter set. This approach formulates parameter recovery as an optimization problem while retaining the stabilizing effect of data assimilation.

%============= Section 3.1: Data Assimilation  =============

\subsection{Data Assimilation}

Using the nudging approach, one introduces a replica of the original system, called the \textit{nudged system}, and adds a feedback term to the equations corresponding to the observed components. This feedback term penalizes the mismatch between the observed data and the corresponding components of the nudged solution.

We now describe the nudging formulation used in this work. Consider the nonlinear dynamical system given in Eq.~\eqref{eq15}. Assume that only partial observations of the state are available:
\begin{equation}
\label{eq:general_observation_nudging}
    {\cal O}(t)=H{\bf x}(t),
\end{equation}
where ${\cal O}(t)\in\mathbb{R}^q$ denotes the observed data and $H:\mathbb{R}^n\to\mathbb{R}^q$ is a linear observation operator, with $q\leq n$. 
A typical example of an observation operator is an orthogonal projection onto a subspace of the phase space. In this case, denoting by $P_{\text{Ran}\, H}$ the orthogonal projection onto the range of $H$, we have
\[
H=H^\top =H^\top H= P_{\text{Ran}\, H}
\quad \text{and} \quad
{\cal O}(t)=P_{\text{Ran}\, H}{\bf x}(t).
\]

To incorporate the observations into the model dynamics, we introduce a replica of Eq.~\eqref{eq15}, referred to as the nudged system:
\begin{equation}
\label{eq:general_nudged_system}
    \dot{\tilde{{\bf x}}}(t)
    =
    F(\tilde{{\bf x}}(t),\tilde{\theta})
    +
     {\cal B}\bigl({\cal O}(t)-H\tilde{{\bf x}}(t)\bigr),
    \qquad
    \tilde{{\bf x}}(0)=\tilde{{\bf x}}_0 .
\end{equation}
Here, ${\cal B}:\mathbb{R}^q\to\mathbb{R}^n$ is a chosen linear feedback operator that incorporates the observation residual into the state equations, $\tilde{{\bf x}}(t)$ denotes the nudged state, and $\tilde{\theta}$ denotes the parameter values used in the nudged model. The nudging term $    {\cal B}\bigl({\cal O}(t)-H\tilde{{\bf x}}(t)\bigr)$ acts as a feedback correction that penalizes the discrepancy between the observed data and the corresponding observations of the nudged solution. In many applications, the true initial condition ${\bf x}_0$ is unknown. The nudged system can therefore be initialized independently of the true system; for instance, one may choose $\tilde{{\bf x}}(0)=0$.

The goal is to choose, when possible, a feedback operator ${\cal B}$ such that
\[
\lim_{t \to \infty} \|{\bf x}(t) - \tilde{{\bf x}}(t)\| = 0,
\]
regardless of how the nudged system is initialized. In control theory terminology, Eq.~\eqref{eq:general_nudged_system} is a type of Luenberger observer \cite{Luenberger}.

When $H$ is an orthogonal projection, a suitable choice often turns out to be ${\cal B}=\mu H^\top$, where \(\mu\) is the nudging coefficient. Therefore, in this case, we have
\[
{\cal O}(t)= P_{\text{Ran}\, H}{\bf x}(t)= H^\top H{\bf x}(t)
\quad \text{and} \quad
{\cal B}\bigl({\cal O}(t)-H\tilde{{\bf x}}(t)\bigr)=\mu 
P_{\text{Ran}\, H}({\bf x}(t) - \tilde{{\bf x}}(t)).
\]
Thus, the nudging term acts on the observed components of the state and drives the nudged solution toward the reference trajectory through the available observations. As we shall see later, in many situations $\mu$ can be chosen sufficiently large to ensure 
\[
\lim_{t \to \infty} \|{\bf x}(t)-\tilde{{\bf x}}(t)\|=0.
\]

The effectiveness of this approach depends on whether the observation operator \(H\) captures enough information about the underlying dynamics. This is closely related to the theory of determining functionals in geophysical fluid dynamics \cite{FoiasTiti1991,JonesTiti1993,OlsonTiti2003,FoiasJollyKravchenkoTiti2012}. More precisely, if the observation operator has the \emph{determining property} given in Definition~\ref{deter}, then partial observations can, in principle, determine the full asymptotic dynamics. This is precisely the mechanism that motivates the use of data assimilation for parameter estimation from partial observations.

\begin{definition} \label{deter}
Following \normalfont\cite{FoiasTemam1984,FoiasTiti1991}, let ${\cal H} = \mathbb{R}^n$ be the state space of the dynamical system and let $\{S(t)\}_{t \ge 0}$ denote the associated flow. 
The observation operator $H$ is called determining if, for any two trajectories $u(t)=S(t)u_0$ and $v(t)=S(t)v_0$, the condition
\[
\lim_{t \to \infty} \|H(u(t)-v(t))\| = 0
\quad \text{implies} \quad
\lim_{t \to \infty} \|u(t)-v(t)\| = 0.
\]
\end{definition}

The concept of determining operators is related to detectability in linear control theory \cite{sontag2013mathematical}. In the parameter estimation setting considered here, this property indicates that partial observations may contain enough information to reconstruct the relevant state dynamics. Therefore, data assimilation provides a way to reduce the gap between partial-observation parameter estimation and the ideal full-observation case.

%============= Section 3.2: Nudging-Augmented Parameter Estimation  =============

\subsection{Nudging-Augmented Parameter Estimation} \label{ParamEstimation}

This subsection describes how the nudging framework is used to construct a parameter estimation method from partial observations. The main motivation is the following. If all state variables were observed, then parameter estimation could be carried out by directly comparing the full model trajectory with the full observed trajectory. This represents the ideal observation setting and provides the best possible information available to any trajectory-based estimation framework. In most applications, however, only a subset of the state variables is observed. The purpose of data assimilation is to use these partial observations, together with the model dynamics, to reconstruct the missing components of the state. Thus, if the nudged system synchronizes with the true system when the correct parameter is used, then the parameter estimation problem based on partial observations may be expected to approach the corresponding full-observation problem.

This observation provides the rationale for the proposed framework. For a trial parameter value $\tilde{\theta}$, the nudged system uses the available data to correct the observed components of the model trajectory. When $\tilde{\theta}=\theta^*$, the feedback term drives the nudged trajectory toward the true trajectory, even when the nudged system is initialized arbitrarily. In this case, the observed trajectory mismatch should decay after a transient time. On the other hand, when $\tilde{\theta}$ is far from $\theta^*$, the model dynamics are inconsistent with the observed trajectory, and the nudging term cannot fully remove the resulting discrepancy. Therefore, the time-delayed cost functional provides a natural objective for parameter recovery.

A closely related approach was developed by Biswas and Hudson~\cite{biswas2023determining} in the PDE setting for the two-dimensional incompressible Navier--Stokes equations. In that work, the authors recovered the viscosity from observations of finitely many Fourier modes by using the determining map associated with a nudging algorithm to define a cost functional. The minimization of this cost functional was used to solve the inverse problem of identifying the true viscosity. Their analysis established conditions for uniqueness and well-posedness of the inverse problem, showed that smallness of the cost functional implies proximity to the true viscosity, and provided an algorithm with convergence guarantees.

The present work extends this determining-map-based philosophy to a different setting. Instead of recovering a single viscosity parameter in an infinite-dimensional PDE from finitely many modal observations, we consider finite-dimensional ODE systems with possibly several unknown parameters and only partial state observations. Thus, the framework developed here can be viewed as an ODE and multi-parameter analogue of the nudging-based inverse-problem approach in~\cite{biswas2023determining}. The central question is whether data assimilation can reduce the gap between partial-observation parameter estimation and the ideal full-state observation problem.

Suppose that the observation data ${\cal O}(t)=H{\bf x}(t)$ are available on the time interval $[\tau,\tau+T]$, and that either the full parameter vector $\theta$ or a subset of its components is unknown. For each trial parameter value $\tilde{\theta}\in\Theta$, we solve the nudged system
\[
\dot{\tilde{{\bf x}}}(t)
=
F(\tilde{{\bf x}}(t),\tilde{\theta})
+
{\cal B}\bigl({\cal O}(t)-H\tilde{{\bf x}}(t)\bigr),
\qquad
\tilde{{\bf x}}(0)=\tilde{{\bf x}}_0 .
\]
Here, $\tilde{{\bf x}}_0$ can be chosen independently of the true initial condition. The estimated parameter is then obtained by minimizing the time-delayed cost functional, which is defined as a time-delayed mismatch between the observed data and the corresponding observed component of the nudged solution:
\begin{equation}
\label{eq:general_cost}
C(\tilde{\theta})
=
\frac{1}{T} \int_{\tau}^{\tau+T}
\left\lVert
{\cal O}(t)-H\tilde{{\bf x}}(t;\tilde{\theta})
\right\rVert_{\mathbb R^q}^2\,dt,
\qquad \tau\ge 0,\quad T>0.
\end{equation}
The use of the time-delay $\tau$ in the cost functional is important in practice because the nudged system may be initialized far from the true state, so the initial part of the trajectory may contain synchronization error caused by the arbitrary initial condition rather than by the parameter mismatch. Discarding this initial period makes the cost functional more closely reflect the compatibility between the trial parameter and the observed dynamics. Moreover, averaging the mismatch over a time interval reduces sensitivity to instantaneous oscillations in the synchronization error, as well as to error in the observed data, and provides a more robust cost functional than a mismatch evaluated at a single time.

This formulation differs significantly from the standard least-squares trajectory fitting. In standard trajectory fitting with partial observations, each trial parameter generates an unconstrained model trajectory from a prescribed initial condition, and the resulting mismatch may reflect both parameter error and error in the initial condition. In the present framework, the trajectory is continuously corrected by the data through the nudging term. Consequently, the cost functional measures whether the trial parameter can produce a dynamically consistent assimilated trajectory, rather than whether an uncorrected trajectory happens to remain close to the observations over a finite time interval. This is particularly useful in chaotic systems, where small errors in the initial condition can rapidly dominate the trajectory mismatch.

The same viewpoint also clarifies the relationship between the present framework and the continuous data assimilation parameter-learning method of Newey et al.~\cite{NEWEY2025114121}. Their work considers a related observable-error cost functional and interprets CDA-based parameter learning through finite-dimensional root-finding and optimization methods, including Newton, Gauss--Newton, and Levenberg--Marquardt-type updates. A key difference is that their update formulas are based on the observable error at a sufficiently large time and rely on assumptions about the long-time behavior and differentiability of this error. In contrast, the present work directly formulates and studies a time-delayed, time-averaged cost-functional minimization problem.

This distinction is not merely algorithmic. A single-time mismatch may fluctuate in time, especially away from the true parameter or in the presence of noise. The time-averaged cost functional in Eq.~\eqref{eq:general_cost} is designed to reduce the effect of such fluctuations. In addition, root-finding or local derivative-based updates may identify a local stationary point of the parameter landscape, while the present formulation treats parameter recovery as an inverse problem based on minimization of the nudging-induced cost functional. This allows us to connect small values of the cost functional to small parameter error under suitable nondegeneracy assumptions.

Another difference concerns the observation requirements. In the Lorenz--63 example considered in~\cite{NEWEY2025114121}, the authors note that their parameter update appears to require the number of estimated parameters to be no larger than the rank of the observation operator; in practice, estimating all three Lorenz parameters leads them to nudge the full state. In contrast, the framework studied here is designed specifically to use partial state observations. For Lorenz--63, we focus on the case where only one component, such as \(x(t)\), is observed, and we show theoretically and numerically how nudging can use this partial observation to recover the missing state information relevant for parameter estimation.

Algorithm~\ref{alg2} should therefore be interpreted not merely as a numerical recipe, but as a way of approximating the ideal full-observation parameter estimation problem using only partial observations: the nudging step reconstructs dynamically consistent state information from the available data, while the optimization step selects the parameter value for which this assimilated trajectory best agrees with the observations.

\begin{algorithm}[H]
\caption{Data-assimilation-augmented parameter estimation}
\label{alg2}

\KwIn{
Observed data ${\cal O}(t)=H{\bf x}(t)$ on $[\tau,\tau+T]$; 
initial condition $\tilde{{\bf x}}_0$ for the nudged system; 
initial parameter guess $\tilde{\theta}^{(0)}\in\Theta$;
feedback operator ${\cal B}$; 
stopping criterion $\eta$
}

\KwOut{Estimated parameter vector $\theta_e$}

Define the cost functional
\[
C(\tilde{\theta})
=
\frac{1}{T}\int_{\tau}^{\tau+T}
\left\lVert
{\cal O}(t)-H\tilde{{\bf x}}(t;\tilde{\theta})
\right\rVert_{\mathbb R^q}^2\,dt,
\]
where $\tilde{{\bf x}}(t;\tilde{\theta})$ solves the nudged system
\eqref{eq:general_nudged_system} with parameter value $\tilde{\theta}$\;

Initialize
\[
\tilde{\theta}\leftarrow \tilde{\theta}^{(0)}.
\]

\While{the stopping criterion $\eta$ is not satisfied}{
    Solve the nudged system \eqref{eq:general_nudged_system} on $[\tau,\tau+T]$
    using the current parameter value $\tilde{\theta}$\;
    
    Evaluate the cost functional
    \[
    C(\tilde{\theta})
    =
    \frac{1}{T}\int_{\tau}^{\tau+T}
    \left\lVert
    {\cal O}(t)-H\tilde{{\bf x}}(t;\tilde{\theta})
    \right\rVert_{\mathbb R^q}^2\,dt.
    \]
    
    Update $\tilde{\theta}$ using a numerical optimization algorithm\;
}

Set
\[
\theta_e \leftarrow \tilde{\theta}.
\]

\Return{$\theta_e$}
\end{algorithm}

%============= Section 4 =============

\section{The Lorenz'63 System}
\label{lorenzsystem}
In 1963, Edward Lorenz introduced a mathematical model to describe atmospheric convection \cite{lorenz1963}. This model, widely recognized as the Lorenz system, is characterized by three ordinary differential equations (ODEs):

\begin{equation}
\begin{cases}
    \dot{x}= \sigma (y - x), \\
    \dot{y}= x(\rho - z) - y, \\
    \dot{z}= xy - \beta z,
\end{cases}
\label{lorenz}
\end{equation}
where ${\mbf \theta}=(\sigma,\rho, \beta)$ is the parameter vector for the system.  These equations describe the dynamics of a two-dimensional fluid layer that is uniformly heated from below and cooled from above. They specifically capture the temporal rate of change of three critical variables: \(x\), which is proportional to the rate of convection; \(y\), which denotes the horizontal temperature variation; and \(z\), which reflects the vertical temperature variation \cite{sparrow1982}. The parameter \(\sigma\) represents the Prandtl number, while \(\rho\) corresponds to the Rayleigh number and $\beta$ represents certain physical dimensions of the layer itself. From a technical perspective, the Lorenz system is classified as a nonlinear, aperiodic, three-dimensional, deterministic ODE.

The dynamics of the Lorenz system is highly sensitive to the values of its parameters, particularly the Rayleigh number \( \rho \). While all three parameters (\( \sigma, \beta, \rho \)) influence the dynamics, we focus here on the role of the bifurcation parameter \( \rho \) in determining the structure and stability of equilibrium points. For \( \rho > 1 \), the system admits three equilibrium points: the origin \( (0, 0, 0) \), and two symmetric and nontrivial equilibrium points given by
\begin{equation}  \label{equil}
\left( \pm \sqrt{\beta(\rho - 1)}, \pm \sqrt{\beta(\rho - 1)}, \rho - 1 \right).
\end{equation}
In contrast, when \( \rho < 1 \), the origin is the unique equilibrium and is globally asymptotically stable; all trajectories converge to the origin as \( t \to \infty \). As \( \rho \) increases beyond unity, the origin loses stability, and the nontrivial equilibrium points emerge through a pitchfork bifurcation. These  remain locally asymptotically stable for values of \( \rho \) up to the critical threshold
\[
\rho_H := \sigma \cdot \frac{\sigma + \beta + 3}{\sigma - \beta - 1},
\]
provided \( \sigma > \beta + 1 \). For \( \rho > \rho_H \), nontrivial equilibrium points lose stability via a Hopf bifurcation, and the system transitions to a regime characterized by chaotic dynamics, marked by the emergence of the Lorenz attractor—a {\em strange} attractor with sensitive dependence on initial conditions \cite{hirsch2003}. The bifurcation structure and stability properties of these equilibria are central to understanding the rich dynamical behavior exhibited by the Lorenz system.
It is also well known that the Lorenz system is dissipative and possesses an absorbing ball in the phase space, which contains the attractor. In particular, one can show that \cite{Doering_Gibbon_1995}, 
\begin{equation}  \label{lbound}
    \limsup_{t\to \infty} \sqrt{x^2 + y^2 +z^2} \le R\ \text{where}\ R=\frac{3(\rho+\sigma)}{\sqrt{\min\{\sigma,\beta,1\}}}.
\end{equation}
Moreover, the ball $B(0;R)$ is absorbing and invariant for the Lorenz semigroup and contains the Lorenz attractor. Subsequently, we will assume that the parameter vector ${\mbf \theta=(\sigma,\rho,\beta)}$ belongs to a compact set. Due to \eqref{lbound}, we will henceforth assume that all trajectories $\{X_{\mbf \theta}(t): t \in {\cal I}\}$ satisfy the bound
\begin{equation}  \label{lunifbd}
    \sup_{t \in {\cal I}}\|X_{\mbf \theta}(t)\| \le 
    \fk,
\end{equation}
where $\fk > 0$ is a constant.

%============= Section 4.1 Nudging and Forward Prediction for Lorenz'63 =============

\subsection{Nudging and Forward Prediction for Lorenz'63}
We  first illustrate the synchronization properties that allow us to establish the well-posedness of the data-to-parameter inverse map based on partial state observation. It also motivates the parameter estimation framework outlined in Section~\ref{pe}. In the Lorenz system, observing  either the $x$ or the $y$ component is sufficient to synchronize the system and recover the full state, while observing only the $z$ component is insufficient for this purpose.

The \emph{determining functional} property enables recovery of the full system dynamics, and consequently the unknown parameters, using only partial observations. In what follows, we assume that the $x$-component is continuously observed. The nudged system corresponding to the Lorenz system \eqref{lorenz} is given by

\begin{equation}
    \begin{cases}
\dot{\tilde{x}} &= \tilde{\sigma} (\tilde{y} - \tilde{x}) + \mu (x- \tilde{x}), \\
\dot{\tilde{y}} &= \tilde{x} (\tilde{\rho} - \tilde{z}) - \tilde{y}, \\
\dot{\tilde{z}} &= \tilde{x}\tilde{y} - \tilde{\beta} \tilde{z},
    \end{cases}
    \label{nudgedsystem}
\end{equation}
where $\mu$ denotes the nudging coefficient and $x$ corresponds to the data obtained by numerically solving the Lorenz system in Eq.~\eqref{lorenz}. Next, we analyze the behavior of the nudged system when the parameter values used in the nudged model do not coincide with the true parameters. This situation naturally arises in parameter estimation: the true parameter vector is unknown, and the nudged system must be solved using an initial guess. Thus, the parameter mismatch represents the difference between the true parameter vector and the current trial parameter vector used in the nudged dynamics. Theorem~\ref{thm2} quantifies how the nudged solution remains exponentially stable and converges to a neighborhood of the true trajectory determined by the magnitude of the parameter mismatch.

The result in Theorem~\ref{thm2} does not rely on chaotic behavior of the Lorenz system. Rather, it uses only the uniform boundedness of the reference trajectory, as given in Eq.~\eqref{lunifbd}, and a sufficiently large nudging coefficient. Consequently, the estimate applies in both chaotic and non-chaotic parameter regimes, provided the stated hypotheses are satisfied.

\begin{theorem}[Exponential synchronization and the determining map]
\label{thm2}
Let \(X_{\mbf \theta}(t)=(x(t),y(t),z(t))\) be a solution of the Lorenz--63 system \eqref{lorenz} corresponding to the parameter vector \({\mbf \theta}=(\sigma,\rho,\beta)\). Assume that \(X_{\mbf \theta}\) satisfies the uniform bound \eqref{lunifbd}.
\begin{itemize}
\item[(i)] Let $(\tilde{x}(t),\tilde{y}(t),\tilde{z}(t))$ be a solution of the nudged system \eqref{nudgedsystem}
initialized from an arbitrary initial condition $(\tilde{x}(0),\tilde{y}(0),\tilde{z}(0)) \in \mathbb{R}^3$, with mismatched parameters $(\tilde{\sigma},\tilde{\rho},\tilde{\beta}) \neq (\sigma,\rho,\beta)$. Assume that the $x-$component of the reference solution is observed. Here $\tilde{\sigma},\tilde{\beta}>0$, $\tilde{\rho}\in\mathbb{R}$, and $\mu>0$.
Define the errors
\[
\bar{x}=x-\tilde{x},\qquad \bar{y}=y-\tilde{y},\qquad \bar{z}=z-\tilde{z},
\]
and the parameter mismatches
\[
\delta_\sigma:=\sigma-\tilde{\sigma},\qquad
\delta_\rho:=\rho-\tilde{\rho},\qquad
\delta_\beta:=\beta-\tilde{\beta}.
\]
If $\mu$ satisfies
\begin{equation}\label{eq:mu_condition_mismatch}
\mu>
\big(|\tilde{\sigma}+\tilde{\rho}|+\fk \big)^2+\frac{\fk^2}{\tilde{\beta}}+\frac14-\tilde{\sigma},
\end{equation}
then there exists $\gamma>0$ such that, for all $t\ge t_0$,
\begin{equation}\label{eq:mismatch_estimate}
\bar{x}^2(t)+\bar{y}^2(t)+\bar{z}^2(t)
\le
\big(\bar{x}^2(t_0)+\bar{y}^2(t_0)+\bar{z}^2(t_0)\big)e^{-2\gamma (t-t_0)}
+
\frac{C_\fk}{\gamma}\big(\delta_\sigma^2+\delta_\rho^2+\delta_\beta^2\big),
\end{equation}
$\quad \text{where } C_{\fk} = \fk^2 \max\left \{4,1,\dfrac{1}{\tilde{\beta}} \right \}$. In particular,
\[
\limsup_{t\to\infty}\big(\bar{x}^2(t)+\bar{y}^2(t)+\bar{z}^2(t)\big)
\le
\frac{C_{\fk}}{\gamma}\big(\delta_\sigma^2+\delta_\rho^2+\delta_\beta^2\big).
\]
So the nudged solution converges exponentially fast to a $O(|\delta|)$-neighborhood of the true solution.
\item[(ii)] If $\tilde{\sigma} = \sigma^*, \tilde{\rho}= \rho^*, \tilde{\beta}= \beta^*$ and the trajectory $\{{\bf x}(t), t \in \R\}$ is on the global attractor ${\cal A}$ of the Lorenz system, then $$\lim_{t\to\infty}\|\tilde{\bf x}(t)-{\bf x}(t)\|=0.$$
\item[(iii)] Assume that $\tilde{X}_i(t)=(\tilde{x}_i(t),\tilde{y}_i(t),\tilde{z}_i(t), i=1,2$, be two solutions of \eqref{nudgedsystem} with $x(t)$ replaced by $x_i(t), i=1,2$ respectively, where $x_i(\cdot) \in C_{bd}({\cal I}; \R)$. Here \(C_{bd}({\cal I};\mathbb{R})\) denotes the space of bounded continuous real-valued functions on the time interval \({\cal I}\), and ${\cal I}$ contains an interval of the form $[t_0, \infty)$. Then there exists a constant $\tilde{C}_\fk$ and $\mu$ sufficiently large depending on $\fk$ and $(\tilde{\sigma}, \tilde{\rho}, \tilde{\beta})$ such that 
\[
\|\tilde{X}_1(\cdot) - \tilde{X}_2(\cdot)\|_\infty \le e^{- \gamma (t-t_0)}\|\tilde{X}_1(t_0)-\tilde{X}_2(t_0)\|+ \tilde{C}_{\fk} \|x_1(\cdot)-x_2(\cdot)\|_\infty . 
\]
In particular, if either $\tilde{X}_1(t_0)=\tilde{X}_2(t_0)$, or the trajectories $\tilde{X}_i,i=1,2$ are on the attractor with ${\cal I}=\R$, then
the following estimate holds:
\begin{equation}
\|\tilde{X}_1(\cdot) - \tilde{X}_2(\cdot)\|_\infty \le \tilde{C}_\fk \|x_1(\cdot)-x_2(\cdot)\|_\infty . 
\end{equation}
\end{itemize}
\end{theorem}
\begin{proof}
We first prove (i). Without loss of generality, assume $t_0=0$.
Subtract the nudged system from the true system and set $\bar{x}=x-\tilde{x}$,
$\bar{y}=y-\tilde{y}$, $\bar{z}=z-\tilde{z}$. A direct computation gives:
\begin{equation}\label{eq:error_system_mismatch}
\begin{aligned}
\dot{\bar{x}}
&=\sigma(y-x)-\tilde{\sigma}(\tilde{y}-\tilde{x})-\mu(x-\tilde{x})
=-(\tilde{\sigma}+\mu)\bar{x}+\tilde{\sigma}\bar{y}+\delta_\sigma (y-x),\\
\dot{\bar{y}}
&=x(\rho-z)-y-\big(\tilde{x}(\tilde{\rho}-\tilde{z})-\tilde{y}\big)
=\tilde{\rho}\,\bar{x}-\bar{y}+(\tilde{x}\tilde{z}-xz)+\delta_\rho x,\\
\dot{\bar{z}}
&=xy-\beta z-\big(\tilde{x}\tilde{y}-\tilde{\beta}\tilde{z}\big)
=(xy-\tilde{x}\tilde{y})-\tilde{\beta}\bar{z}-\delta_\beta z.
\end{aligned}
\end{equation}
Multiply the three equations in Eq.~\eqref{eq:error_system_mismatch} by $\bar{x}$, $\bar{y}$, $\bar{z}$, respectively, and add. This yields:
\begin{equation}\label{eq:energy_identity_mismatch}
\frac12\frac{d}{dt}\big(\bar{x}^2+\bar{y}^2+\bar{z}^2\big)
+(\tilde{\sigma}+\mu)\bar{x}^2+\bar{y}^2+\tilde{\beta}\bar{z}^2
=
(\tilde{\sigma}+\tilde{\rho})\bar{x}\bar{y}
+(\tilde{x}\tilde{z}-xz)\bar{y}
+(xy-\tilde{x}\tilde{y})\bar{z}
+\delta_\sigma (y-x)\bar{x}
+\delta_\rho x\bar{y}
-\delta_\beta z\bar{z}.
\end{equation}
Expand the nonlinear terms:
\[
(\tilde{x}\tilde{z}-xz)\bar{y}
=\big((x-\bar{x})(z-\bar{z})-xz\big)\bar{y}
=\big(-x\bar{z}-z\bar{x}+\bar{x}\bar{z}\big)\bar{y},
\]
\[
(xy-\tilde{x}\tilde{y})\bar{z}
=\big(xy-(x-\bar{x})(y-\bar{y})\big)\bar{z}
=\big(x\bar{y}+y\bar{x}-\bar{x}\bar{y}\big)\bar{z}.
\]
Adding these two identities cancels the terms $x\bar{y}\bar{z}$ and $\bar{x}\bar{y}\bar{z}$, leaving
\begin{equation}\label{eq:cancellation_mismatch}
(\tilde{x}\tilde{z}-xz)\bar{y}+(xy-\tilde{x}\tilde{y})\bar{z}
=
y\bar{x}\bar{z}-z\bar{x}\bar{y}.
\end{equation}
Substituting Eq.~\eqref{eq:cancellation_mismatch} into Eq.~\eqref{eq:energy_identity_mismatch} gives
\begin{equation}\label{eq:energy_after_cancel_mismatch}
\frac12\frac{d}{dt}V
+(\tilde{\sigma}+\mu)\bar{x}^2+\bar{y}^2+\tilde{\beta}\bar{z}^2
=
(\tilde{\sigma}+\tilde{\rho})\bar{x}\bar{y}
+y\bar{x}\bar{z}
-z\bar{x}\bar{y}
+\delta_\sigma (y-x)\bar{x}
+\delta_\rho x\bar{y}
-\delta_\beta z\bar{z},
\end{equation}
where $V(t):=\bar{x}^2+\bar{y}^2+\bar{z}^2$. By the uniform boundedness of Lorenz (see \cite{lorenz1963}), there exists $\fk>0$ such that $|x(t)|,|y(t)|,|z(t)|\le \fk$ for all $t\ge 0$, hence
\[
(\tilde{\sigma}+\tilde{\rho})\bar{x}\bar{y}-z\bar{x}\bar{y}
\le \big(|\tilde{\sigma}+\tilde{\rho}|+\fk\big)|\bar{x}\bar{y}|,
\qquad
y\bar{x}\bar{z}\le \fk|\bar{x}\bar{z}|,
\]
and also $|y-x|\le |y|+|x|\le 2\fk$. Therefore, from Eq.~\eqref{eq:energy_after_cancel_mismatch},
\begin{equation}\label{eq:pre_young_mismatch}
\frac12\dot V
+(\tilde{\sigma}+\mu)\bar{x}^2+\bar{y}^2+\tilde{\beta}\bar{z}^2
\le
\big(|\tilde{\sigma}+\tilde{\rho}|+\fk\big)|\bar{x}\bar{y}|
+\fk|\bar{x}\bar{z}|
+2\fk|\delta_\sigma||\bar{x}|
+\fk|\delta_\rho||\bar{y}|
+\fk|\delta_\beta||\bar{z}|.
\end{equation}
By Young's inequality, we have:
\[
\big(|\tilde{\sigma}+\tilde{\rho}|+\fk\big)|\bar{x}\bar{y}|
\le
\big(|\tilde{\sigma}+\tilde{\rho}|+\fk\big)^2\bar{x}^2+\frac14\bar{y}^2,
\]
\[
\fk|\bar{x}\bar{z}|
\le
\frac{\fk^2}{\tilde{\beta}}\bar{x}^2+\frac{\tilde{\beta}}{4}\bar{z}^2,
\]
\[
2\fk|\delta_\sigma||\bar{x}|
\le 4\fk^2\delta_\sigma^2+\frac14\bar{x}^2,
\qquad
\fk|\delta_\rho||\bar{y}|
\le \fk^2\delta_\rho^2+\frac14\bar{y}^2,
\qquad
\fk|\delta_\beta||\bar{z}|
\le \frac{\fk^2}{\tilde{\beta}}\delta_\beta^2+\frac{\tilde{\beta}}{4}\bar{z}^2.
\]
Substituting these bounds into Eq.~\eqref{eq:pre_young_mismatch} yields
\begin{align}
\frac12\dot V
&+
\Big(\tilde{\sigma}+\mu-\big(|\tilde{\sigma}+\tilde{\rho}|+\fk\big)^2-\frac{\fk^2}{\tilde{\beta}}-\frac14\Big)\bar{x}^2
+\frac12\bar{y}^2
+\frac{\tilde{\beta}}{2}\bar{z}^2
\notag\\
&\le
4\fk^2\delta_\sigma^2+\fk^2\delta_\rho^2+\frac{\fk^2}{\tilde{\beta}}\delta_\beta^2
\le C_{\fk} (\delta_\sigma^2+\delta_\rho^2+\delta_\beta^2). 
\label{eq:after_young_mismatch}
\end{align}
Assumption in Eq.~\eqref{eq:mu_condition_mismatch} ensures that the coefficient of $\bar{x}^2$ in Eq.~\eqref{eq:after_young_mismatch} is positive. Define
\[
\gamma:=
\min\left\{
\tilde{\sigma}+\mu-\big(|\tilde{\sigma}+\tilde{\rho}|+\fk\big)^2-\frac{\fk^2}{\tilde{\beta}}-\frac14,\;
\frac12,\;
\frac{\tilde{\beta}}{2}
\right\}>0.
\]
Since $V=\bar{x}^2+\bar{y}^2+\bar{z}^2$, inequality \eqref{eq:after_young_mismatch} implies
\[
\frac12\dot V+\gamma V\le C_\fk (\delta_\sigma^2+\delta_\rho^2+\delta_\beta^2).
\]
By Gr\"onwall's inequality,
\[
V(t)\le V(0)e^{-2\gamma t}+\frac{C_\fk}{\gamma}\big(\delta_\sigma^2+\delta_\rho^2+\delta_\beta^2\big),
\]
which is exactly Eq.~\eqref{eq:mismatch_estimate}. The $\limsup$ bound follows by letting $t\to\infty$.

Proof of (ii) follows from \eqref{eq:mismatch_estimate} with $\delta_\sigma = \delta_\rho = \delta_\beta = 0$, recalling that all trajectories on the attractor satisfy the uniform bound \eqref{lunifbd} and letting $t \to \infty$.
The proof of (iii) is similar to the proof of (i) and is omitted.
\end{proof}

\begin{rem}
    Similar convergence results can be obtained when only the $y$–component is observed.
\end{rem}

%============= Section 4.2 Well-posedness of the Data-to-Parameter Map for Lorenz'63 =============

\subsection{Well-posedness of the Data-to-Parameter Map for Lorenz '63}
We  start by providing a comprehensive analysis of this issue when {\em all} state variables are observed. Clearly, if the parameter determination problem is ill-posed when the full state is observed, it will be so for partial state observations as well. In Theorem~\ref{thm:wellposedness}, we provide a complete analysis of the well-posedness of the inverse map in this case.

 \begin{theorem}[Well-posedness when the full state is observed]  \label{thm:wellposedness}
 Let ${\bf x}(t)=(x(t),y(t),z(t))$ be a $C^1$ trajectory of the Lorenz system~\eqref{lorenz},
where $\sigma,\rho,\beta \in \mathbb{R}$ are unknown parameters.
\begin{enumerate}
\item[(i)] 
Fix a time $t_0$, and assume that we observe the values
$(x(t_0),y(t_0),z(t_0),\dot{x}(t_0),\dot{y}(t_0),\dot{z}(t_0))$. If $y(t_0) \neq x(t_0)$, $x(t_0) \neq 0$, and $z(t_0) \neq 0$, then $(\sigma,\rho,\beta)$ is uniquely determined by
\[
\sigma = \frac{\dot{x}}{y-x}, 
\qquad
\rho = \frac{\dot{y} + xz + y}{x},
\qquad
\beta = \frac{xy - \dot{z}}{z}.
\]
\item[(ii)] Let ${\cal S}=\{(x,y,z): y=x \ \text{or}\ x=0 \ \text{or}\ z=0\}$ and for $\delta>0$, let ${\cal S}_\delta=\{{\bf x} \in \mathbb{R}^3: d({\bf x}, {\cal S}) \ge \delta \}.$ Let ${\bf \theta}_i=(\sigma_i,\rho_i,\beta_i), i=1,2$ be two parameter vectors for the Lorenz system such that the corresponding trajectories ${\bf x}_i$ belong to ${\cal S}_\delta$ for all $t \in \cal{I}$. Then, we have the (well-posedness) estimate
\begin{equation} \label{west}
\|{\bf \theta}_1 - {\bf \theta}_2\|_{\R^3}
\le C_{\delta,M} \|{\bf x}_1 - {\bf x}_2\|_{C_{bd}^1({\cal I};\R^{3})},
\end{equation}
where $M=\max_{i=1,2}\|{\bf x}_i\|_{C_{bd}^1({\cal I};\mathbb R^3)},$ $C_{bd}^1({\cal I};\mathbb R^3)
=
\left\{
{\bf x}\in C^1({\cal I};\mathbb R^3):
\|{\bf x}\|_{C_{bd}^1({\cal I};\mathbb R^3)}<\infty
\right\},$ and $\|{\bf x}\|_{C_{bd}^1({\cal I};\mathbb R^3)}
=
\sup_{t\in{\cal I}}\|{\bf x}(t)\|_{\mathbb R^3}
+
\sup_{t\in{\cal I}}\|\dot{\bf x}(t)\|_{\mathbb R^3}.$
\item[(iii)] 
If the trajectory is the equilibrium solution
\[
(x(t), y(t), z(t)) = (0, 0, 0), \qquad \text{for all } t,
\] then none of the parameters $\sigma, \rho, \beta$ is identifiable.
On the other hand, if the trajectory is one of the two non-trivial equilibrium points given in $\eqref{equil}$, then $\beta$ and $\rho$ are uniquely identifiable while $\sigma$ is not identifiable.
\item[(iv)] For the trajectory $x(t)=y(t)=0, z(t)=e^{-\beta t}z(0)$, where $z(0) \neq 0$, only the parameter $\beta$ is identifiable.
\item[(v)] Assume that $\sigma\neq 0$ and $\rho\neq 0$. If ${\bf x}(\cdot)$ is any trajectory of the Lorenz system, then unless it is an equilibrium point, or the trajectory given in (iv), there exists an interval $\cal{I}$ and a constant $\delta>0$ such that $\{{\bf x}(t):t \in \cal{I}\} \subset \cal{S}_\delta$ in which case the well-posedness estimate 
\eqref{west} holds.
\end{enumerate}
 \end{theorem}
 \comments{
\begin{lemma} \label{lemmatime}
Let $(x(t),y(t),z(t))$ be a $C^1$ trajectory of the Lorenz system
\[
\begin{cases}
\dot{x} = \sigma (y - x), \\[2pt]
\dot{y} = x \rho - x z - y, \\[2pt]
\dot{z} = x y - \beta z,
\end{cases}
\]
where $\sigma,\rho,\beta \in \mathbb{R}$ are unknown parameters. Fix a time $t_0$, and assume that we observe the values
$(x(t_0),y(t_0),z(t_0),\dot{x}(t_0),\dot{y}(t_0),\dot{z}(t_0))$.
\begin{enumerate}
\item[(i)] 
If $y(t_0) \neq x(t_0)$, $x(t_0) \neq 0$, and $z(t_0) \neq 0$, then $(\sigma,\rho,\beta)$ is uniquely determined by
\[
\sigma = \frac{\dot{x}}{y-x}, 
\qquad
\rho = \frac{\dot{y} + xz + y}{x},
\qquad
\beta = \frac{xy - \dot{z}}{z}.
\]

\item[(ii)] 
If $x=0$, $y=0$, and $z\neq 0$, then $\sigma$ and $\rho$ are not identifiable from this pointwise observation, while $\beta$ is uniquely determined by
\[
\beta = -\frac{\dot{z}}{z}.
\]

\item[(iii)] 
If the trajectory is at an equilibrium, then $\dot{x}=\dot{y}=\dot{z}=0$ and $y=x$. In particular, $\sigma$ is not identifiable. Moreover, if $x\neq 0$ and $z\neq 0$, the system imposes only
\[
\rho = z+1,
\qquad
\beta = \frac{x^2}{z},
\]
so the parameters are not uniquely identifiable from a single pointwise observation.
\end{enumerate}
\end{lemma}
}
\begin{proof}

\begin{enumerate}
    \item[(i)] Fix $t_0$ and denote
\[
x=x(t_0),\quad y=y(t_0),\quad z=z(t_0),\quad
\dot{x}=\dot{x}(t_0),\quad
\dot{y}=\dot{y}(t_0),\quad
\dot{z}=\dot{z}(t_0).
\]

The Lorenz equations may be written as a linear algebraic system for the parameters:
\[
\begin{bmatrix}
y-x & 0 & 0\\
0 & x & 0\\
0 & 0 & -z
\end{bmatrix}
\begin{bmatrix}
\sigma\\
\rho\\
\beta
\end{bmatrix}
=
\begin{bmatrix}
\dot{x}\\
\dot{y}+xz+y\\
\dot{z}-xy
\end{bmatrix}.
\]
If $y\neq x$, $x\neq 0$, and $z\neq 0$, then the coefficient matrix is
invertible. Hence
\[
\sigma = \frac{\dot{x}}{y-x},
\qquad
\rho = \frac{\dot{y}+xz+y}{x},
\qquad
\beta = \frac{xy-\dot{z}}{z}.
\]

\item[(ii)] Note first that if one observes the trajectory in a time interval, then one can determine its derivative as well.
In view of part (i), define the recovery map
\[
\Phi(x,y,z,\dot{x},\dot{y},\dot{z})
=
\left(
\frac{\dot{x}}{y-x},
\frac{\dot{y}+xz+y}{x},
\frac{xy-\dot{z}}{z}
\right).
\]
By part {\rm (i)}, for each trajectory ${\bf x}_i=(x_i,y_i,z_i)$,
$i=1,2$, the corresponding parameter vector satisfies
\[
{\bf \theta}_i
=
\Phi(x_i,y_i,z_i,\dot{x}_i,\dot{y}_i,\dot{z}_i).
\]
Since ${\bf x}_i(t)\in{\cal S}_\delta$ for all $t\in{\cal I}$, the denominators
$y_i(t)-x_i(t)$, $x_i(t)$, and $z_i(t)$ are uniformly bounded away from zero. Moreover, the quantities $(x_i,y_i,z_i,\dot{x}_i,\dot{y}_i,\dot{z}_i)$
are bounded in terms of
\[
M=\max_{i=1,2}\|{\bf x}_i\|_{C_{bd}^1({\cal I};\mathbb R^3)}.
\]
Hence $\Phi$ is smooth, and therefore Lipschitz, on the corresponding bounded
subset of $\mathbb R^6\setminus \{y=x \ \text{or}\ x=0 \ \text{or}\ z=0\}.$
Consequently, there exists a constant $C>0$ such that
\[
\begin{aligned}
\|{\bf \theta}_1-{\bf \theta}_2\|_{\mathbb R^3}
&\le
\frac{CM}{\delta}
\left\|
(x_1,y_1,z_1,\dot{x}_1,\dot{y}_1,\dot{z}_1)
-
(x_2,y_2,z_2,\dot{x}_2,\dot{y}_2,\dot{z}_2)
\right\|_{L^\infty({\cal I})} \\
&\le
\frac{CM}{\delta}
\|{\bf x}_1-{\bf x}_2\|_{C_{bd}^1({\cal I};\mathbb R^3)}.
\end{aligned}
\]
This proves the well-posedness estimate.

\item[(iii)] At the equilibrium point $(0,0,0)$, we have $\dot{x}=\dot{y}=\dot{z}=0.$ Substitution into the Lorenz equations gives only the identities
\[
0=0,\qquad 0=0,\qquad 0=0.
\]
Thus no condition is imposed on $\sigma$, $\rho$, or $\beta$, and none of the
parameters is identifiable from this equilibrium trajectory. 

Now consider a non-trivial equilibrium point given in $\eqref{equil}$. At such a point,
\[
\dot{x}=\dot{y}=\dot{z}=0
\qquad \text{and} \qquad y=x.
\]
The first Lorenz equation becomes $0=\sigma(y-x)=0,$ and therefore imposes no condition on $\sigma$. The remaining two equations give
\[
0=x(\rho-z)-x,
\qquad
0=x^2-\beta z.
\]
Since the non-trivial equilibrium satisfies $x\neq 0$ and $z\neq 0$, these equations imply
\[
\rho=z+1,
\qquad
\beta=\frac{x^2}{z}.
\]
Thus $\rho$ and $\beta$ are uniquely determined by the nonzero equilibrium state, while $\sigma$ remains unidentifiable.

\item[(iv)] Suppose $x(t)=y(t)=0, z(t)=e^{-\beta t}z(0),$ with $z(0) \neq 0$. Then the first two Lorenz equations reduce to
\[
0=0,
\qquad
0=0,
\]
and hence impose no conditions on $\sigma$ or $\rho$. The third equation gives $\dot{z}=-\beta z.$ Since $z(t)\neq 0$ on the interval under consideration, we obtain
\[
\beta=-\frac{\dot{z}(t)}{z(t)}.
\]
Thus $\beta$ is identifiable, while $\sigma$ and $\rho$ are not identifiable.

\item[(v)] Assume that $\sigma\neq 0$ and $\rho\neq 0$. Let ${\bf x}(t)=(x(t),y(t),z(t))$ be a trajectory of the Lorenz system that is neither an equilibrium trajectory nor the trajectory described in {\rm (iv)}. We claim that there exists a time $t_*$ such that
\[
x(t_*)\neq 0,\qquad y(t_*)\neq x(t_*),\qquad z(t_*)\neq 0.
\]
Suppose, for contradiction, that no such time exists. Then
\[
x(t)(y(t)-x(t))z(t)=0
\]
for all $t$ in the time interval under consideration. Since the Lorenz vector field is polynomial, its solutions are real analytic in time. Hence the function \[
t\mapsto x(t)(y(t)-x(t))z(t)
\]
is real analytic and vanishes identically. By \cite[Corollary~1.2.6]{KrantzParks},  a real analytic function whose zero set has an accumulation point in an interval must vanish identically on that interval. Therefore, if none of the factors $x(t)$, $y(t)-x(t)$, and $z(t)$ vanished identically on any nontrivial subinterval, each factor would have only isolated zeros. The union of three isolated zero sets cannot contain an interval. Hence, on some nontrivial subinterval, one of the following alternatives holds:
\[
x(t)\equiv 0,\qquad y(t)\equiv x(t),\qquad \text{or}\qquad z(t)\equiv 0.
\]
We examine these cases. First, suppose $x(t)\equiv 0$ on a nontrivial interval. Then $\dot{x}(t)\equiv 0$, and the first Lorenz equation gives
\[
0=\sigma(y(t)-x(t))=\sigma y(t).
\]
Since $\sigma\neq 0$, it follows that $y(t)\equiv 0$ on this interval. The third equation then reduces to
\[
\dot{z}(t)=-\beta z(t).
\]
Thus
\[
z(t)=e^{-\beta(t-t_0)}z(t_0),
\]
and the trajectory is precisely of the type described in {\rm (iv)}, which is excluded.

Second, suppose $y(t)\equiv x(t)$ on a nontrivial interval.  Then
\[
\dot{x}(t)=\sigma(y(t)-x(t))=0,
\]
so $x(t)\equiv c$ and $y(t)\equiv c$ on that interval. The second equation becomes
\[
0=c(\rho-z(t))-c=c(\rho-z(t)-1).
\]
If $c=0$, then $x(t)=y(t)=0$, and the trajectory reduces to the case described in {\rm (iv)}. If $c\neq 0$, then
\[
z(t)\equiv \rho-1.
\]
The third equation then gives
\[
0=c^2-\beta z.
\] Thus $z$ is constant, and hence $x$, $y$, and $z$ are all constant on this interval. Therefore, this is an equilibrium trajectory and thus excluded  by assumption. 

Third, suppose $z(t)\equiv 0$ on a nontrivial interval. Then $\dot{z}(t)\equiv 0$, and the third Lorenz equation gives
\[
0=x(t)y(t).
\]
Since $x$ and $y$ are continuous, either $x(t)\equiv 0$ or $y(t)\equiv 0$ on a nontrivial sub-interval. If $x(t)\equiv 0$, then the first case applies. If $y(t)\equiv 0$, then the second Lorenz equation gives
\[
0=x(t)(\rho-z(t))-y(t)=\rho x(t),
\]
because $z(t)\equiv 0$ and $y(t)\equiv 0$. Since $\rho\neq 0$, it follows that $x(t)\equiv 0$, and we again reduce to the first case.

In each case, the trajectory is an equilibrium trajectory or the trajectory described in {\rm (iv)}, both of which are excluded. Therefore there must exist a time $t_*$ such that
\[
x(t_*)\neq 0,\qquad y(t_*)\neq x(t_*),\qquad z(t_*)\neq 0.
\]
By continuity, these inequalities persist on a sufficiently small interval ${\cal I}$ containing $t_*$. Hence there exists $\delta>0$ such that
\[
\{{\bf x}(t):t\in{\cal I}\}\subset {\cal S}_\delta.
\]
The well-posedness estimate \eqref{west} then follows from part {\rm (ii)}.
\end{enumerate}
\end{proof}

We now pass from full-state information to partial-state information. The previous result shows that, away from the degeneracy set \({\cal S}\), the full Lorenz trajectory determines the parameter vector in a well-conditioned way. However, the parameter-estimation framework developed in this paper uses only partial observations, and in the numerical experiments we primarily use the \(x\)-component. Theorem~\ref{thm:partialwellposedness} shows that this partial observation is sufficient: if two admissible Lorenz trajectories remain away from the degeneracy set, then the distance between their parameter vectors can be controlled directly by the difference between their observed \(x\)-components. The proof follows by combining Theorem~\ref{thm2} and Theorem~\ref{thm:wellposedness}.

\begin{theorem}[Well-posedness under partial observation of the state] \label{thm:partialwellposedness}
Let $${\cal S}=\{(x,y,z): y=x \ \text{or}\ x=0 \ \text{or}\ z=0\}$$ and for $\delta>0$, let $${\cal S}_\delta=\{{\bf x} \in \mathbb{R}^3: d({\bf x}, {\cal S}) \ge \delta \}.$$ Let \(\Theta_{\mathrm{ad}}\subset\R^3\) be a compact set of admissible parameter values, and let $${\mbf \theta}_i=(\sigma_i,\rho_i,\beta_i)\in\Theta_{\mathrm{ad}}, \quad i=1,2$$ be two parameter vectors for the Lorenz system. Assume that the corresponding trajectories $$X_i(\cdot)=(x_i(\cdot),y_i(\cdot),z_i(\cdot)), \quad i=1,2,$$ belong to ${\cal S}_\delta$ for all $t \in {\cal I}$, where ${\cal I} \subset \R$ is any nontrivial interval if $X_i(\cdot) \in {\cal A}, i=1,2$, or else contains the interval  $[t_0,\infty)$ for some $t_0 >0$. Then,  there exists a constant \(C_{\delta,\Theta_{\mathrm{ad}}}>0\) such that we have the (well-posedness) estimate
\begin{equation} \label{wobsest}
\|{\mbf \theta}_1 - {\mbf \theta}_2\|_{\R^3}
\le C_{\delta,\Theta_{\mathrm{ad}}} \| x_1 -  x_2\|_{L^\infty({\cal I})},
\end{equation}
where the trajectories satisfy \eqref{lunifbd}.
\end{theorem}

\begin{proof}
 Let $e_x=\norm{x_1-x_2}_{L^\infty({\cal I})}.$ Since \(x\) is a determining functional for the Lorenz system on the class of
trajectories satisfying \eqref{lunifbd}, there exists a constant
\(C_D>0\), depending only on the absorbing bound and on the admissible parameter
set \(\Theta_{\mathrm{ad}}\), such that
\begin{equation*}
\label{detxestimate}
\norm{X_1-X_2}_{L^\infty({\cal I})}
\le
C_D \norm{x_1-x_2}_{L^\infty({\cal I})}
=
C_D e_x.
\end{equation*}
In particular,
\begin{equation}
\label{ydiffzdiff}
\norm{y_1-y_2}_{L^\infty({\cal I})}
+
\norm{z_1-z_2}_{L^\infty({\cal I})}
\le
C_D e_x .
\end{equation}
Choose a compact subinterval \([a,b]\subset{\cal I}\) and set
\(L=b-a>0\). Since \(X_i(t)\in{\cal S}_\delta\), we have
\[
|y_i(t)-x_i(t)|\ge \delta,\qquad |x_i(t)|\ge\delta,\qquad |z_i(t)|\ge\delta,
\qquad t\in[a,b].
\]
By continuity, each of \(y_i-x_i\), \(x_i\), and \(z_i\) has a fixed sign on
\([a,b]\). Hence
\begin{equation}
\label{denombounds}
\left|\int_a^b (y_i-x_i)\,dt\right|\ge \delta L,\qquad
\left|\int_a^b x_i\,dt\right|\ge \delta L,\qquad
\left|\int_a^b z_i\,dt\right|\ge \delta L .
\end{equation}
We estimate the three parameters separately. From the first Lorenz equation, $\dot{x}_i=\sigma_i(y_i-x_i).$ Integrating over \([a,b]\) gives
\[
x_i(b)-x_i(a)
=
\sigma_i\int_a^b (y_i-x_i)\,dt.
\]
Therefore,
\[
\sigma_i
=
\frac{x_i(b)-x_i(a)}
{\int_a^b (y_i-x_i)\,dt}.
\]
Using \eqref{denombounds}, \eqref{ydiffzdiff}, and the uniform bound
\eqref{lunifbd}, we obtain
\[
|\sigma_1-\sigma_2|
\le
C_{\delta,\Theta_{\mathrm{ad}}} e_x .
\]

Next, from the second Lorenz equation,
\[
\dot{y}_i=\rho_i x_i-y_i-x_i z_i.
\]
Integrating over \([a,b]\) gives
\[
y_i(b)-y_i(a)
=
\rho_i\int_a^b x_i\,dt
-
\int_a^b y_i\,dt
-
\int_a^b x_i z_i\,dt.
\]
Hence
\[
\rho_i
=
\frac{
y_i(b)-y_i(a)
+
\int_a^b y_i\,dt
+
\int_a^b x_i z_i\,dt
}
{\int_a^b x_i\,dt}.
\]
Again using \eqref{denombounds}, the uniform bound \eqref{lunifbd}, and
\eqref{ydiffzdiff}, we find
\[
|\rho_1-\rho_2|
\le
C_{\delta,\Theta_{\mathrm{ad}}} e_x .
\]

Finally, from the third Lorenz equation,
\[
\dot{z}_i=x_i y_i-\beta_i z_i.
\]
Integrating over \([a,b]\) gives
\[
z_i(b)-z_i(a)
=
\int_a^b x_i y_i\,dt
-
\beta_i\int_a^b z_i\,dt.
\]
Therefore,
\[
\beta_i
=
\frac{
\int_a^b x_i y_i\,dt
-
\bigl(z_i(b)-z_i(a)\bigr)
}
{\int_a^b z_i\,dt}.
\]
Using \eqref{denombounds}, \eqref{lunifbd}, and \eqref{ydiffzdiff}, we obtain
\[
|\beta_1-\beta_2|
\le
C_{\delta,\Theta_{\mathrm{ad}}} e_x .
\]

Combining the three estimates yields
\[
\norm{{\mbf \theta}_1-{\mbf \theta}_2}_{\R^3}
\le
C_{\delta,\Theta_{\mathrm{ad}}}
\norm{x_1-x_2}_{L^\infty({\cal I})}.
\]
This proves \eqref{wobsest}.

\end{proof}

%============= Section 4.3 Nudging-Augmented Optimization for Lorenz ’63 Parameter Estimation =============

\subsection{Nudging-Augmented Optimization for Lorenz ’63 Parameter Estimation}
\label{pe1}
Theorem~\ref{thm:wellposedness} shows that, in principle, the parameters of the Lorenz system can be recovered from pointwise measurements of the state and its time derivatives. This corresponds to the classical setting of parameter estimation, where both the full state and derivative information are assumed to be available. However, this setting is rarely available in practice. First, frequently we only have partial information of the state, either,  for instance, some components of the state variables or a functional of them. Secondly, even when full state is observed, 
derivative information is typically not directly accessible from observed data and must be approximated numerically, introducing additional error. This might introduce large errors in derivative estimates for a chaotic system. We therefore employ the nudging-augmented optimization framework outlined in Section~\ref{ParamEstimation}. For the Lorenz--63 experiments, Algorithm~\ref{alg2} is applied with $\theta=(\sigma,\rho,\beta),$ and with observation of the $x$-component. Thus,
\[
{\cal O}(t)=x(t),
\qquad
H(x,y,z)=x.
\]
The general cost functional \eqref{eq:general_cost} therefore reduces to

\begin{equation} \label{cost}
    C(\tilde{\mbf \theta})=C(\tilde{\sigma},\tilde{\rho},\tilde{\beta})
=
\frac 1T \int_\tau^{\tau+T}
|\tilde{x}(t)-x(t)|^2\,dt, \tau \ge 0, T>0
\ \text{and}\ \tilde{\mbf \theta}=(\tilde{\sigma},\tilde{\rho},\tilde{\beta}).
\end{equation}
Let ${\mbf \theta}^*=(\sigma^*,\rho^*,\beta^*)$ and $X_{{\mbf \theta}^*} 
=(x^*,y^*,z^*)$ be the true trajectory of Lorenz on the attractor corresponding to the parameter ${\mbf \theta}^*$. Our parameter estimation strategy is to minimize this cost functional with respect to $\tilde{\mbf \theta}$. However, our numerical experiments suggest that this cost functional is highly nonconvex and therefore it is conceivable that it might have multiple minima. Due to part (ii) of Theorem \ref{thm2}, we must have $C({\mbf \theta}^*)=C(\sigma^*,\rho^*,\beta^*)=0$. 
In other words, the minimum value of the cost functional is zero. The question then arises as to whether this minimum is unique. In other words, if $C(\tilde{\mbf \theta})=0$, must it imply $\tilde{\mbf \theta}={\mbf \theta}^*$. We prove a partial uniqueness result of the minimizer in Theorem~\ref{Th2}. 
\begin{theorem}[Uniqueness of parameter recovery with \(\sigma^*\) fixed] \label{Th2}
 Let $X(t)=(x(t),y(t),z(t)), t \in \R$ be a non-equilibrium solution of the Lorenz system with parameters ${\mbf \theta}^*=(\sigma^*,\rho^*,\beta^*)$ and assume that the trajectory lies on the global attractor.
 \comments{
\[
\begin{cases}
\dot{x}=\sigma^*(y-x),\\[2pt]
\dot{y}=x(\rho^*-z)-y,\\[2pt]
\dot{z}=xy-\beta^* z,
\end{cases}
\]
 Consider the nudged system, driven by the observable $x(t)$,
\[
\begin{cases}
\dot{\tilde{x}}=\sigma^*(\tilde{y}-\tilde{x})+\mu\,(x(t)-\tilde{x}),\\[2pt]
\dot{\tilde{y}}=\tilde{x}(\rho-\tilde{z})-\tilde{y},\\[2pt]
\dot{\tilde{z}}=\tilde{x}\tilde{y}-\beta \tilde{z},
\end{cases}
\]
with $\mu>0$ and $\rho, \beta \in\mathbb{R}$. For $T>0$, define
\[
C(\rho, \beta)=\int_0^T |\tilde{x}(t)-x(t)|^2\,dt,
\]
and Assume that $x(t)$ is not identically zero on $[0,T]$.} Then, for sufficiently large $\mu$,
\[
C(\sigma^*, \rho, \beta)=0 \quad\Longleftrightarrow\quad (\rho, \beta)=(\rho^*, \beta^*).
\]
\end{theorem}

\begin{proof}
Assume $C(\sigma^*, \rho, \beta)=0$. Then for all $t\in[0,T]$, $\tilde{x}(t)=x(t)$. On $[0,T]$, substituting $\tilde{x}=x$ into the $\tilde{x}$–equation gives $\dot{x}=\dot{\tilde{x}}=\sigma^*(\tilde{y}-x)$, while the true $x$–equation yields $\dot{x}=\sigma^*(y-x)$. Since $\sigma^* >0$, this immediately yields
 $\tilde{y}(t)=y(t)$ on $[0,T]$. Comparing the $\tilde{y}$ and $y$ equations, we have
\[
\dot{\tilde{y}}=x(\rho-\tilde{z})-y,\qquad \dot{y}=x(\rho^*-z)-y,
\]
one obtains
\[
0=\dot{\tilde{y}}-\dot{y}=x(t)\bigl(\rho-\tilde{z}(t)-\rho^*+z(t)\bigr)\quad\text{for }t\in[0,T].
\]
Because $x$ is continuous and not identically zero on $[0,T]$, we get
\begin{equation}
    \tilde{z}(t)-z(t)=\rho-\rho^* =\text{constant},\quad\text{for all }t\in[0,T]. \label{9}
\end{equation}
Subtracting the $\tilde{z}$ and $z$ equations gives
\[
\dot{\tilde{z}}-\dot{z}=-\beta\tilde{z}+\beta^* z.
\]
Since $\tilde{z}-z$ is constant on $[0,T]$, its derivative vanishes and hence 
\begin{equation}
    -\beta\tilde{z}+\beta^* z=0, \quad\text{for all }t\in[0,T].\label{10}
\end{equation}
 Then, Eq.~\eqref{9} and \eqref{10} gives: 
 \begin{equation} \label{eq24}
     \beta (\rho - \rho^*) + z (\beta - \beta^*) = 0, \quad\text{for all }t\in[0,T].
 \end{equation}
If $\beta\neq\beta^*$, then Eq.~\eqref{eq24} gives
\[
z(t)
=
-\frac{\beta(\rho-\rho^*)}{\beta-\beta^*},
\qquad t\in[0,T],
\]
so $z$ is constant on $[0,T]$.  By Lemma~\ref{lemma1}, the trajectory $(x(\cdot), y(\cdot),z(\cdot))$ then must be an equilibrium point, contradicting the hypothesis of the theorem. Thus $\beta=\beta^*$ and consequently from \eqref{eq24}, $\rho=\rho^*$ as well. 
\comments{
If $\rho\neq\rho^*$ and $\beta=\beta^*$, then Eq.~\eqref{eq24} reduces to
\[
\beta^*(\rho-\rho^*)=0.
\]
Since $\beta^*>0$, this is impossible. Finally, if $\rho=\rho^*$ and $\beta\neq\beta^*$, then Eq.~\eqref{eq24} reduces to
\[
z(t)(\beta-\beta^*)=0,
\qquad t\in[0,T].
\]
Thus $z(t)=0$ for all $t\in[0,T]$, so $z$ is again constant on $[0,T]$,
contradicting Lemma~\ref{lemma1}. Therefore $(\rho, \beta)=(\rho^*, \beta^*).$  
}

Conversely, assume $(\rho, \beta)=(\rho^*, \beta^*).$ Since the true trajectory $(x,y,z)$ lies on the global attractor, then, for $\mu$ chosen as in the Theorem \ref{thm2}, the nudged system converges to the true system and all variables synchronize: $(\tilde{x}, \tilde{y}, \tilde{z}) \equiv (x, y, z)$. Consequently, $C(\sigma^*, \rho, \beta) = 0$. Therefore, $C(\sigma^*, \rho, \beta) = 0$ if and only if $(\rho, \beta)=(\rho^*, \beta^*).$
\end{proof}

We now derive a quantitative time-delayed stability estimate for the recovery of the full parameter vector $\theta=(\sigma,\rho,\beta)$. \comments{In contrast to Theorem~\ref{Th2}, the result in Theorem~\ref{thm:quant_ident_shifted} does not require $\sigma$ to be fixed.} The resulting bound controls the full parameter error in terms of the time-delayed cost functional (i.e. the value of cost functional at the approximate minimum) and an exponentially decaying synchronization residual. We will first need the following two somewhat technical propositions. These propositions separate the argument into two steps: first establishing an abstract parameter-error estimate from the mismatch map, and then expressing the required nondegeneracy condition in terms of computable partial derivative equations. In what follows, the \(L^2(\tau,\tau+T)\)-norm is taken with respect to the
normalized Lebesgue measure \(dt/T\). Thus,
\[
\|f\|_{L^2(\tau,\tau+T)}^2
:=
\frac{1}{T}\int_{\tau}^{\tau+T}|f(t)|^2\,dt.
\]

\begin{prop}
\label{prop:quant_ident_shifted}
Let $X(t)=(x(t),y(t),z(t))$ be a  trajectory of \eqref{lorenz} corresponding to the parameter vector. 
$\theta^*=(\sigma^*,\rho^*,\beta^*)$.  For each
$\theta=(\sigma,\rho,\beta)$ in a neighborhood $B(\theta^*)$ of $\theta^*$, let $\widetilde X(t;\theta)
=
(\widetilde x(t;\theta),\widetilde y(t;\theta),\widetilde z(t;\theta))$
denote the solution of the nudged system \eqref{nudgedsystem}. The nudged system
is initialized at an arbitrary initial condition, not necessarily equal to
$X(0)$. For $\tau\geq 0$ and $T>0$, define the observation map
\[
{\mathcal F}(\theta)(t)
=
x(t)-\widetilde x(t;\theta),
\qquad t\in[\tau,\tau+T],
\]
viewed as a map ${\mathcal F}:B(\theta^*)\subset\mathbb R^3\to L^2(\tau,\tau+T)$. Note that, the cost functional, given in Eq.~\eqref{cost}, can now be written as
\[
C(\theta)
=
 \lVert {\mathcal F}(\theta)\rVert_{L^2(\tau,\tau+T)}^2
= \frac 1T
\int_{\tau}^{\tau+T}
|x(t)-\widetilde x(t;\theta)|^2\,dt.
\]
Assume further that the conditions on $\mu$ and $\gamma$ are as in Theorem~\ref{thm2}. 
\begin{itemize}
 \item[(i)] We have the estimate  
 \begin{equation}  \label{cest}
     C(\theta) \le e^{-2 \gamma \tau}\left( 2\fk +\frac{\|h\|^2}{2\gamma }\right)
     \left(
\frac{1-e^{-2\gamma T}}{2\gamma T}
\right), \ \text{where}\ h={\mbf \theta} - {\mbf \theta}^*.
 \end{equation}
\item[(ii)] ${\cal F}$ is continuously differentiable in $B(\theta^*)$. Let $D$ denote the differentiation with respect to the parameter vector $\theta$ and $U\subset B(\theta^*)$ be a convex neighborhood of $\theta^*$. Thus, for any $\theta\in U$, the line segment $\theta_s=\theta^*+s(\theta-\theta^*), \ 0\leq s\leq 1,$ is contained in $U.$ Assume that there exists $\alpha>0$ such that
\[
\lVert D{\mathcal F}(\theta^*)h\rVert_{L^2(\tau,\tau+T)}
\geq
\alpha\lVert h\rVert_{\mathbb R^3}
\qquad
\text{for all }h\in\mathbb R^3,
\]
and
\[
\lVert (D{\mathcal F}(\eta)-D{\mathcal F}(\theta^*))h\rVert_{L^2(\tau,\tau+T)}
\leq
\frac{\alpha}{2}\lVert h\rVert_{\mathbb R^3}
\qquad
\text{for all }\eta\in U
\text{ and all }h\in\mathbb R^3.
\]
 Then, for every $\theta\in U$,
\[
\lVert \theta-\theta^*\rVert_{\mathbb R^3}
\leq
\frac{2}{\alpha}C(\theta)^{1/2}
+
\frac{2}{\alpha}
\left(
\frac{1-e^{-2\gamma T}}{2\gamma T}
\right)^{1/2}
e^{-\gamma\tau}
\lVert X(0)-\widetilde X(0;\theta^*)\rVert_{\mathbb R^3}.
\]
\end{itemize}
\end{prop}

\begin{proof}

\begin{enumerate}
    \item[(i)] Applying Theorem~\ref{thm2} with
\(\theta=\theta^*+h\), and using the fact that the observed mismatch is only the
\(x\)-component of the full synchronization error, we obtain
\[
|x(t)-\widetilde x(t;\theta)|^2
\leq
\|X(t)-\widetilde X(t;\theta)\|_{\mathbb R^3}^2
\leq
e^{-2\gamma t}
\left(
2\fk+\frac{\|h\|^2}{2\gamma}
\right).
\]
Therefore, integrating this estimate over the time-delayed window
\([\tau,\tau+T]\) gives
\[
C(\theta)
=
\frac1T\int_\tau^{\tau+T}
|x(t)-\widetilde x(t;\theta)|^2\,dt
\leq
e^{-2\gamma\tau}
\left(
2\fk+\frac{\|h\|^2}{2\gamma}
\right)
\left(
\frac{1-e^{-2\gamma T}}{2\gamma T}
\right),
\]
which is precisely \eqref{cest}.

\item[(ii)] \noindent{\em Proof of Part (ii):}
Let \(h=\theta-\theta^*\). By the fundamental theorem of calculus in
Banach spaces,
\[
{\mathcal F}(\theta)-{\mathcal F}(\theta^*)
=
\int_0^1 D{\mathcal F}(\theta^*+s h)h\,ds .
\]
Hence,
\[
{\mathcal F}(\theta)
=
{\mathcal F}(\theta^*) + D{\mathcal F}(\theta^*)h
+
\int_0^1
\left[
D{\mathcal F}(\theta^*+s h)-D{\mathcal F}(\theta^*)
\right]h\,ds .
\]
Taking the \(L^2(\tau,\tau+T)\)-norm and using the reverse triangle
inequality gives
\[
\|{\mathcal F}(\theta)\|_{L^2(\tau,\tau+T)}
\geq
\|D{\mathcal F}(\theta^*)h\|_{L^2(\tau,\tau+T)}
-
\|{\mathcal F}(\theta^*)\|_{L^2(\tau,\tau+T)}
-
\int_0^1
\left\|
\left[
D{\mathcal F}(\theta^*+s h)-D{\mathcal F}(\theta^*)
\right]h
\right\|_{L^2(\tau,\tau+T)}
\,ds .
\]
By the assumed lower bound on \(D{\mathcal F}(\theta^*)\),
\[
\|D{\mathcal F}(\theta^*)h\|_{L^2(\tau,\tau+T)}
\geq
\alpha \|h\|_{\mathbb R^3}.
\]
Moreover, since \(\theta^*+s h\in U\) for \(0\leq s\leq 1\), the continuity
assumption on \(D{\mathcal F}\) implies
\[
\int_0^1
\left\|
\left[
D{\mathcal F}(\theta^*+s h)-D{\mathcal F}(\theta^*)
\right]h
\right\|_{L^2(\tau,\tau+T)}
\,ds
\leq
\frac{\alpha}{2}\|h\|_{\mathbb R^3}.
\]
Therefore,
\[
\|{\mathcal F}(\theta)\|_{L^2(\tau,\tau+T)}
\geq
\frac{\alpha}{2}\|h\|_{\mathbb R^3}
-
\|{\mathcal F}(\theta^*)\|_{L^2(\tau,\tau+T)}.
\]
Rearranging, we obtain
\[
\|h\|_{\mathbb R^3}
\leq
\frac{2}{\alpha}
\|{\mathcal F}(\theta)\|_{L^2(\tau,\tau+T)}
+
\frac{2}{\alpha}
\|{\mathcal F}(\theta^*)\|_{L^2(\tau,\tau+T)}.
\]
Since
\[
C(\theta)
=
\|{\mathcal F}(\theta)\|_{L^2(\tau,\tau+T)}^2,
\]
we have
\[
\|{\mathcal F}(\theta)\|_{L^2(\tau,\tau+T)}
=
C(\theta)^{1/2}.
\]

It remains to estimate \(\|{\mathcal F}(\theta^*)\|_{L^2(\tau,\tau+T)}\).
At the true parameter value, Theorem~\ref{thm2} gives
\[
|\widetilde x(t;\theta^*)-x(t)|
\leq
e^{-\gamma t}
\|X(0)-\widetilde X(0;\theta^*)\|_{\mathbb R^3}.
\]
Thus, using the normalized \(L^2\)-norm,
\[
\begin{aligned}
\|{\mathcal F}(\theta^*)\|_{L^2(\tau,\tau+T)}
&=
\left(
\frac1T
\int_\tau^{\tau+T}
|x(t)-\widetilde x(t;\theta^*)|^2\,dt
\right)^{1/2}  \\
&\leq
\left(
\frac1T
\int_\tau^{\tau+T}
e^{-2\gamma t}\,dt
\right)^{1/2}
\|X(0)-\widetilde X(0;\theta^*)\|_{\mathbb R^3} \\
&=
e^{-\gamma\tau}
\left(
\frac{1-e^{-2\gamma T}}{2\gamma T}
\right)^{1/2}
\|X(0)-\widetilde X(0;\theta^*)\|_{\mathbb R^3}.
\end{aligned}
\]
Combining the above estimates yields
\[
\|\theta-\theta^*\|_{\mathbb R^3}
\leq
\frac{2}{\alpha}C(\theta)^{1/2}
+
\frac{2}{\alpha}
e^{-\gamma\tau}
\left(
\frac{1-e^{-2\gamma T}}{2\gamma T}
\right)^{1/2}
\|X(0)-\widetilde X(0;\theta^*)\|_{\mathbb R^3}.
\]
This proves the result. \end{enumerate}
\end{proof}

Proposition~\ref{prop:quant_ident_shifted} shows that the cost functional controls the parameter error, provided the derivative \(D{\mathcal F}(\theta^*)\) is bounded below and this lower bound persists locally. The remaining issue is how this nondegeneracy condition can be checked for the Lorenz nudged system. Proposition~\ref{prop:partial_derivative_equations} derives the partial derivatives of the nudged trajectory with respect to \(\sigma\), \(\rho\), and \(\beta\). It also shows that the required lower bound is equivalent to a positive-definiteness condition on a Gram matrix built from the observed \(x\)-components of these partial derivatives.

\begin{prop}
\label{prop:partial_derivative_equations}
Let $\theta^*=(\sigma^*,\rho^*,\beta^*)$ be the true parameter vector, and let $\tilde{\theta}=(\tilde{\sigma},\tilde{\rho},\tilde{\beta})$ denote the parameter vector used in the nudged Lorenz system \eqref{nudgedsystem}.  Fix \(\tau\geq 0\) and \(T>0\), and set $I_{\tau,T}:=[0,\tau+T].$ Assume that the observed trajectory \(x(t)\) is continuous on \(I_{\tau,T}\), and that the initial condition for the nudged system is independent of \(\tilde{\theta}\). Let \(U\) be a compact neighborhood of \(\theta^*\) such that $\tilde{\sigma}>0, \tilde{\beta}>0,\tilde{\theta}\in U.$ Assume further that the nudging coefficient \(\mu\) satisfies the condition
\eqref{eq:mu_condition_mismatch} for all \(\tilde{\theta}\in U\). Let $\tilde X^*(t)
=
(\tilde{x}^*(t),\tilde{y}^*(t),\tilde{z}^*(t))
:=
\tilde X(t;\theta^*),
\theta^*=(\sigma^*,\rho^*,\beta^*).$
Define
\[
(p_\sigma,q_\sigma,r_\sigma)
:=
\partial_{\tilde{\sigma}}
(\tilde{x},\tilde{y},\tilde{z})
\big|_{\tilde{\theta}=\theta^*},
\]
\[
(p_\rho,q_\rho,r_\rho)
:=
\partial_{\tilde{\rho}}
(\tilde{x},\tilde{y},\tilde{z})
\big|_{\tilde{\theta}=\theta^*},
\qquad
(p_\beta,q_\beta,r_\beta)
:=
\partial_{\tilde{\beta}}
(\tilde{x},\tilde{y},\tilde{z})
\big|_{\tilde{\theta}=\theta^*}.
\]
\begin{itemize}
\item[(i)] (Equations for the partial derivatives) The partial derivatives satisfy the following systems on \(I_{\tau,T}\).

For \(\tilde{\sigma}\),
\[
\begin{cases}
p_\sigma'
=
(\tilde{y}^*-\tilde{x}^*)
-(\sigma^*+\mu)p_\sigma
+\sigma^*q_\sigma,\\[2pt]
q_\sigma'
=
(\rho^*-\tilde{z}^*)p_\sigma
-\tilde{x}^*r_\sigma
-q_\sigma,\\[2pt]
r_\sigma'
=
\tilde{y}^*p_\sigma
+\tilde{x}^*q_\sigma
-\beta^*r_\sigma.
\end{cases}
\]

For \(\tilde{\rho}\),
\[
\begin{cases}
p_\rho'
=
-(\sigma^*+\mu)p_\rho
+\sigma^*q_\rho,\\[2pt]
q_\rho'
=
\tilde{x}^*
+
(\rho^*-\tilde{z}^*)p_\rho
-\tilde{x}^*r_\rho
-q_\rho,\\[2pt]
r_\rho'
=
\tilde{y}^*p_\rho
+\tilde{x}^*q_\rho
-\beta^*r_\rho.
\end{cases}
\]

For \(\tilde{\beta}\),
\[
\begin{cases}
p_\beta'
=
-(\sigma^*+\mu)p_\beta
+\sigma^*q_\beta,\\[2pt]
q_\beta'
=
(\rho^*-\tilde{z}^*)p_\beta
-\tilde{x}^*r_\beta
-q_\beta,\\[2pt]
r_\beta'
=
\tilde{y}^*p_\beta
+\tilde{x}^*q_\beta
-\beta^*r_\beta
-\tilde{z}^*.
\end{cases}
\]
Moreover, since the initial condition of the nudged system is independent of
\(\tilde{\theta}\),
\[
(p_j,q_j,r_j)(0)=(0,0,0),
\qquad
j\in\{\sigma,\rho,\beta\}.
\]
\item[(ii)] (Gram matrix lower bound) Define ${\mathcal F}(\tilde{\theta})(t)
=
x(t)-\tilde{x}(t;\tilde{\theta}),
\ t\in[\tau,\tau+T].$
Let
\[
\langle f,g\rangle_{\tau,T}
:=
\int_{\tau}^{\tau+T} f(t)g(t)\,dt.
\]
Define the Gram matrix
\[
G_{\tau,T}
=
\begin{pmatrix}
\langle p_\sigma,p_\sigma\rangle_{\tau,T} &
\langle p_\sigma,p_\rho\rangle_{\tau,T} &
\langle p_\sigma,p_\beta\rangle_{\tau,T} \\[2mm]
\langle p_\rho,p_\sigma\rangle_{\tau,T} &
\langle p_\rho,p_\rho\rangle_{\tau,T} &
\langle p_\rho,p_\beta\rangle_{\tau,T} \\[2mm]
\langle p_\beta,p_\sigma\rangle_{\tau,T} &
\langle p_\beta,p_\rho\rangle_{\tau,T} &
\langle p_\beta,p_\beta\rangle_{\tau,T}
\end{pmatrix},
\]
where \(p_\sigma,p_\rho,p_\beta\) are the \(x\)-components of the partial derivatives
from part (i). Assume that
\(G_{\tau,T}\) is positive definite. Then, for every
\(h=(h_\sigma,h_\rho,h_\beta)\in\mathbb{R}^3\),
\[
\lVert D{\mathcal F}(\theta^*)h\rVert_{L^2(\tau,\tau+T)}
\geq
\sqrt{\lambda_{\min}(G_{\tau,T})}
\lVert h\rVert_{\mathbb{R}^3}.
\]
\end{itemize}
\end{prop}

\begin{proof}
We defer the proof to Appendix~\ref{A.1}.
\end{proof}

We now combine the quantitative estimate in Proposition~\ref{prop:quant_ident_shifted} with the computable nondegeneracy condition in Proposition~\ref{prop:partial_derivative_equations}. This gives the main post-transient result for the nudging-based cost functional used in the optimization procedure. The theorem states that the true parameter gives zero post-transient mismatch in the noise-free case, that multiplicative observational noise produces a controlled increase in the mismatch at the true parameter, and that, under the stated nondegeneracy assumptions, the post-transient cost functional controls the full parameter error. Denote
\[
C_{\textup{lim}}(\mbf{\theta})=\limsup_{\tau \to \infty} C(\mbf{\theta}).
\]
We refer to \(C_{\textup{lim}}(\mbf{\theta})\) as the \emph{post-transient cost functional}.

\begin{theorem}[Uniqueness and stability of the post-transient cost-functional minimizer] \label{thrm7}
Let $X(t)=(x(t),y(t),z(t)), t \in [t_0,\infty)$ be a trajectory of \eqref{lorenz} corresponding to the parameter ${\mbf \theta}^* =(\sigma^*, \rho^*,\beta^*)$. Let $\tilde{X}_{\mbf \theta}(t)=(x_{\mbf \theta}(t), y_{\mbf \theta}(t),z_{\mbf \theta}(t))$ be the solution of the \eqref{nudgedsystem} corresponding to the parameter ${\mbf \theta}=(\sigma, \rho, \beta)$. 
\begin{itemize}
\item[(i)]  We have the equality $C_{\textup{lim}}({\mbf \theta}^*)=0$.
\item[(ii)] Let ${\cal O}(t)=x(t)(1 +\varepsilon(t))$, where $\varepsilon(t)$ is the observational error. Let $\tilde{X}_{\mbf \theta}$ be the solution of \eqref{nudgedsystem} where the true observation $x(t)$ is replaced by the contaminated observation ${\cal O}(t)$. Then, we have the estimate $C_{\textup{lim}}({\mbf \theta}^*)\le C_\fk \|\varepsilon\|_{\infty}^2.$.
\item[(iii)] Assume that there exists $\alpha>0$ such that
\[
\liminf_{\tau \to \infty} \lVert D{\mathcal F}(\theta^*)h\rVert_{L^2(\tau,\tau+T)}
\geq
\alpha\lVert h\rVert_{\mathbb R^3}
\qquad
\text{for all }h\in\mathbb R^3,
\]
and
\[
\limsup_{\tau \to \infty} \lVert (D{\mathcal F}(\eta)-D{\mathcal F}(\theta^*))h\rVert_{L^2(\tau,\tau+T)}
\leq
\frac{\alpha}{2}\lVert h\rVert_{\mathbb R^3}
\qquad
\text{for all }\eta\in U
\text{ and all }h\in\mathbb R^3.
\]
 Then, for every $\theta\in U$, we have the estimate
 \[
 \|{\mbf \theta} - {\mbf \theta}^*\|_{\R^3} \le \frac{2}{\alpha} \left(C_{\textup{lim}}({\mbf \theta})\right)^{1/2}.
 \]
 \end{itemize}
\end{theorem}

\begin{proof}
The proof follows readily from Theorem \ref{thm2} and Proposition \ref{prop:quant_ident_shifted}. 
\end{proof}

\begin{rem}[Interpretation of the nondegeneracy assumptions in Theorem~\ref{thrm7}]
 The assumptions in part~(iii) of Theorem~\ref{thrm7} express a local nondegeneracy condition for the post-transient observation mismatch map
\[
{\mathcal F}(\theta)(t)=x(t)-\tilde{x}(t;\theta),
\qquad t\in[\tau,\tau+T].
\]
The first condition requires that the linearized map \(D{\mathcal F}(\theta^*)\) be bounded below as \(\tau\to\infty\). In other words, nonzero perturbations of the parameter vector must produce a nonzero first-order change in the observed component of the nudged trajectory.

Proposition~\ref{prop:partial_derivative_equations} provides a concrete way to verify this condition. Indeed, the derivative \(D{\mathcal F}(\theta^*)h\) is determined by the \(x\)-components of the partial derivative equations,
\[
p_\sigma
=
\partial_{\tilde{\sigma}}\tilde{x}\big|_{\tilde{\theta}=\theta^*},
\qquad
p_\rho
=
\partial_{\tilde{\rho}}\tilde{x}\big|_{\tilde{\theta}=\theta^*},
\qquad
p_\beta
=
\partial_{\tilde{\beta}}\tilde{x}\big|_{\tilde{\theta}=\theta^*}.
\]
Thus, the relevant information is encoded in the Gram matrix \(G_{\tau,T}\). If \(G_{\tau,T}\) is positive definite on the observation window, then Proposition~\ref{prop:partial_derivative_equations} gives
\[
\lVert D{\mathcal F}(\theta^*)h\rVert_{L^2(\tau,\tau+T)}
\geq
\sqrt{\lambda_{\min}(G_{\tau,T})}
\lVert h\rVert_{\mathbb R^3}.
\]
Therefore, the lower-bound constant in Theorem~\ref{thrm7} may be chosen in terms of the smallest eigenvalue of \(G_{\tau,T}\), provided this eigenvalue remains bounded away from zero along the post-transient windows.

The second condition in part~(iii) of Theorem~\ref{thrm7} requires this linearized nondegeneracy to persist uniformly for parameters \(\eta\in U\) near \(\theta^*\). This prevents the parameter-to-observation map from becoming nearly singular in a neighborhood of the true parameter.
Together, these two assumptions ensure that the post-transient cost functional controls the parameter error, leading to the estimate
\[
\|\theta-\theta^*\|_{\mathbb R^3}
\leq
\frac{2}{\alpha}
\left(C_{\textup{lim}}(\theta)\right)^{1/2}.
\]

\end{rem}

%=====================  Section 5 Numerical Results  =====================

\section{Numerical Results}
\label{sec:results}
In this section, we present numerical results for the Lorenz system in both chaotic and non-chaotic regimes. To represent a chaotic regime, we choose the true parameter values
\[
\sigma^*=10, \qquad \rho^*=30, \qquad \beta^*=\frac{8}{3} \approx 2.67,
\]
whereas to represent the non-chaotic regime, we choose
\[
\sigma^*=10, \qquad \rho^*=14, \qquad \beta^*=\frac{8}{3} \approx 2.67.
\]
Unless otherwise stated, the initial condition for the reference Lorenz system is
\[
(x_0,y_0,z_0)=(1,1,1),
\]
and, in the nudging and parameter estimation experiments, the initial condition for the nudged system is
\[
(\tilde{x}_0,\tilde{y}_0,\tilde{z}_0)=(1,0,0).
\]
Noisy observations are generated using the noise model described in Eq.~\eqref{eq4}. For parameter estimation, the cost functional in Eq.~\eqref{cost} is minimized using the Nelder--Mead algorithm unless another optimization method is explicitly specified. 

%=====================  Section 5.1 Identifiability and Sensitivity  =====================

\subsection{Parameter Identifiability and Sensitivity for the Lorenz System}
This subsection examines structural identifiability, practical identifiability, and parameter sensitivity for the Lorenz system, following the framework discussed in Section~\ref{sec:identifiability}. The analysis is performed in both the chaotic and non-chaotic regimes using either \(x\)- or \(y\)-observations. For the sensitivity analysis, the parameters are sampled within a \(5\%\) neighborhood of their true values using a Sobol sequence. The Lorenz system is then simulated for each sample, and the resulting output variance is used to estimate the sensitivity indices.

%===========================================================================

\subsubsection{Chaotic Regime}\label{chaotic-ident}
We begin with the identifiability and sensitivity analysis in the chaotic regime. The results are organized into two parts: structural and practical identifiability, followed by Sobol sensitivity analysis.

%===========================================================================

\paragraph{Parameter Identifiability.}
Structural identifiability is assessed through the input--output equation obtained from \(x(t)\)-observations. For the Lorenz system in Eq.~\eqref{lorenz}, the variables \(y(t)\) and \(z(t)\) can be eliminated from the governing equations. This yields the following input--output differential equation involving only \(x(t)\) and its derivatives:
\begin{equation}\label{LorenzInOut}
    x\dddot{x}
-\dot{x}\ddot{x}
+(\sigma+\beta+1)x\ddot{x}
-(\sigma+1)\dot{x}^{2}
+x^{3}\dot{x}
+\sigma x^{4}
+\beta(\sigma+1)x\dot{x}
+\sigma\beta(1-\rho)x^{2}
=0.
\end{equation}
Both DAISY and SIAN confirm that all parameters are globally structurally identifiable. {\em However, we note that $x(t)=y(t)=0, z(t)=z_0e^{- \beta^* t}, t \in \R$ is a global trajectory of the Lorenz system, and for this trajectory, none of the parameters is identifiable if one observes $x$}, which validates our discussion in Theorem~\ref{thm:wellposedness}. Theorem~\ref{thm:wellposedness} identifies the exceptional set
\[
{\cal S}
=
\{(x,y,z): y=x \ \text{or}\ x=0 \ \text{or}\ z=0\}.
\]
The trajectory above belongs to this set for all time. Thus, the identifiability conclusion should be interpreted away from such degenerate trajectories, while Theorem~\ref{thm:wellposedness} identifies the exceptional dynamical situations that must be excluded in order to obtain a well-posed data-to-parameter inverse map.

Practical identifiability is assessed on the time interval \([0,30]\). The results in Table~\ref{tab1} show that the ARE values for \(\sigma\) and \(\beta\) exceed the corresponding noise levels for both \(x\)-only and \(y\)-only observations. For \(\rho\), the ARE values are also above the noise level at lower noise levels and remain around \(6\%\) across the experiments. According to the criterion described in Section~\ref{sec:sidentifiability2}, these results indicate that the full parameter vector \((\sigma,\rho,\beta)\) is not practically identifiable from either \(x\)-only or \(y\)-only observations in this setting.

\begin{table}[H]
\centering
\begin{tabular}{c cc cc cc}
\hline
\multirow{2}{*}{Noise Level $(\%)$}
& \multicolumn{2}{c}{$ARE_{\sigma} (\%)$} 
& \multicolumn{2}{c}{$ARE_{\rho}(\%)$}
& \multicolumn{2}{c}{$ARE_{\beta}(\%)$} \\
\cline{2-7}
& $x-$ observ. & $y-$ observ.
& $x-$ observ. & $y-$ observ.
& $x-$ observ. & $y-$ observ. \\
\hline

$2$  
& 16.02 &   13.25
& 6.01 &  6.10
& 18.84 &  20.29 \\

 $4$
& 16.11 &  12.89
& 6.03 &   6.12
& 18.85 &   20.47 \\

$6$  
& 16.06 &  13.17
& 6.07 &  6.11
& 18.76 &  20.37 \\

 $8$ 
& 16.20 & 12.84  
& 6.04 &  6.10
& 18.96 &   20.28 \\

 $10$ 
& 16.16 &  13.32
& 6.04 &  6.13
& 18.77 & 20.41 \\
\hline
\end{tabular}
\caption{Average Relative Percentage Error (ARE) for Lorenz system in chaotic regime using only \(x\)-observations or only \(y\)-observations.}
\label{tab1}
\end{table}

%===========================================================================
\paragraph{Parameter Sensitivity.}
 The influence of each parameter is quantified using Sobol sensitivity indices. This analysis helps explain how the relative importance of \(\sigma\), \(\rho\), and \(\beta\) changes along the chaotic trajectory.

Figure~\ref{fig3} presents the Sobol sensitivity indices for the Lorenz system in the chaotic regime. Figures~\ref{3a} and~\ref{3b} show the first-order Sobol sensitivity indices with respect to the \(x\)-observation and the \(y\)-observation, respectively. For sensitivity with respect to the \(x\)-observation, \(\rho\) and \(\beta\) exhibit noticeable early-time effects, while the sensitivity with respect to \(\sigma\) remains small throughout the interval. After the initial transient phase, the first-order indices for \(\rho\) and \(\beta\) also decrease substantially and remain close to zero for most of the remaining time interval. A similar pattern is observed for sensitivity with respect to the \(y\)-observation. Thus, for both observation choices, the first-order effects are mainly transient, with \(\rho\) and \(\beta\) contributing during the early dynamics and \(\sigma\) showing weak individual influence over the full interval.

Figures~\ref{3c} and~\ref{3d} show the total-order Sobol sensitivity indices with respect to the \(x\)-observation and the \(y\)-observation, respectively. For sensitivity with respect to the \(x\)-observation, \(\rho\) and \(\beta\) have large total-order effects during the early and intermediate time intervals, while the total-order sensitivity with respect to \(\sigma\) increases after the transient phase. A similar behavior is observed for sensitivity with respect to the \(y\)-observation. At later times, the total-order indices for all three parameters are close to one, indicating that the parameters remain influential collectively even when their individual first-order effects are small.

\begin{figure}[H] 
    \centering

    \begin{subfigure}{0.48\textwidth}
        \centering
        \includegraphics[width=0.75\textwidth]{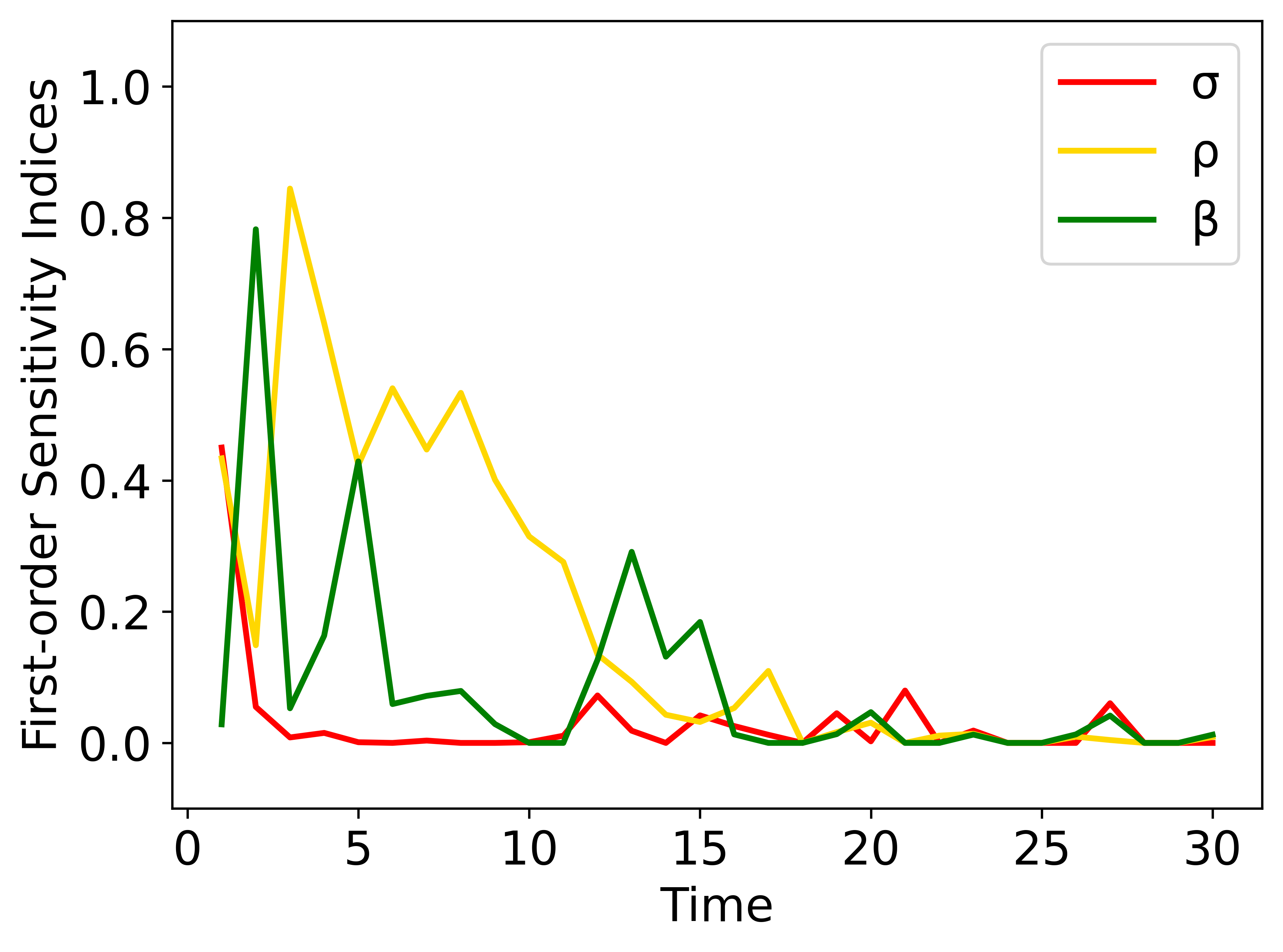}
        \caption{First-order sensitivity indices -- $x$ observable.}
        \label{3a}
    \end{subfigure}
    \hfill
    \begin{subfigure}{0.48\textwidth}
        \centering
        \includegraphics[width=0.75\textwidth]{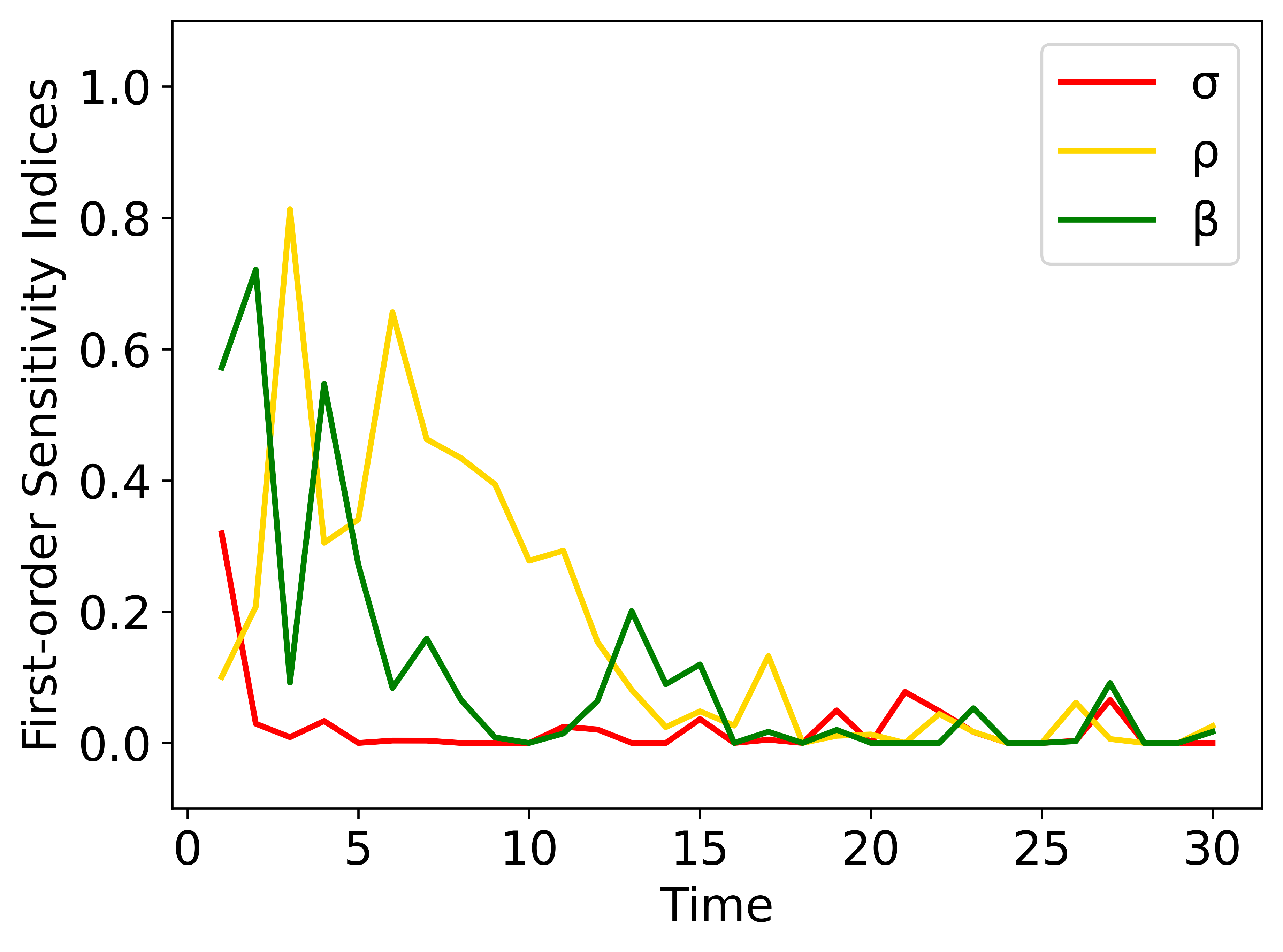}
        \caption{First-order sensitivity indices -- $y$ observable.}
        \label{3b}
    \end{subfigure}

    \par\medskip

    \begin{subfigure}{0.48\textwidth}
        \centering
        \includegraphics[width=0.75\textwidth]{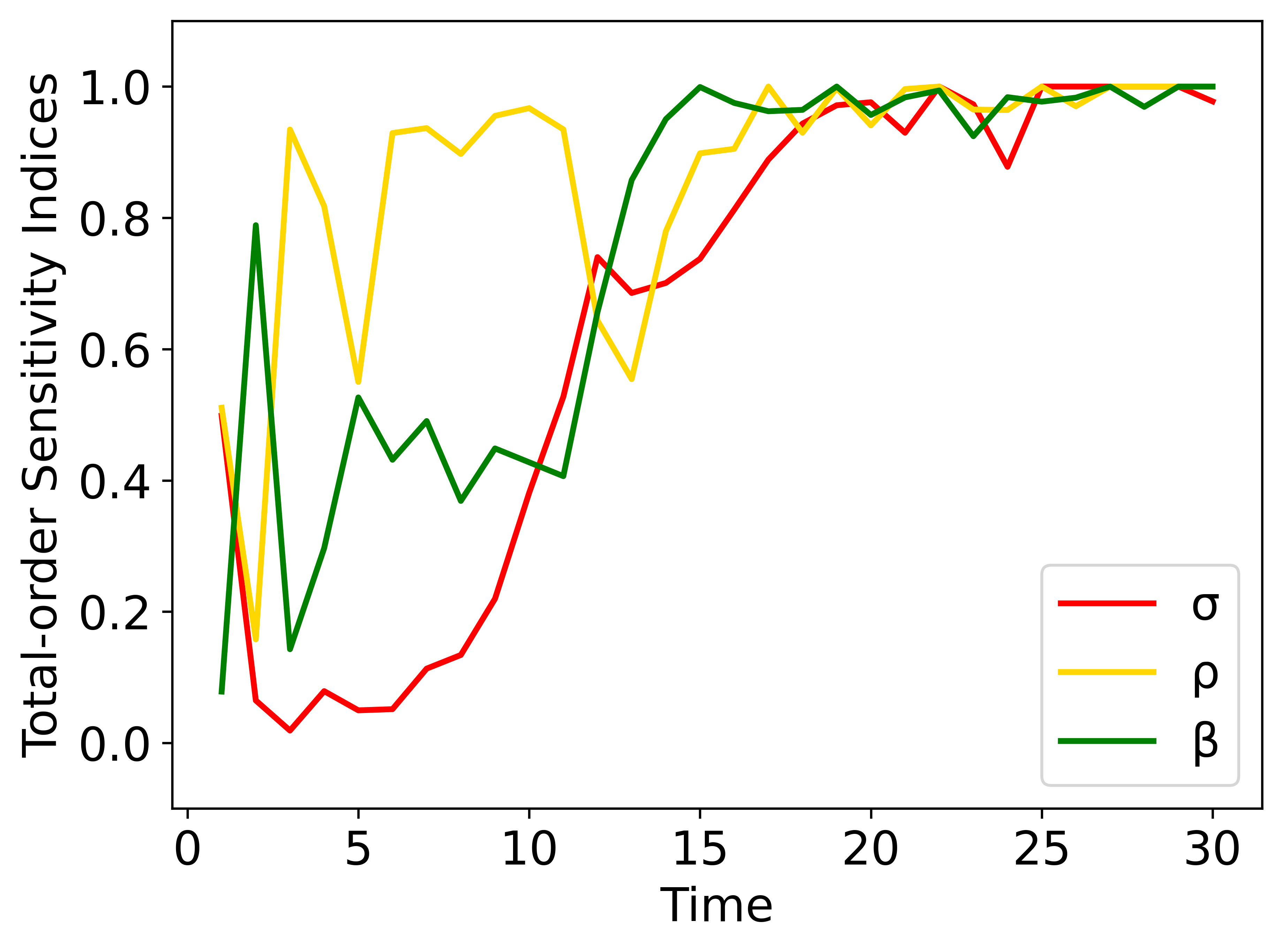}
        \caption{Total-order sensitivity indices -- $x$ observable.}
        \label{3c}
    \end{subfigure}
    \hfill
    \begin{subfigure}{0.48\textwidth}
        \centering
        \includegraphics[width=0.75\textwidth]{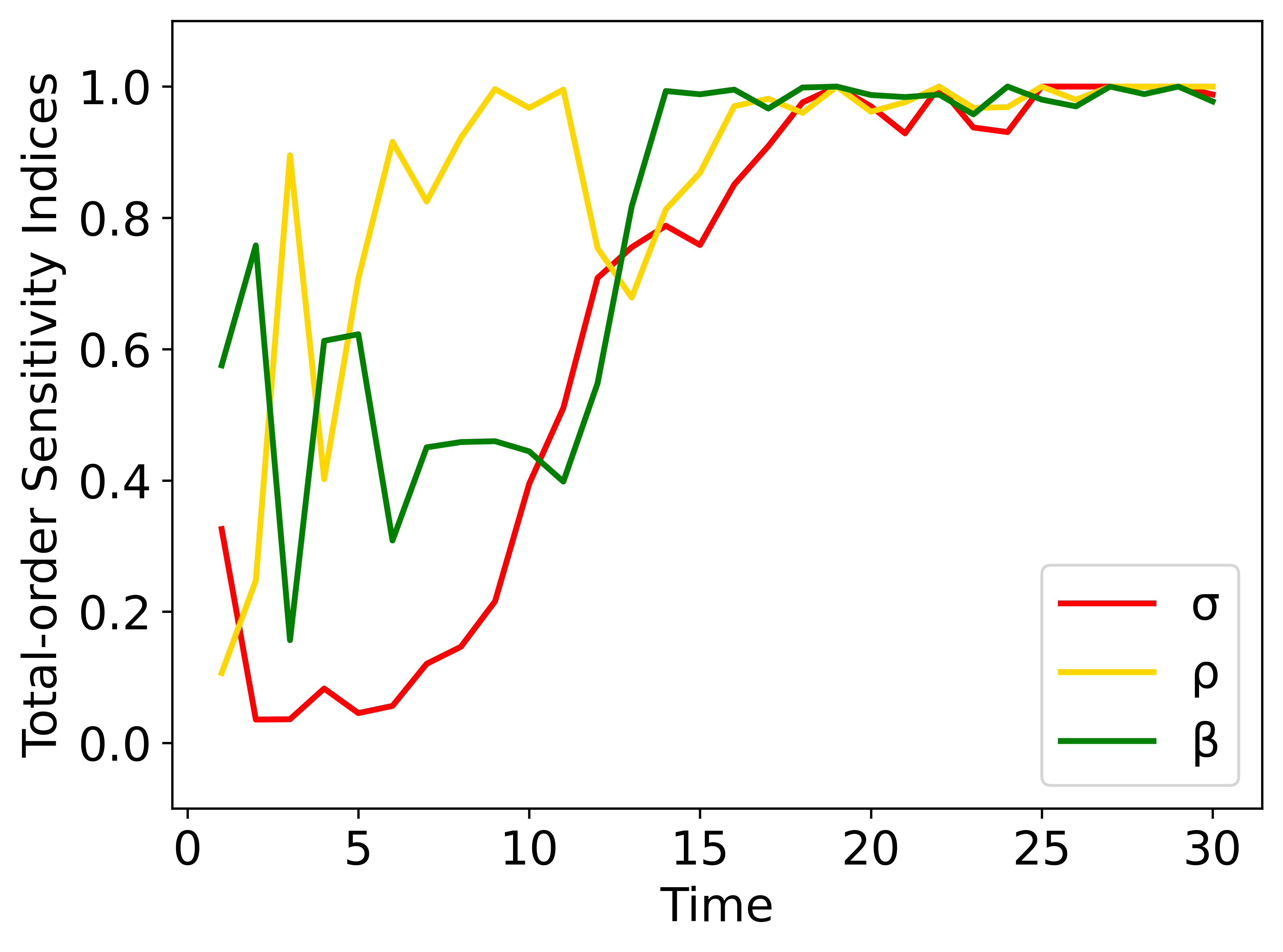}
        \caption{Total-order sensitivity indices -- $y$ observable.}
        \label{3d}
    \end{subfigure}

    \caption{Plots of the first-order and total-order Sobol sensitivity indices in
    the chaotic regime using \(x\)- and \(y\)-observations. Figures (a) and (b) show
    the first-order indices for \(x\)- and \(y\)-observations, respectively, while
    Figures (c) and (d) show the corresponding total-order indices. The abscissa
    denotes time, and the ordinate represents the sensitivity index for each
    parameter.}
    \label{fig3} 
\end{figure}

\subsubsection{Non-Chaotic Regime}

The same identifiability and sensitivity analysis is repeated in the non-chaotic regime. In this case, the trajectory approaches an equilibrium, which leads to different practical identifiability and sensitivity behavior.

%===========================================================================
\paragraph{Parameter Identifiability.}

The structural identifiability result is the same as in the chaotic regime, since structural identifiability depends on the model equations and the observation scheme rather than on the particular parameter regime used in the numerical experiment. Practical identifiability is then evaluated over the time interval \([0,30]\), using the same procedure as in the chaotic regime. The results in Table~\ref{tab5} show that all three parameters are practically identifiable from either \(x\)-observations or \(y\)-observations alone.

\begin{table}[H]
\centering
\begin{tabular}{c cc cc cc}
\hline
\multirow{2}{*}{Noise Level $(\%)$}
& \multicolumn{2}{c}{$ARE_{\sigma} (\%)$} 
& \multicolumn{2}{c}{$ARE_{\rho}(\%)$}
& \multicolumn{2}{c}{$ARE_{\beta}(\%)$} \\
\cline{2-7}
& $x-$ observ. & $y-$ observ.
& $x-$ observ. & $y-$ observ.
& $x-$ observ. & $y-$ observ. \\
\hline

$2$  
& 0.57  &  0.74
& 0.06 &  0.07
& 0.07 & 0.10 \\

 $4$
& 1.11 &  1.37
& 0.11 &  0.13 
& 0.15 &   0.20 \\

$6$  
& 1.60 & 2.03
& 0.17 &  0.20
& 0.22 &  0.32 \\

 $8$ 
&2.06 &  2.63
& 0.23  &  0.27
& 0.28 &  0.41  \\

 $10$ 
&2.47  &  3.18
& 0.28 &  0.33
& 0.35 & 0.51 \\
\hline
\end{tabular}
\caption{Average Relative Percentage Error (ARE) for Lorenz system in non-chaotic regime using only \(x\)-observations or only \(y\)-observations.}
\label{tab5}
\end{table}
%===========================================================================

\paragraph{Parameter Sensitivity.}

The first-order and total-order Sobol sensitivity indices are then computed in the non-chaotic regime. Figure~\ref{fig5} shows the corresponding indices using both \(x\)- and \(y\)-observations. Figures~\ref{5a} and \ref{5b} present the first-order indices for \(x\)- and \(y\)-observations, respectively. In both cases, the indices rapidly level off after a short transient. The parameters \(\rho\) and \(\beta\) exhibit persistently larger first-order effects, while the sensitivity index for \(\sigma\) remains close to zero throughout the time interval. This suggests that, in the non-chaotic regime, the direct contribution of \(\sigma\) to the observed dynamics is negligible compared with the effects of \(\rho\) and \(\beta\).

Figures~\ref{5c} and \ref{5d} show the corresponding total-order sensitivity indices for \(x\)- and \(y\)-observations, respectively. The total-order indices display the same overall behavior: \(\rho\) has the largest influence, followed by \(\beta\), while \(\sigma\) remains small. In contrast to the chaotic regime, the sensitivity profiles in the non-chaotic regime are smoother and approach nearly constant values over time.

\begin{figure}[H]
     \centering
     \begin{subfigure}{0.48\textwidth}
         \centering
         \includegraphics[width=0.75\textwidth]{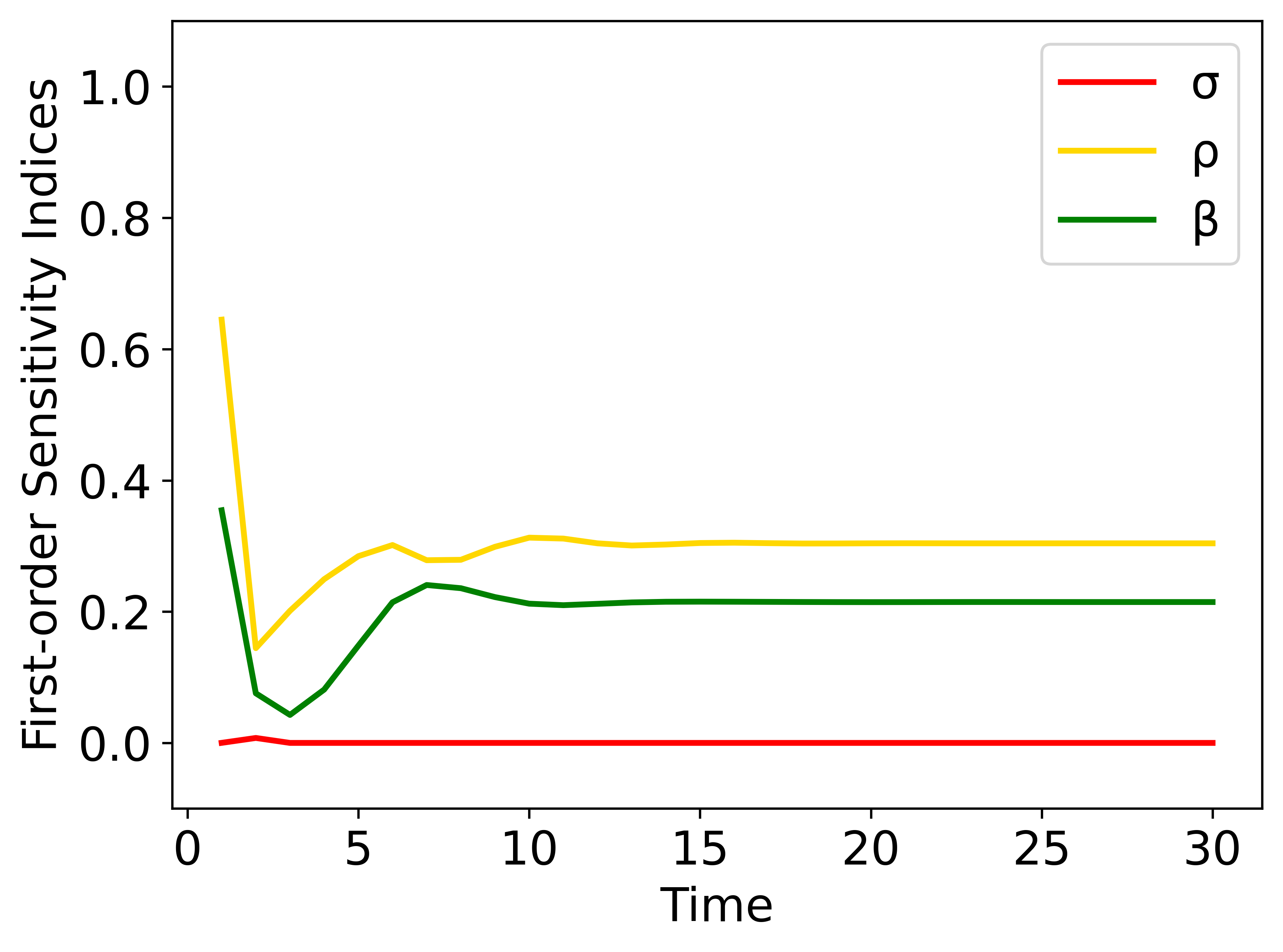}
         \caption{First order sensitivity indices -- $x$ observable.}
         \label{5a}
     \end{subfigure}
     \hfill
     \begin{subfigure}{0.48\textwidth}
         \centering
         \includegraphics[width=0.75\textwidth]{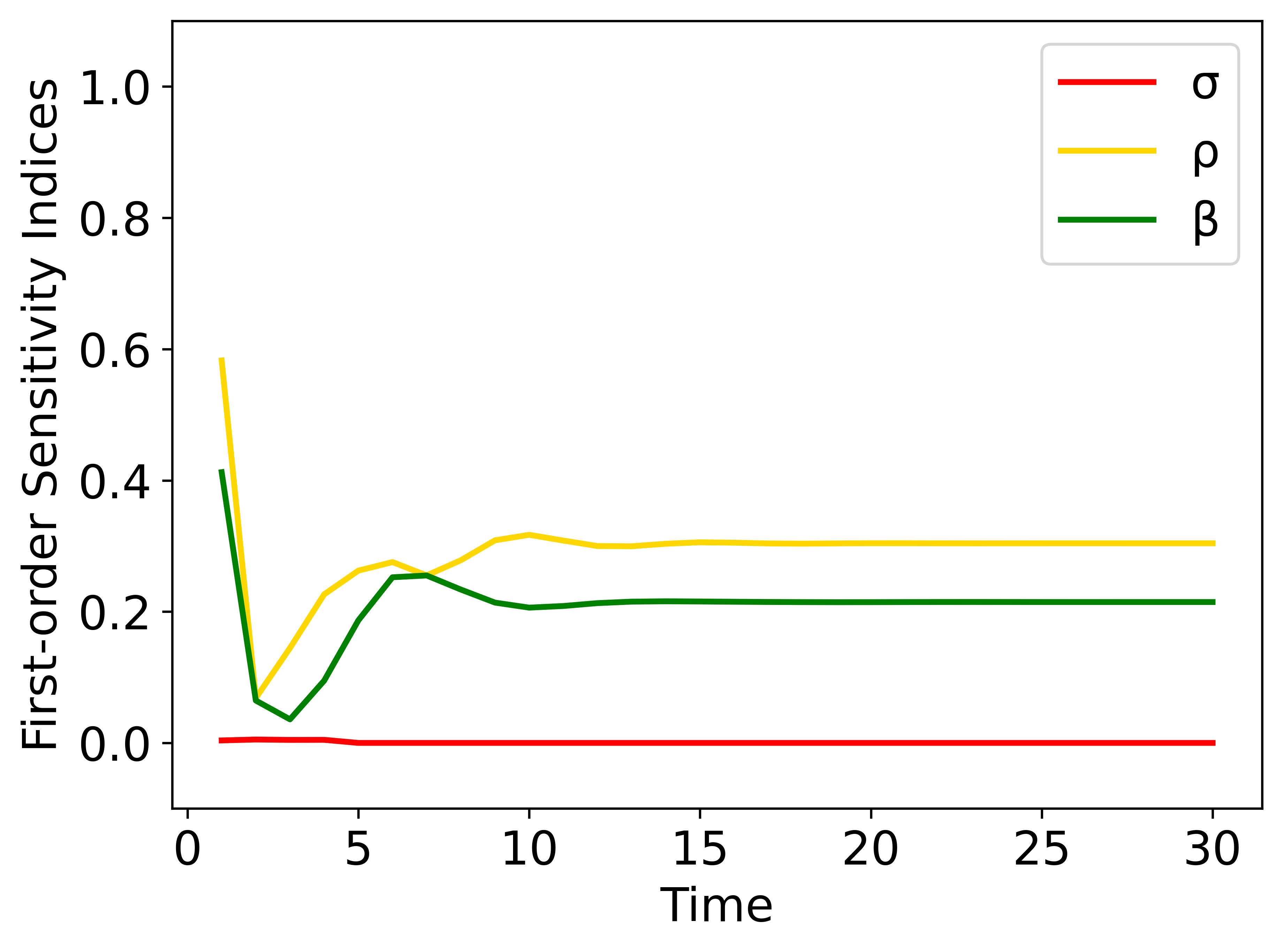}
         \caption{First order sensitivity indices -- $y$ observable.}
         \label{5b}
     \end{subfigure}\\
          \begin{subfigure}{0.48\textwidth}
         \centering
         \includegraphics[width=0.75\textwidth]{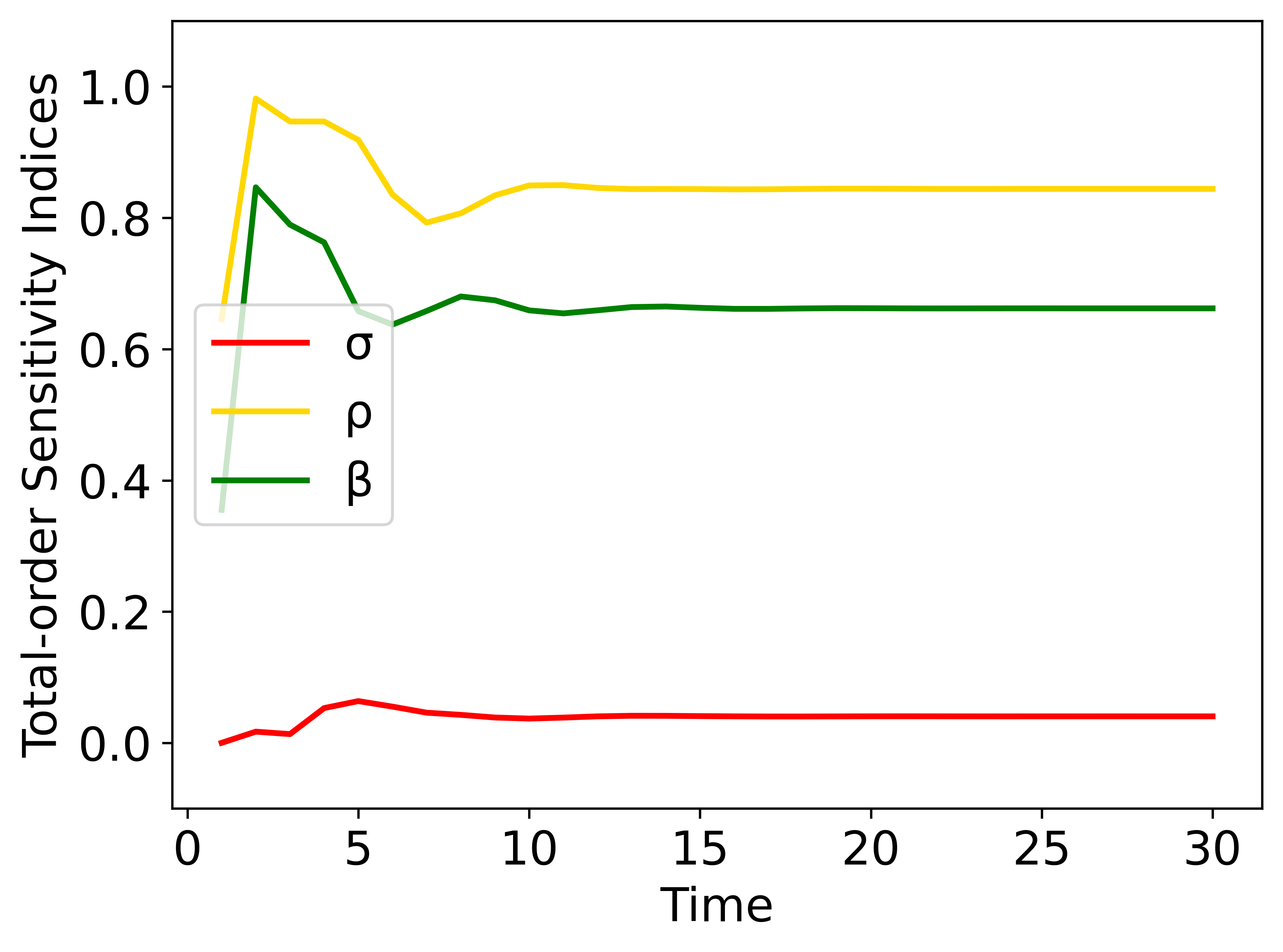}
         \caption{Total order sensitivity indices -- $x$ observable.}
         \label{5c}
     \end{subfigure}
     \hfill
     \begin{subfigure}{0.48\textwidth}
         \centering
         \includegraphics[width=0.75\textwidth]{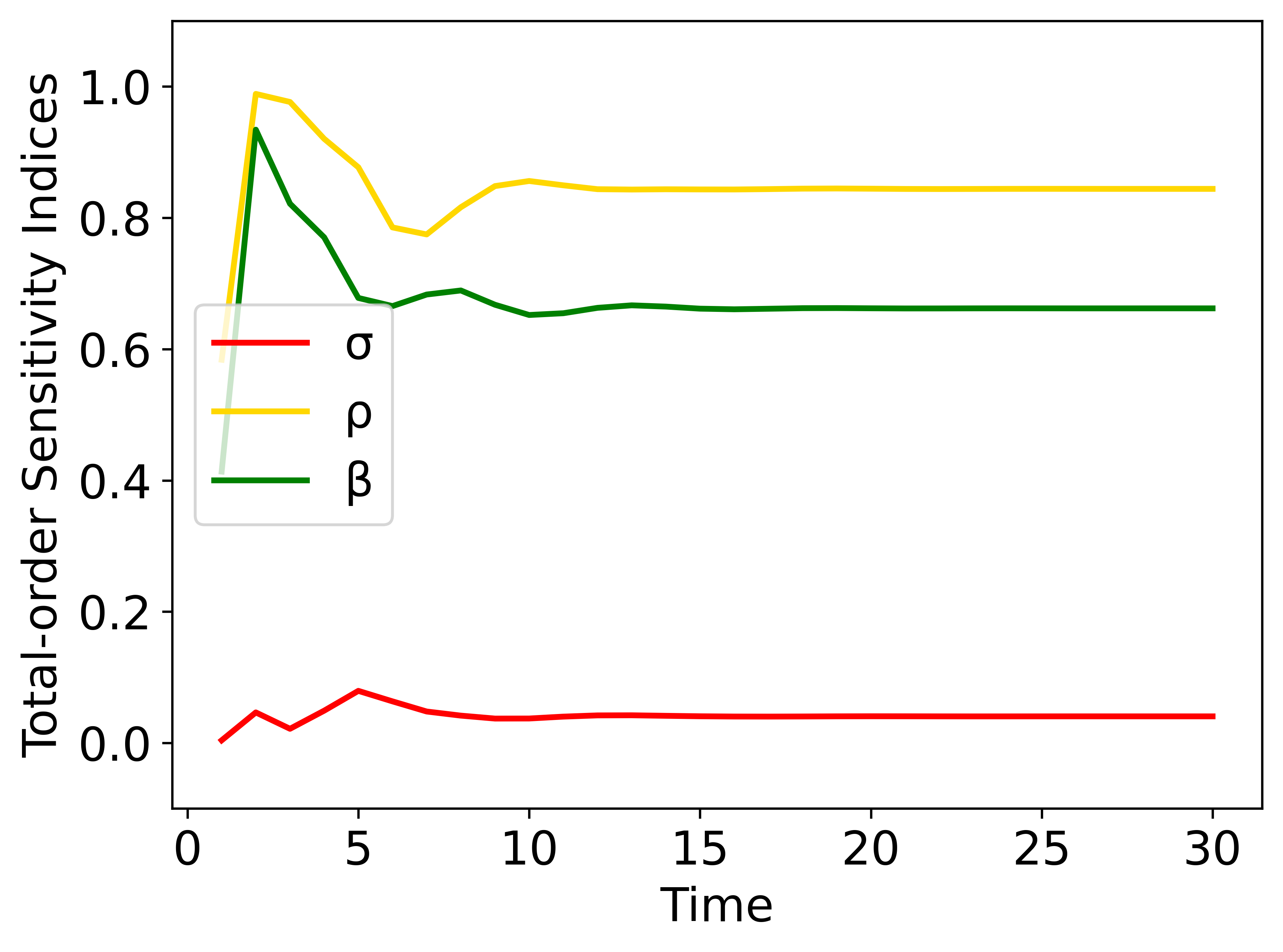}
         \caption{Total order sensitivity indices -- $y$ observable.}
         \label{5d}
     \end{subfigure}
        \caption{Plots of the first-order and total-order Sobol sensitivity indices in the non-chaotic regime using \(x\)- and \(y\)-observations. Panels (a) and (b) show the first-order indices for \(x\)- and \(y\)-observations, respectively, while panels (c) and (d) show the corresponding total-order indices. The abscissa denotes time, and the ordinate represents the sensitivity index for each parameter, scaled from 0 to 1.}
        \label{fig5} 
\end{figure}

\begin{rem}

The identifiability and sensitivity results provide a useful context for the parameter estimation problem studied in Section~\ref{pe}. In particular, we find that based on observing $x$ or $y$, in the chaotic regime, all parameters are structurally identifiable while none of the parameters are practically identifiable. However, all parameters are highly sensitive in terms of total sensitivity in the chaotic regime. We show subsequently that incorporating data assimilation allows us to reliably recover the parameters in the chaotic regime.

On the other hand, all parameters in the non-chaotic regime are both structurally and practically identifiable. The sensitivity analysis shows that $\sigma$ is not sensitive,  while the other parameters are highly sensitive. In this case also, data assimilation allows us to reliably estimate the parameters. 

 \end{rem}

%===========================================================================

\subsection{Forward Prediction Using Nudging}

This subsection presents forward prediction results for both chaotic and non-chaotic regimes, and include noise-free as well as noisy observations. We begin with \(x\)-observations and then compare the effect of observing different state variables.

For the chaotic regime, the nudged system is driven by noise-free \(x\)-observations with nudging coefficient \(\mu=100\). Figure~\ref{f1} shows the logarithm (base 10) of the absolute relative errors between the reference and nudged state variables, $|x-\tilde{x}|, |y-\tilde{y}|, |z-\tilde{z}|.$ The errors decay exponentially, showing that the nudged trajectory synchronizes with the reference trajectory. This behavior is consistent with the estimate in Eq.~\eqref{eq:mismatch_estimate}. The chaotic-regime experiment is then repeated using noisy \(x\)-observations. Figure~\ref{f-p-noisy} shows that, across the noise levels considered, the nudging approach maintains an accuracy of approximately \(10^{-2}\) in forward prediction.

\begin{figure}[H]
    \centering
         \includegraphics[width=0.6\textwidth]{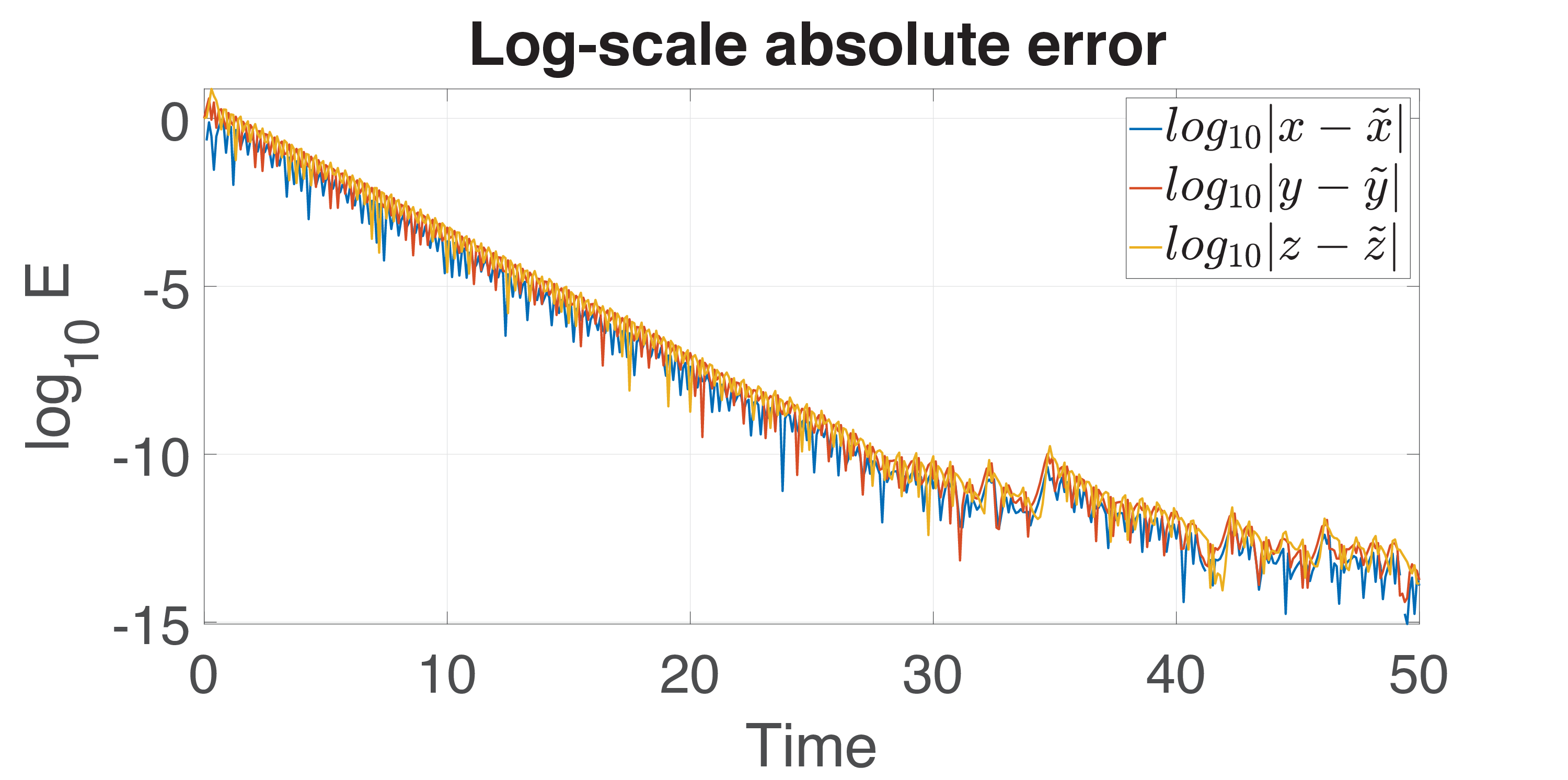}
    \caption{Logarithm (base 10) of the absolute relative errors between the state variables of the Lorenz system and the nudged system in the chaotic regime using noise-free $x$-observations.}
    \label{f1}
\end{figure}

\begin{figure}[H]
    \centering
    \begin{subfigure}{0.47\textwidth}
         \centering
         \includegraphics[width=\textwidth]{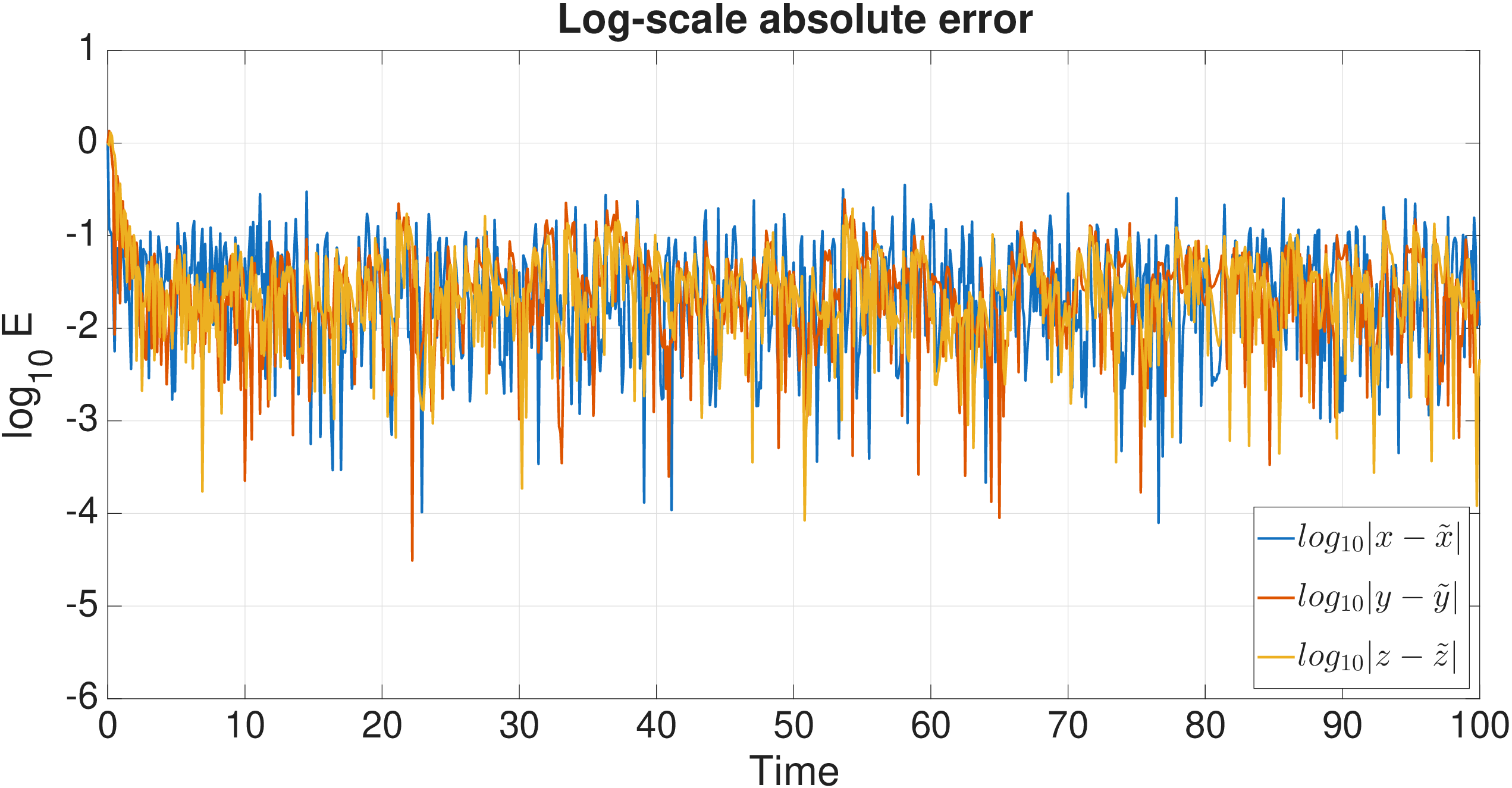} 
         \caption{2\% noise.}
     \end{subfigure}
     \hfill
         \begin{subfigure}{0.47\textwidth}
         \centering
         \includegraphics[width=\textwidth]{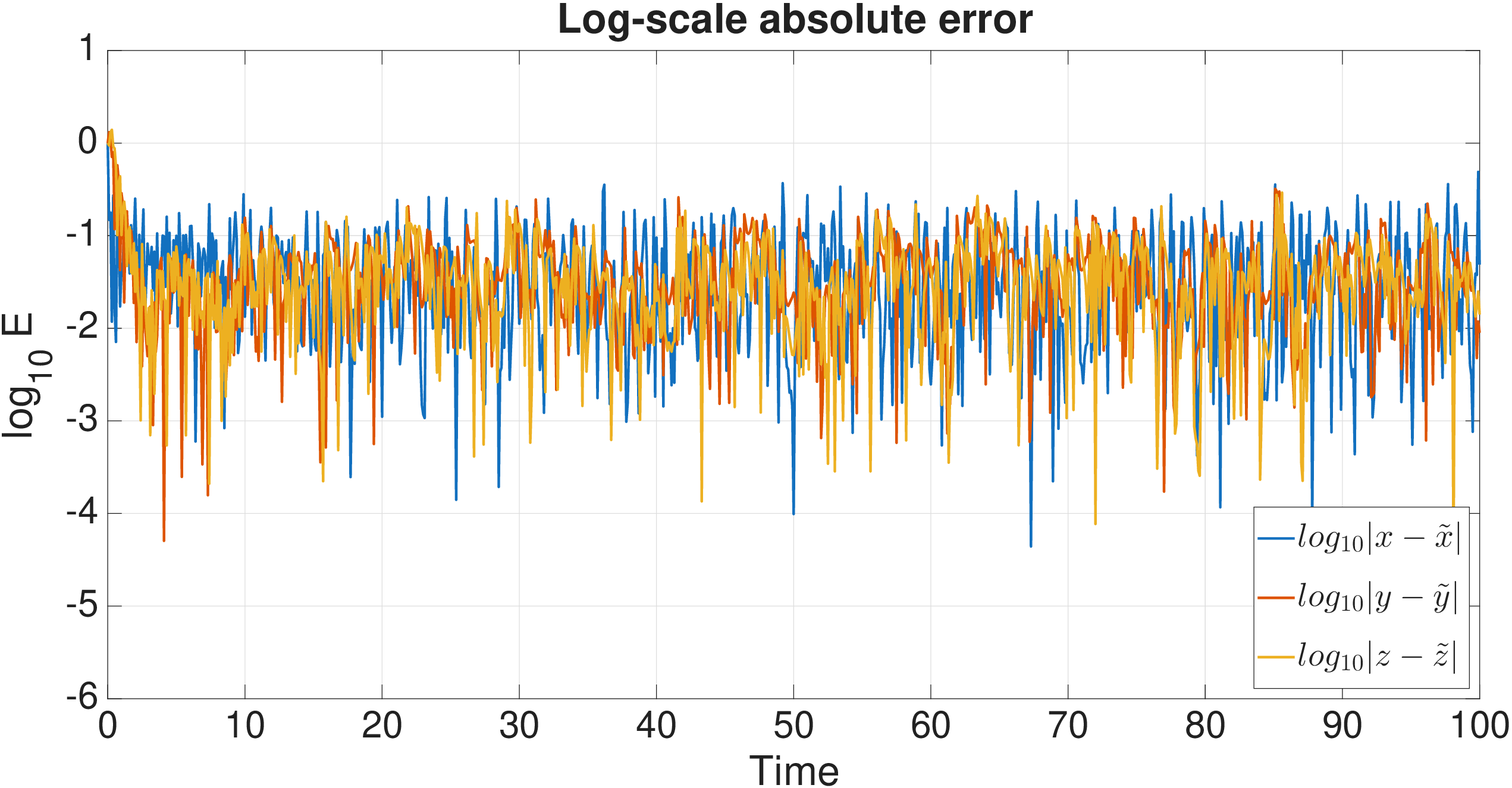} 
         \caption{4\% noise.}
     \end{subfigure}
         \\ \vspace{0.25cm}
         \begin{subfigure}{0.47\textwidth}
         \centering
         \includegraphics[width=\textwidth]{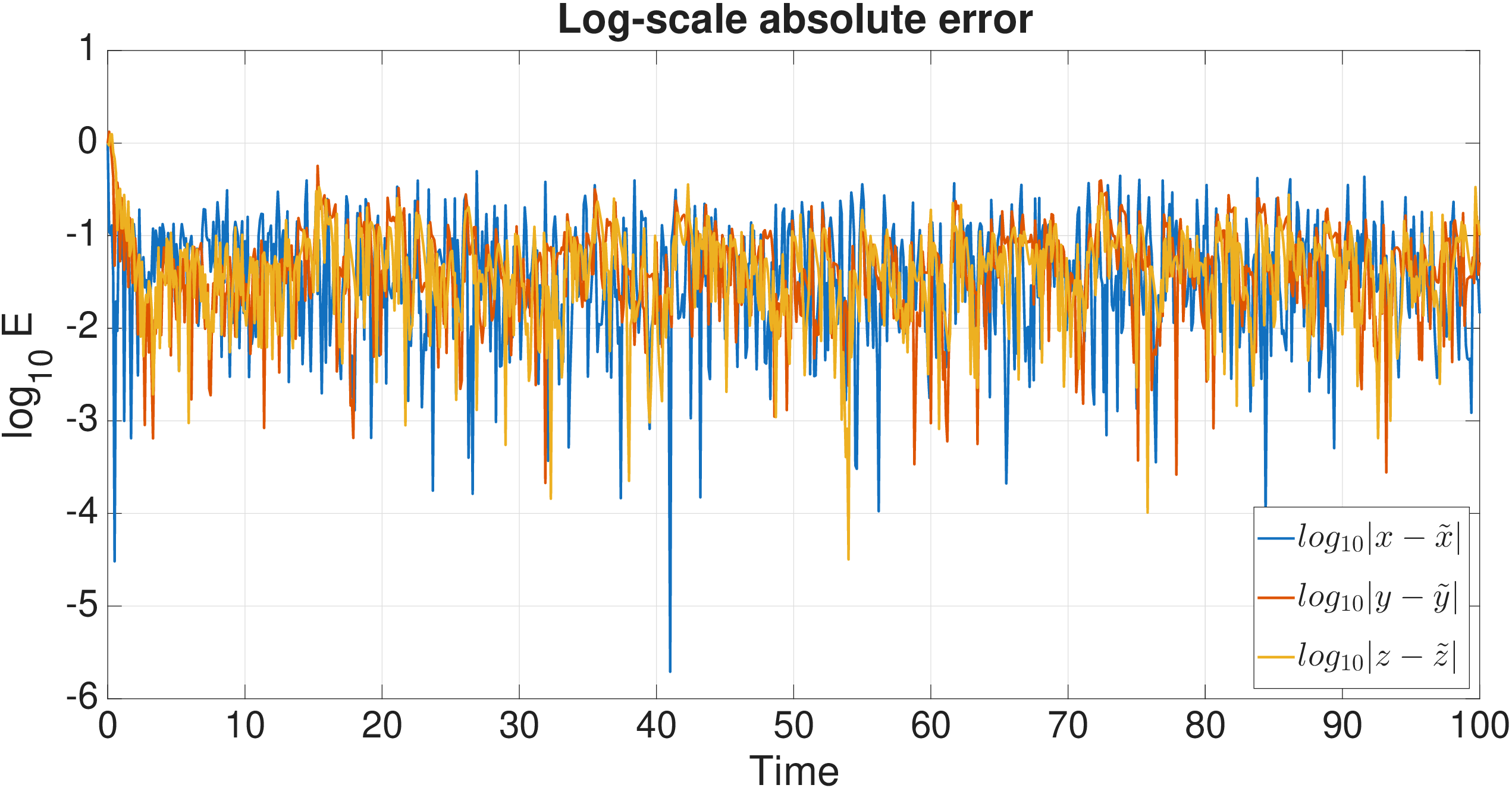} 
         \caption{6\% noise.}
     \end{subfigure}
     \hfill
         \begin{subfigure}{0.47\textwidth}
         \centering
         \includegraphics[width=\textwidth]{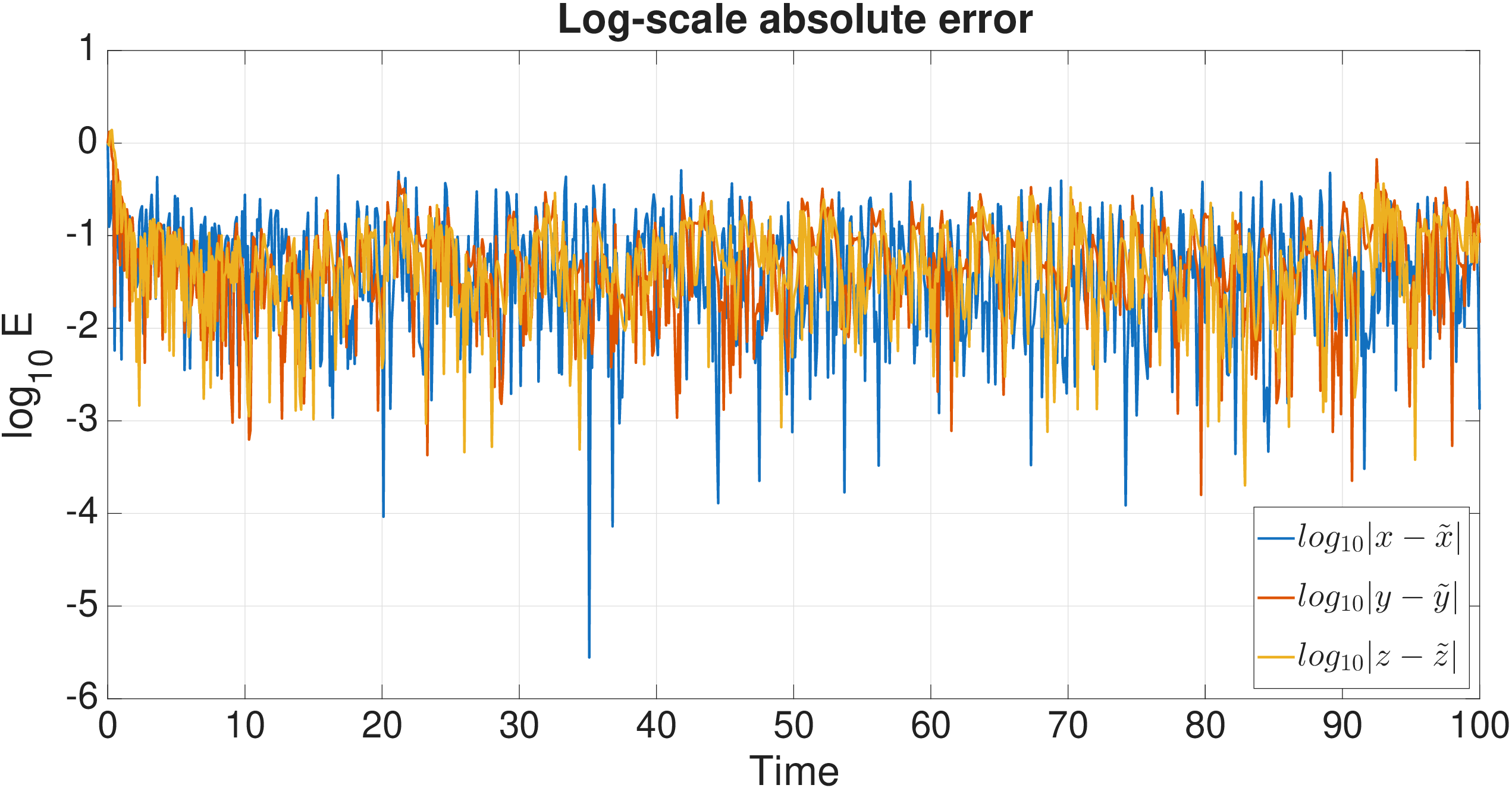} 
         \caption{8\% noise.}
     \end{subfigure}
    \caption{Logarithm (base 10) of the absolute relative errors between the state variables of the Lorenz system and the nudged system in the chaotic regime using noisy $x$-observations. Figures (a)–(d) correspond to different noise levels.}
    \label{f-p-noisy}
\end{figure}

The non-chaotic regime is considered next. Figure~\ref{f-p-noise-free} shows the forward prediction errors for noise-free \(x\)-observations. As in the chaotic regime, the errors decay exponentially for all state variables, even though only the \(x\)-component is observed. Notably, in the non-chaotic regime, the convergence occurs approximately $1.5$ times faster than in the chaotic regime. This behavior is expected, as the non-chaotic regime is characterized by stable dynamics, which reduces sensitivity to initial conditions and facilitates faster synchronization compared to the chaotic regime.

\begin{figure}[H]
    \centering
    \includegraphics[width=0.6\textwidth]{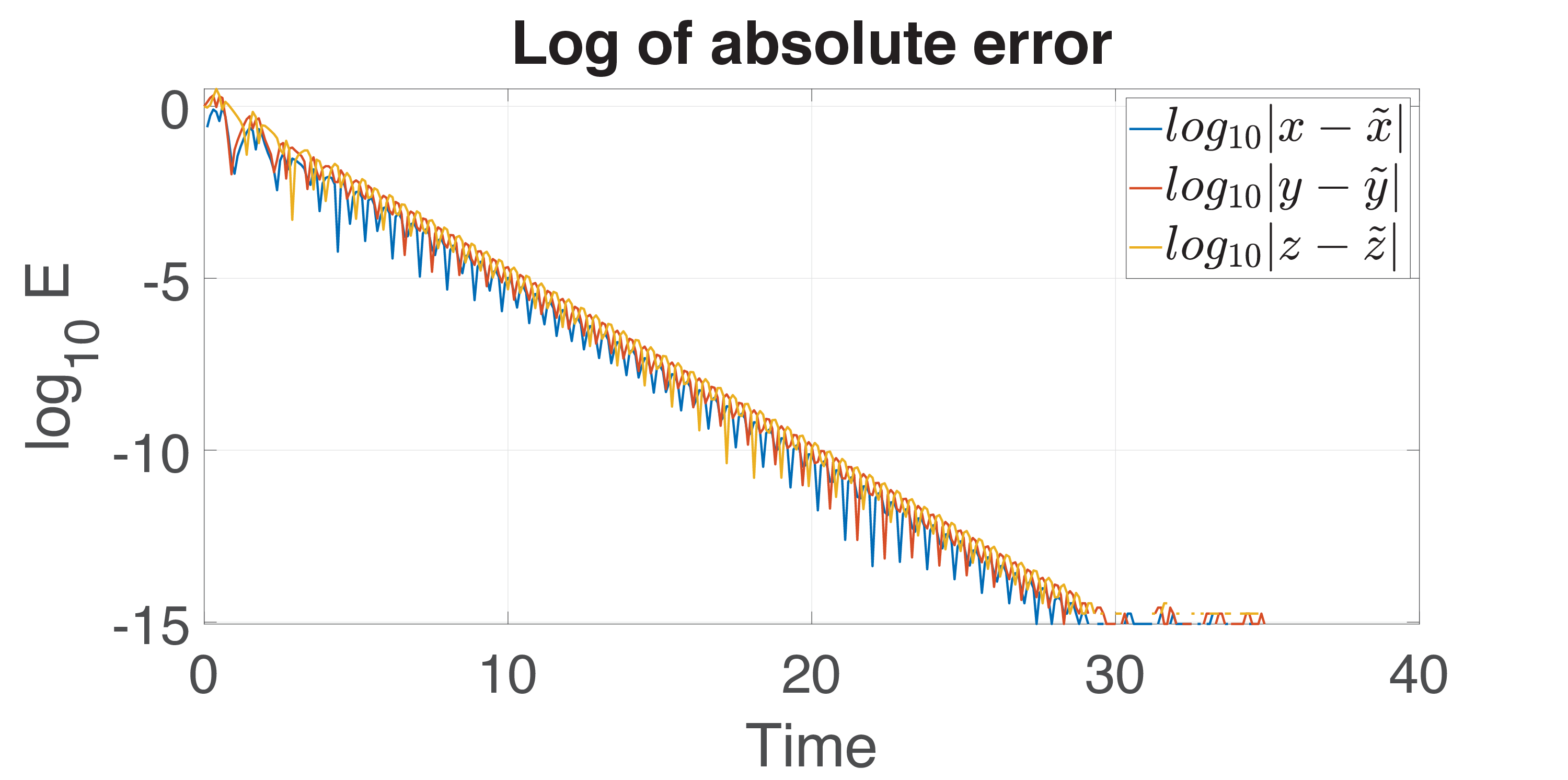}
    \caption{Logarithm (base 10) of the absolute relative errors between the state variables of the Lorenz system and the nudged system in the non-chaotic regime using noise-free $x$-observations.}
    \label{f-p-noise-free}
\end{figure}

We next consider the case of noisy observations in the non-chaotic regime. Figure~\ref{f5} shows that, across different noise levels, the proposed approach maintains an accuracy of approximately $10^{-2}$ in forward prediction.

\begin{figure}[H]
    \centering
    \begin{subfigure}[b]{0.47\textwidth}
         \centering
         \includegraphics[width=\textwidth]{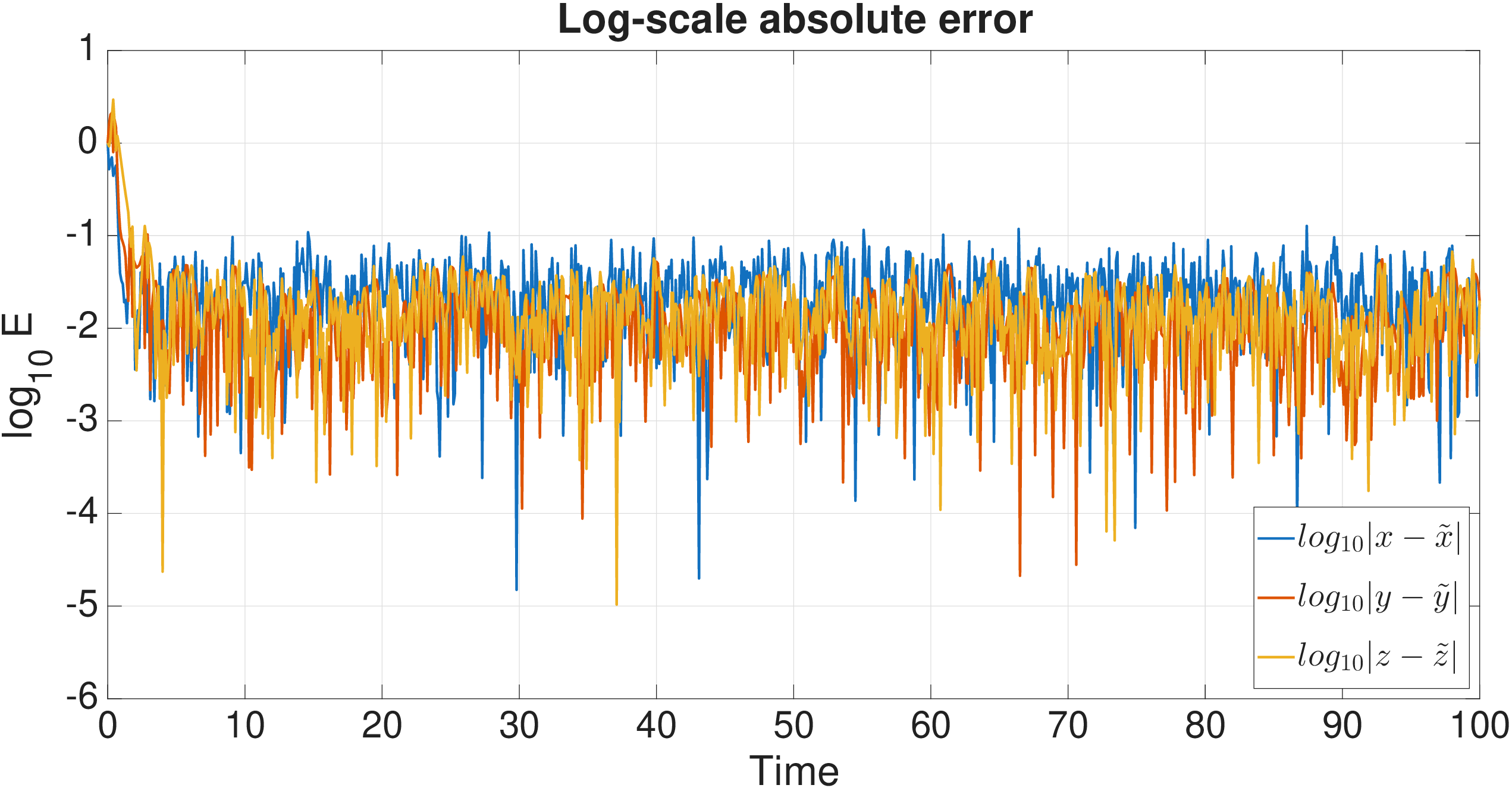} 
         \caption{2\% noise.}
     \end{subfigure}
     \hfill
     \begin{subfigure}[b]{0.47\textwidth}
         \centering
         \includegraphics[width=\textwidth]{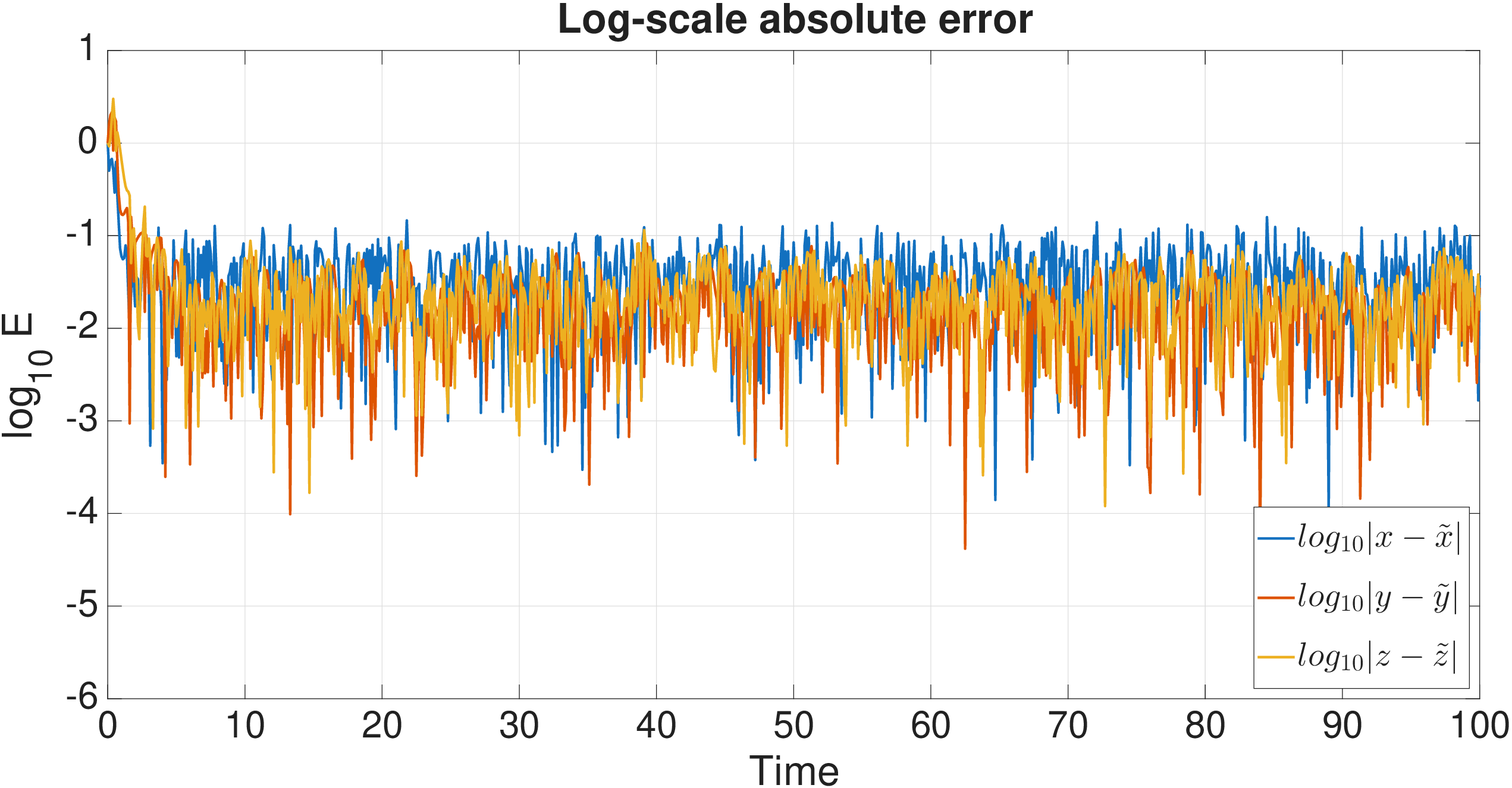} 
         \caption{4\% noise.}
     \end{subfigure}
     \\ \vspace{0.25cm}
     \begin{subfigure}[b]{0.47\textwidth}
         \centering
         \includegraphics[width=\textwidth]{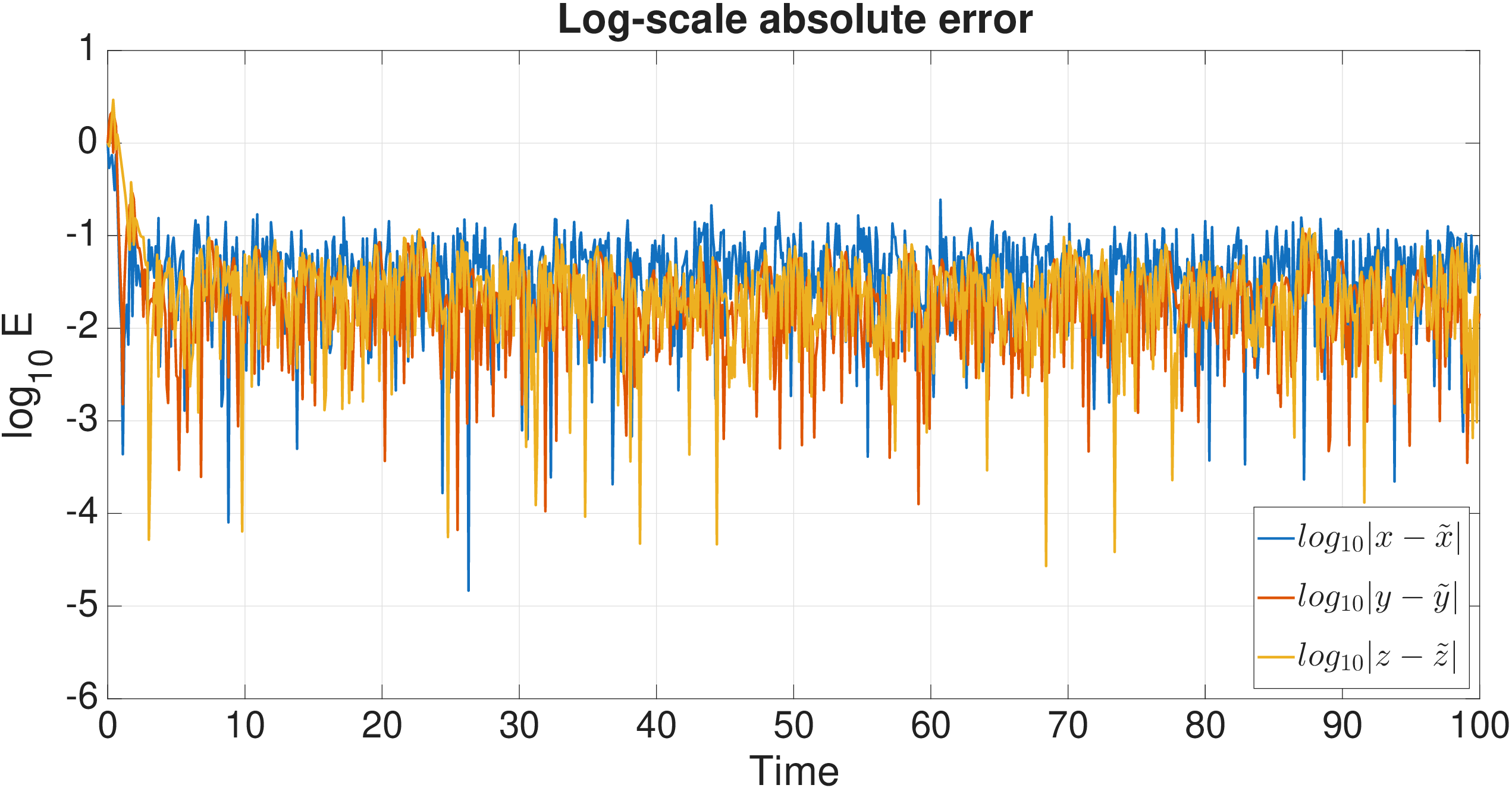} 
         \caption{6\% noise.}
     \end{subfigure}
     \hfill
     \begin{subfigure}[b]{0.47\textwidth}
         \centering
         \includegraphics[width=\textwidth]{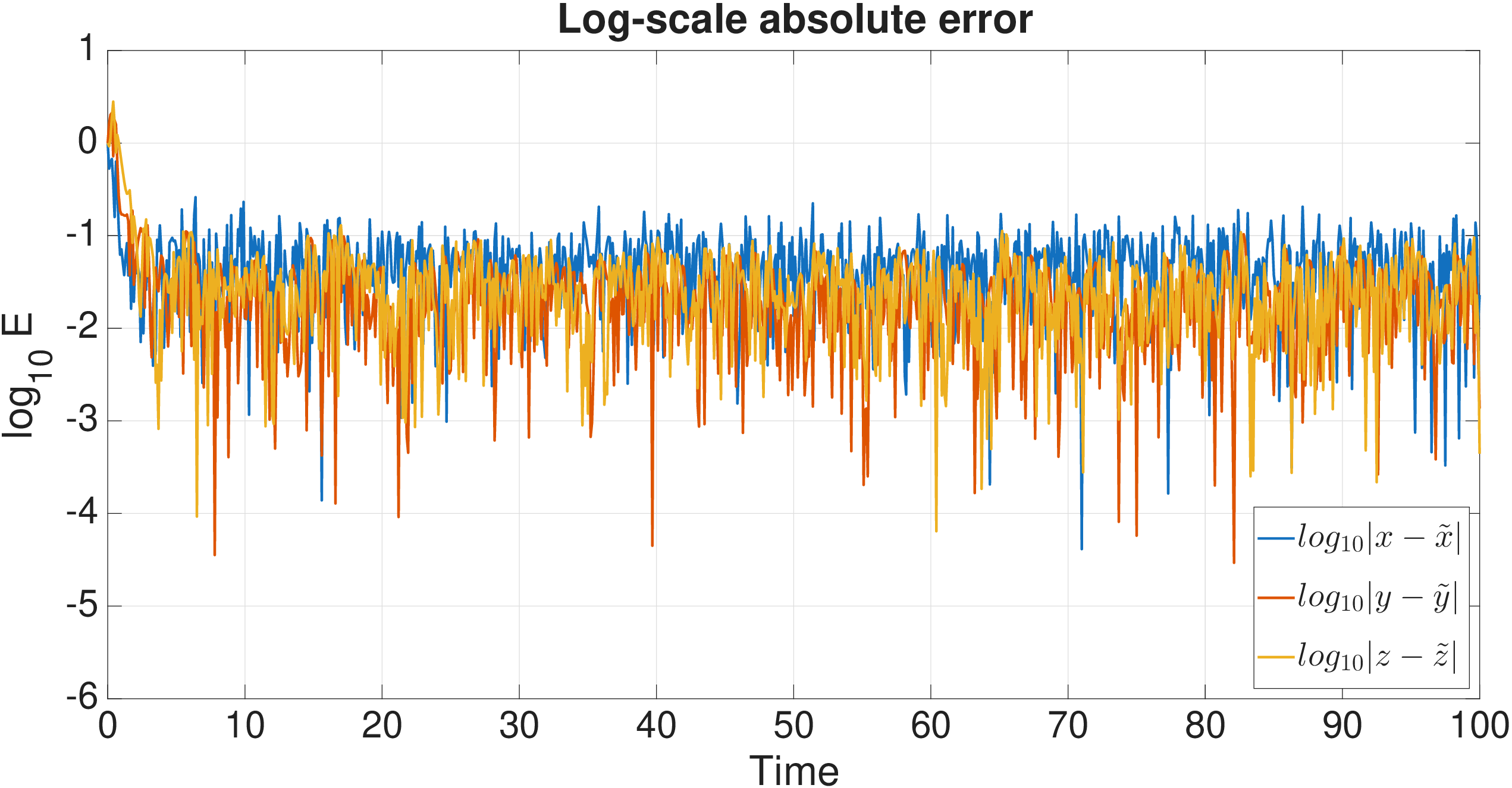} 
         \caption{8\% noise.}
     \end{subfigure}
    \caption{Logarithm (base 10) of the absolute relative errors between the state variables of the Lorenz system and the nudged system in the non-chaotic regime using noisy $x$-observations. Figures $a, b, c, d$ illustrate the results for different noise levels.}
    \label{f5}
\end{figure}

Finally, we compare the effect of observing different state variables. For the Lorenz system, observing either the \(x\)- or \(y\)-component is sufficient to synchronize the nudged system with the reference trajectory and recover the full state. In contrast, observing only the \(z\)-component fails to produce synchronization. This behavior is illustrated in Figure~\ref{figz}.

    \begin{figure}[H]
        \centering
        \includegraphics[width=0.8\textwidth]{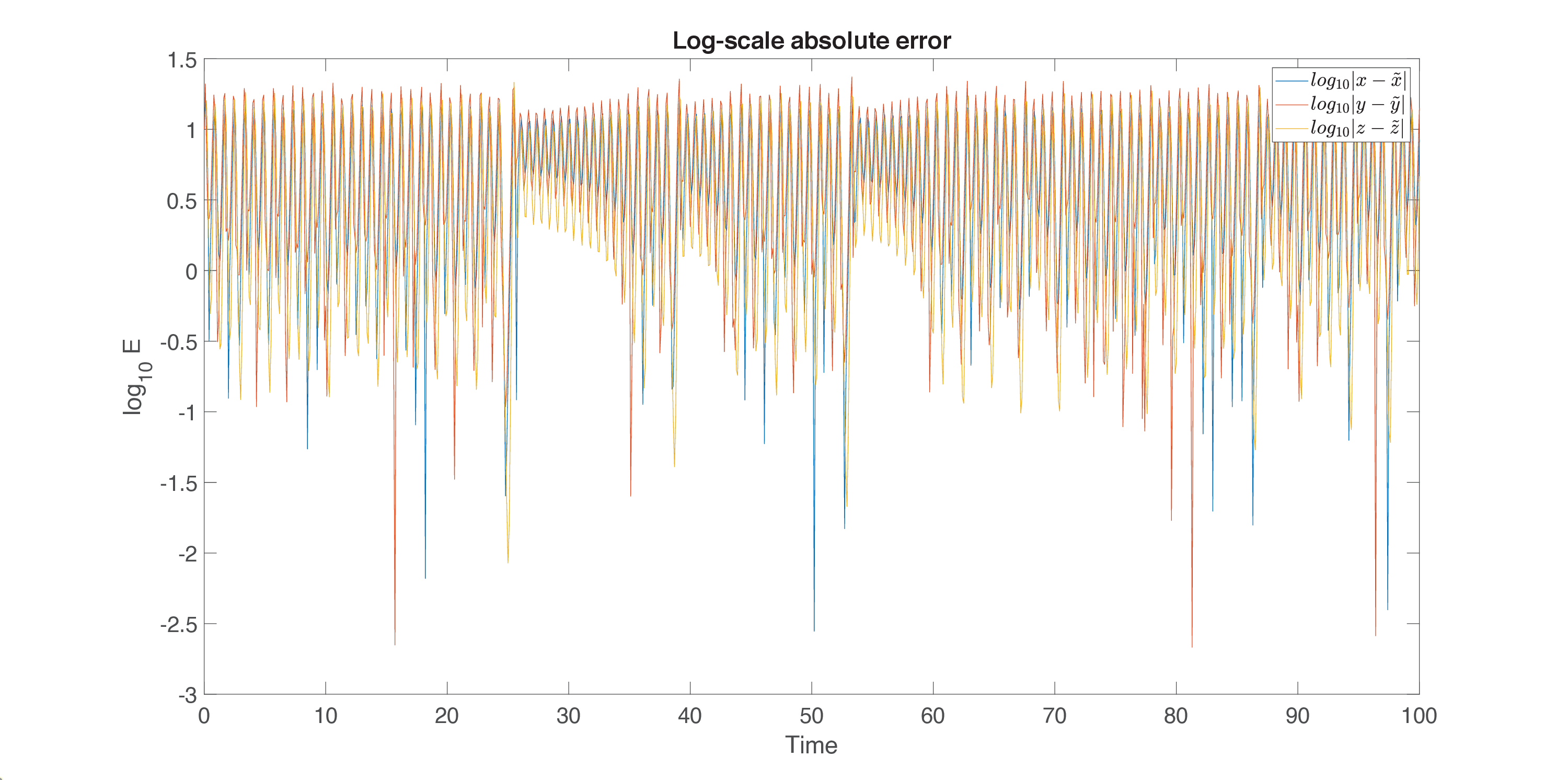}
        \caption{Failure of convergence when only the $z$–component is observed. The nudged system does not synchronize with the reference Lorenz trajectory. The abscissa represents time, and the ordinate represents the logarithm of the absolute error between the true and nudged states.}
            \label{figz}
    \end{figure}
%===========================================================================

\subsection{Parameter Estimation}\label{pe}

Algorithm~\ref{alg2} is now used to estimate the Lorenz parameters from partial observations. Results are reported for both chaotic and non-chaotic regimes and for different levels of observational noise. Nelder--Mead is also compared with Particle Swarm Optimization and Simulated Annealing, and the effect of different initial parameter guesses is examined.
%===========================================================================
\subsubsection{Chaotic Regime}
Parameter estimation in the chaotic regime is considered first. This case is especially of interest because the practical identifiability results from Section~\ref{chaotic-ident} indicate that the parameters are difficult to recover directly from noisy observations.

The initial parameter guesses are set to \(\tilde{\sigma}=12\), \(\tilde{\rho}=35\), and \(\tilde{\beta}=5\). Table~\ref{noise-chaotic} reports the relative percentage errors in the estimated parameters for different noise levels. Across all cases, the estimation error remains smaller than the magnitude of the noise introduced in the observations, indicating that the nudging-based estimation procedure is robust in the chaotic regime.

\begin{table}[H]
\centering
\begin{tabular}{c cc cc cc}
\hline
\multirow{2}{*}{Noise Level $(\%)$}
& \multicolumn{2}{c}{$RE_{\sigma} (\%)$} 
& \multicolumn{2}{c}{$RE_{\rho}(\%)$}
& \multicolumn{2}{c}{$RE_{\beta}(\%)$} \\
\cline{2-7}
& $x-$ observ. & $y-$ observ.
& $x-$ observ. & $y-$ observ.
& $x-$ observ. & $y-$ observ. \\
\hline

$0$  
& 0.11 & 0.09 
& 0.04 & 0.15 
& 0.20 & 0.17 \\

$2$  
& 0.38 & 0.94 
& 0.26 & 0.41 
& 0.56 & 0.22 \\

$4$  
& 0.29 & 1.28 
& 0.31 & 0.88 
& 0.33 & 1.80 \\

$6$  
& 0.65 & 0.68 
& 0.53 & 1.50
& 0.41 & 1.74 \\

$8$  
& 0.97 & 0.19 
& 0.82 & 0.38 
& 0.71 & 0.92 \\

$10$ 
& 1.83 & 1.22 
& 0.93 & 4.40 
& 0.08 & 5.69 \\
\hline
\end{tabular}
\caption{Relative percentage errors in parameter estimation across different noise levels in the chaotic regime using $x$- and $y$-observations.}
\label{noise-chaotic}
\end{table}

The Nelder--Mead method is also compared with Particle Swarm Optimization (PSO) and simulated annealing (SA). These methods are used to minimize the same nudging-based cost functional in Eq.~\eqref{cost}. For background on SA and PSO, we refer to \cite{kirkpatrick1983} and \cite{kennedy1995}, respectively. Table~\ref{tab2} presents the relative percentage errors obtained using the three optimization algorithms. The estimates are similar across the methods. Among the algorithms considered, Nelder--Mead was the fastest in terms of computational time, while simulated annealing was the slowest.

\begin{table}[H]
\centering
\begin{tabular}{cccc}
\hline
Algorithm & {$RE_{\sigma} (\%)$} & {$RE_{\rho} (\%)$} & {$RE_{\beta} (\%)$} \\
\hline
Particle Swarm & 0.0477 & 0.0447 & 0.0541\\
Nelder--Mead & 0.1096 & 0.0435 & 0.0784\\
Simulated Annealing & 0.0965 & 0.0311 & 0.0353\\
\hline
\end{tabular}
\caption{Relative percentage errors for parameter estimation in the chaotic regime of the Lorenz system using different optimization algorithms.}
\label{tab2}
\end{table}

The effect of the initial parameter guesses is examined by selecting ten initial guesses from the boundary of a sphere centered at the true parameter vector \((10,30,8/3)\). This experiment is repeated for several values of the radius \(r\). For each initial guess, the cost functional in Eq.~\eqref{cost} is minimized, and the average relative percentage errors are reported in Table~\ref{tab3}. The results show that the errors remain small across the different radii considered. In all cases, the average relative errors remain below \(0.53\%\), indicating that the method is not highly sensitive to the initial parameter guess in this experiment.

\begin{table}[H]
\centering
\begin{tabular}{cccc}
\hline
Radius & {$RE_{\sigma} (\%)$} & {$RE_{\rho} (\%)$} &{$RE_{\beta} (\%)$}\\
\hline
0.5 & 0.5274 & 0.1367 & 0.3973\\
3   & 0.4620 & 0.0969 & 0.2967\\
5   & 0.4524 & 0.1100 & 0.3203\\
10  & 0.4515 & 0.1078 & 0.3172\\
\hline
\end{tabular}
\caption{Average relative percentage errors for different initial-guess radii in the chaotic regime.}
\label{tab3}
\end{table}

%===========================================================================
 
\subsubsection{Non-Chaotic Regime}

Parameter estimation in the non-chaotic regime is considered next. The initial parameter guesses are set to \(\tilde{\sigma}=12\), \(\tilde{\rho}=16\), and \(\tilde{\beta}=5\). The relative percentage errors for different noise levels are reported in Table~\ref{noise-nonchaotic}. The error in the estimate of \(\sigma^*\) is noticeably larger than the errors for \(\rho^*\) and \(\beta^*\), even at low noise levels. This behavior is consistent with the weak sensitivity of the observed dynamics to \(\sigma\) in the non-chaotic regime. Near equilibrium, where \(x\approx y\), the equation \(\dot{x}=\sigma(y-x)\) provides limited information about \(\sigma\). Thus, accurate recovery of \(\sigma\) is less reliable in this regime. Nevertheless, accurate estimation of \(\sigma\) may not be essential for accurate prediction, as it is not a sensitive parameter in the non-chaotic regime.

\begin{table}[H]
\centering
\begin{tabular}{c cc cc cc}
\hline
\multirow{2}{*}{Noise Level $(\%)$}
& \multicolumn{2}{c}{$RE_{\sigma} (\%)$} 
& \multicolumn{2}{c}{$RE_{\rho}(\%)$}
& \multicolumn{2}{c}{$RE_{\beta}(\%)$} \\
\cline{2-7}
& $x-$ observ. & $y-$ observ.
& $x-$ observ. & $y-$ observ.
& $x-$ observ. & $y-$ observ. \\
\hline

$0$  
& 1.07 &  3.55
& 0.61 &  1.28
& 0.76 &  1.41 \\

$2$  
& 3.93 & 16.3 
& 0.14 &  0.10
& 0.64 &  0.06\\

$4$  
& 12.6 &  16.1
& 0.73 &  2.29
& 1.52 &  3.88\\

$6$  
& 4.85 &  8.70
& 1.23 & 2.38
& 1.41 &  3.59\\

$8$  
& 15.5 & 18.28 
& 1.44 &  1.92
& 2.32 &  3.20\\

$10$ 
& 19.4 &  21.36
& 3.07 &  4.21
& 3.71 &  5.24\\
\hline
\end{tabular}
\caption{Relative percentage errors in parameter estimation across different noise levels in the non-chaotic regime using $x$- and $y$-observations.}
\label{noise-nonchaotic}
\end{table}

The three optimization algorithms are also compared in the non-chaotic regime, and the effect of randomly chosen initial parameter guesses is examined. The results are presented in Tables~\ref{tab9} and \ref{tab6}. From Table~\ref{tab9}, PSO gives the most accurate estimates among the three methods, while simulated annealing gives the least accurate estimates in this case.

\begin{table}[H]
\centering
\begin{tabular}{cccc}
\hline
Algorithm & \(RE_{\sigma}\) (\%) & \(RE_{\rho}\) (\%) & \(RE_{\beta}\) (\%)\\
\hline
Particle Swarm & 0.8817 & 0.4367 & 0.4560\\
Nelder--Mead & 1.0703 & 0.6077 & 0.6407\\
Simulated Annealing & 1.5522 & 1.8644 & 1.8677\\
\hline
\end{tabular}
\caption{Relative percentage errors for parameter estimation in the non-chaotic regime of the Lorenz system using different optimization algorithms.}
\label{tab9}
\end{table}

\begin{table}[H]
\centering
\begin{tabular}{cccc}
\hline
Radius & \(RE_{\sigma}\) (\%) & \(RE_{\rho}\) (\%) & \(RE_{\beta}\) (\%)\\
\hline
0.5 & 0.8923 & 0.4314 & 0.4503\\
3   & 0.8786 & 0.4260 & 0.4435\\
5   & 0.8874 & 0.4301 & 0.4486\\
10  & 0.8835 & 0.4318 & 0.4516\\
\hline
\end{tabular}
\caption{Average relative percentage errors for different initial-guess radii in the non-chaotic regime.}
\label{tab6}
\end{table}

%========================  Section 6 Discussion  =========================

\section{Discussion}
\label{discuss}

We have presented a data-assimilation-based optimization framework for parameter estimation in nonlinear dynamical systems, using the Lorenz--63 system as a test case. The theoretical analysis establishes the main mechanisms that support the proposed method. Theorem~\ref{thm2} shows that, when the parameters in the nudged system match the true parameters, the nudged solution synchronizes with the true trajectory from arbitrary initial conditions. The same result also shows that, under parameter mismatch, the nudged solution remains close to the true trajectory, with an asymptotic error controlled by the size of the parameter mismatch. This justifies the use of the nudged-system cost functional for parameter estimation and explains why the method can reduce sensitivity to inaccurate or unknown initial conditions, which is especially important in chaotic regimes.

The theoretical results also clarify why partial observations can be sufficient for parameter recovery. In the Lorenz--63 system, the \(x\)-component has a determining property, so the full state can be recovered through nudging even when the full state is not observed. The well-posedness results for full and partial observations identify conditions under which the data-to-parameter inverse map is well-conditioned, while also showing that parameter recovery can fail or become ill-conditioned on degeneracy sets, such as equilibrium or otherwise non-informative trajectories. Proposition~\ref{prop:quant_ident_shifted} and Theorem~\ref{thrm7} further connect the cost functional to parameter recovery by showing that, under suitable nondegeneracy assumptions, the cost functional controls the parameter error. Finally, Proposition~\ref{prop:partial_derivative_equations} makes this nondegeneracy condition computable through the sensitivity equations and the associated Gram matrix \(G_{\tau,T}\). Thus, the theoretical analysis links synchronization, partial observation, well-posedness, and local identifiability to the practical optimization procedure used in the numerical experiments.

The computational efficiency of the proposed framework provides another important advantage. All numerical simulations were performed on an Intel Core i7 2.80 GHz processor with 16GB of RAM, and each experiment required less than one minute of runtime, with slightly higher runtimes observed only in noisy chaotic cases. An important practical advantage of the proposed framework is that it does not require observations from every equation in which unknown parameters appear. This distinguishes it from existing on-the-fly parameter estimation approaches such as \cite{Carlson2022, Martinez2024, NEWEY2025114121}, which typically require full state observations to recover all unknown parameters. In many real-world applications, such complete observations may not be available. The proposed framework instead uses only the observed component to drive the nudged system and constructs the parameter-estimation cost functional from the corresponding observation mismatch.

To further examine this issue, we implemented the method proposed in \cite{Carlson2022} in both chaotic and non-chaotic regimes, using noise-free observations and multiplicative noisy observations of all state variables. The noisy data were generated according to the multiplicative noise model in Eq.~\eqref{eq4}. We use this noise model for consistency with the rest of the present work, noting that Carlson et al.~\cite{Carlson2022} instead consider additive noise with a very small noise amplitude. Table~\ref{tab:carlson_mre} shows that the method performs very well in the chaotic noise-free case, with MREs below \(1\%\) for all three parameters. As the noise level increases, however, the parameter estimates become more sensitive to observational perturbations, especially for higher noise levels.

In comparison, this proposed framework uses only \(x\)-observations to estimate all three parameters and maintains small relative errors across the same noise levels while ~\cite{Carlson2022} necessitates observing all the components.  In the chaotic regime, the relative errors remain below \(2\%\) for all parameters, while in the non-chaotic regime the errors remain substantially smaller in the noisier cases. These results suggest that, for the test cases considered here, the proposed nudging-augmented optimization framework remains accurate under measurement noise while requiring only partial observations.

\begin{table}[H]
\centering
\begin{tabular}{c|ccc|ccc}
\hline
\multirow{2}{*}{Noise Level $(\%)$}
& \multicolumn{3}{c|}{Chaotic}
& \multicolumn{3}{c}{Non-chaotic} \\
\cline{2-7}
& $\text{MRE}_{\sigma}(\%)$
& $\text{MRE}_{\rho}(\%)$
& $\text{MRE}_{\beta}(\%)$
& $\text{MRE}_{\sigma}(\%)$
& $\text{MRE}_{\rho}(\%)$
& $\text{MRE}_{\beta}(\%)$ \\
\hline

$0$
& 0.38
& 0.05
& 0.64
& 0.50
& 0.13
& 0.68 \\

$2$
& 15.11
& 1.19
& 18.66
& 22.49
& 0.84
& 12.99 \\

$4$
& 33.37
& 6.16
& $1.01 \times 10^{3}$
& 22.38
& 0.46
& 34.75 \\

$6$
& 21.83
& 4.65
& 81.93
& 362.21
& 7.06
& 28.65 \\

$8$
& 22.61
& 11.03
& 102.71
& 26.52
& 7.37
& 123.74 \\

$10$
& $4.7 \times 10^{4}$
& $1.4 \times 10^{4}$
& 83
& $1.3 \times 10^{3}$
& 448.04
& 65.38 \\
\hline
\end{tabular}
\caption{Mean relative error percentages for the method proposed in \cite{Carlson2022} in the chaotic and non-chaotic Lorenz--63 regimes under different observational noise levels.}
\label{tab:carlson_mre}
\end{table}

\vspace{-0.3cm}

 We further compare the proposed framework with Bayesian approaches, which are widely used for parameter estimation in dynamical systems \cite{BoerschSupan2017,Ghasemi2011,MCMCLorenz,huang2006hierarchical}. To this end, we conducted a Bayesian MCMC-based parameter estimation experiment using \(x\)-observations for both chaotic and non-chaotic regimes under noise-free, \(2\%\), and \(8\%\) noise settings. The resulting posterior density estimates are shown in Figures~\ref{fig:mcmc_chaotic} and~\ref{fig:mcmc_nonchaotic}.

 In the chaotic regime, Figure~\ref{fig:mcmc_chaotic} shows that the posterior distributions can have a noticeable spread, and in some cases the densities exhibit more than one mode. This behavior indicates that the available \(x\)-observations may support multiple plausible parameter values under the Bayesian formulation, particularly as the observational noise level increases. A similar effect is observed in the non-chaotic regime in Figure~\ref{fig:mcmc_nonchaotic}, where several posterior densities remain broad and some display multimodal structure. These features reflect increased uncertainty in the estimated parameters and show that the posterior mean alone may not fully represent the structure of the parameter uncertainty.

 Since Bayesian MCMC methods require repeated forward simulations in order to explore the posterior distribution, they can become computationally expensive for nonlinear dynamical systems. In contrast, the proposed nudging-based optimization framework directly computes point estimates by minimizing the observation mismatch. Using only \(x\)-observations, the proposed method produces accurate estimates across both regimes and noise levels, as shown in Tables~\ref{noise-chaotic} and~\ref{noise-nonchaotic}. Thus, in the experiments considered here, the proposed framework provides accurate point estimates from partial observations, while the Bayesian posterior densities illustrate the uncertainty, spread, and possible multimodality associated with the corresponding MCMC-based inference problem.

\begin{figure}[H]
    \centering
    
    \begin{subfigure}[b]{0.45\textwidth}
        \centering
        \includegraphics[width=\textwidth]{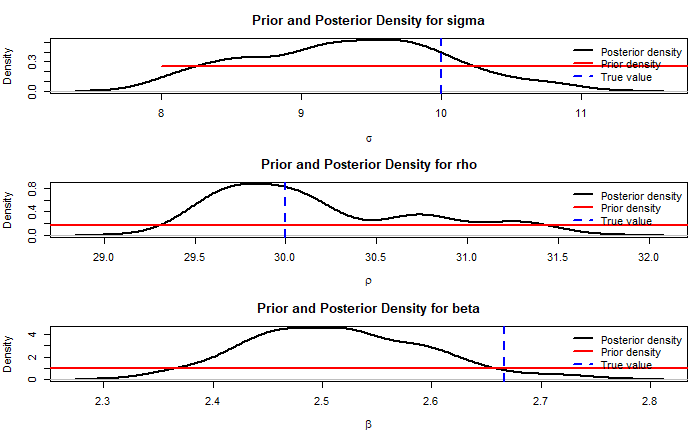}
        \caption{Noise-free}
    \end{subfigure}
    \hfill
    \begin{subfigure}[b]{0.45\textwidth}
        \centering
        \includegraphics[width=\textwidth]{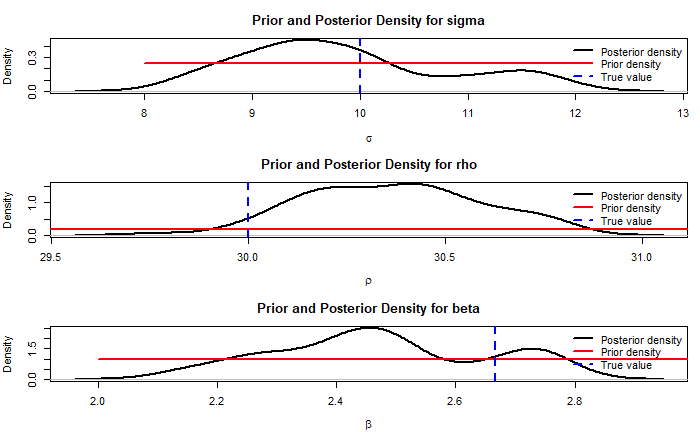}
        \caption{$2\%$ noise}
    \end{subfigure}
    
\vspace{-0.2cm}
    
    \begin{subfigure}[b]{0.45\textwidth}
        \centering
        \includegraphics[width=\textwidth]{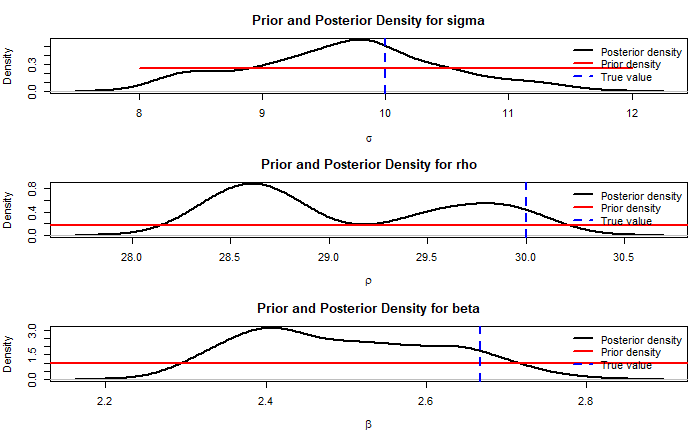}
        \caption{$8\%$ noise}
    \end{subfigure}
    \vspace{-0.2cm}
    \caption{Posterior density estimates obtained using Bayesian MCMC parameter estimation in the chaotic regime with $x$ observations under varying noise levels. The dashed blue lines represent the true parameter values.}
    \label{fig:mcmc_chaotic}
\end{figure}

\begin{figure}[H]
    \centering
    
    \begin{subfigure}[b]{0.48\textwidth}
        \centering
        \includegraphics[width=\textwidth]{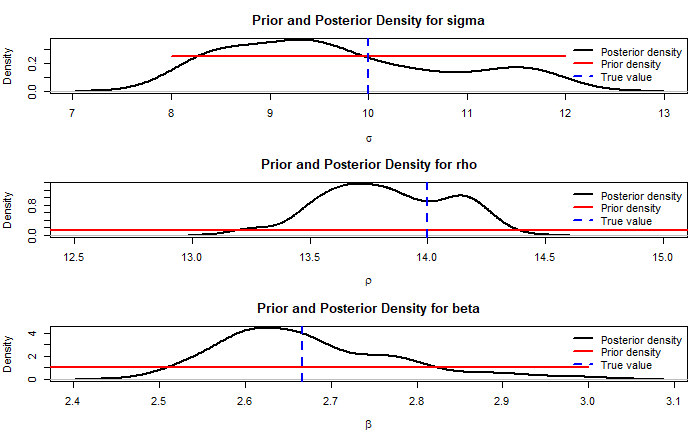}
        \caption{Noise-free}
    \end{subfigure}
    \hfill
    \begin{subfigure}[b]{0.48\textwidth}
        \centering
        \includegraphics[width=\textwidth]{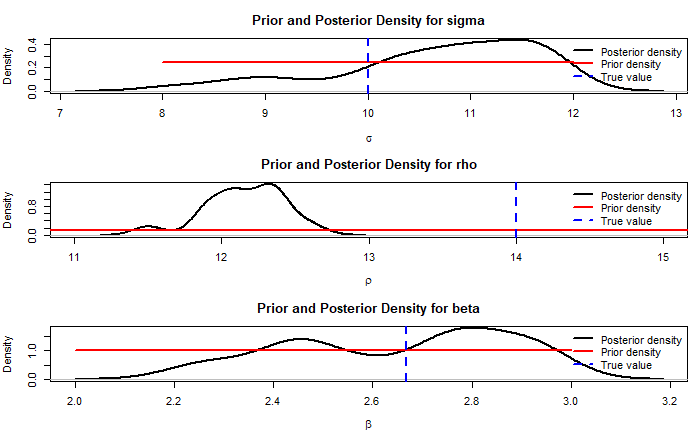}
        \caption{$2\%$ noise}
    \end{subfigure}

    \begin{subfigure}[b]{0.48\textwidth}
        \centering
        \includegraphics[width=\textwidth]{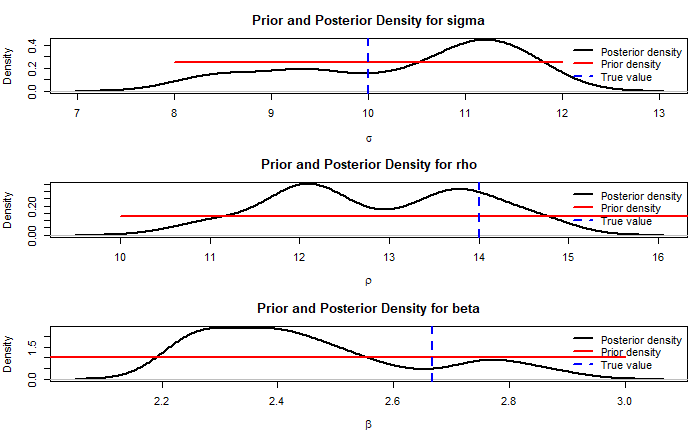}
        \caption{$8\%$ noise}
    \end{subfigure}
    \caption{Posterior density estimates obtained using Bayesian MCMC parameter estimation in the non-chaotic regime with $x$ observations under varying noise levels. The dashed blue lines represent the true parameter values.}
    \label{fig:mcmc_nonchaotic}
\end{figure}

 \noindent 
{\bf Future directions:} There are several promising directions for future research. Although the Lorenz--63 parameters considered here enter the equations through the linear term, the proposed framework itself does not require this condition. A natural next step is to test the method on ODE systems in which unknown parameters occur in the nonlinear term  and to develop corresponding theoretical framework for well-posedness of the inverse map. Another important direction involves models with time-varying parameters, which arise frequently in practical applications and would require additional methodological development. Finally, although the present Lorenz--63 experiments were computationally efficient, the proposed optimization procedure requires repeated evaluations of the cost functional, and each such evaluation involves solving the nudged dynamical system. For larger or higher-dimensional systems, this repeated ODE-solving step may increase the computational cost. One possible way to reduce this cost is to replace repeated forward solves by  learned surrogate models, such as physics-informed neural networks, neural ODEs, or related scientific machine-learning approaches \cite{Raissi2019PINNs,Chen2018NeuralODE,Rackauckas2020UDE}. In addition to this repeated-solve cost, larger parameter spaces may introduce further computational challenges for optimization, especially when a broad search over admissible parameter values is required. One may explore quantum optimization and hybrid quantum--classical optimization methods to address this issue, since quantum optimization algorithms have shown promise for optimization problems \cite{Farhi2014QAOA,Blekos2024QAOA}. Initial work in this direction is being pursued in our manuscript, \emph{A Data-Assimilation-Informed Quantum Optimization Framework for Parameter Estimation  in Dynamical Systems} \cite{Ahmad2026QuantumUnderPrep}.

\section*{Credit authorship contribution statement}

\textbf{Muhammad Jalil Ahmad:} Writing -- review \& editing, Writing -- original draft, Visualization, Validation, Software, Methodology, Investigation, Formal analysis; \textbf{Animikh Biswas:} Writing -- review \& editing,  Writing -- original draft, Supervision, Resources, Project administration, Methodology, Funding acquisition,  Investigation, Formal analysis, Conceptualization; \textbf{Kathleen Hoffman:} Writing -- review \& editing, Supervision, Resources, Project administration, Methodology, Investigation, Formal analysis, Conceptualization.

\section*{Acknowledgments}

The authors gratefully acknowledge support from the UMBC Strategic Awards for Research Transitions (START) Program. A.~Biswas was also partially supported by NSF grant DMS-2529382. The authors also thank Dr.~Katherine Gurski and Boris Alemi for their helpful assistance with the computational and numerical aspects of this work.

\appendix

\section{Technical Proofs}
\label{appendix:technical}
%------------------------ Non-constancy of the \(z\)-component away from equilibria ------- %
\begin{lemma}[Non-constancy of the \(z\)-component away from equilibria] \label{lemma1}
Let $(x(t), y(t), z(t))$ be the solution of the Lorenz system
\[
\begin{cases}
\dot{x}=\sigma^*(y-x),\\[2pt]
\dot{y}=x(\rho^*-z)-y,\\[2pt]
\dot{z}=xy-\beta^* z,
\end{cases}
\]
with true parameters $\sigma^*>0$, $\beta^*>0$, and $\rho^*\neq 0$.  
If the trajectory is not at an equilibrium, then $z(t)$ cannot be constant on any open time interval.
\end{lemma}

\begin{proof}
Assume, for contradiction, that $z(t)\equiv\bar{z}$ is constant on some open interval $I$. From the third equation, $\dot{z} = 0$ implies
\[
xy - \beta^*\bar{z} = 0, \quad\text{i.e.,}\quad xy \equiv c := \beta^*\bar{z}.
\]
\textbf{Case 1: $c\neq 0$.}  
Since $x$ and $y$ never vanish on $I$, we may write $y = c/x$.  
From the first Lorenz equation:
\begin{equation*}
    \dot{x} = \sigma^*\!\left(\frac{c}{x} - x\right) = \frac{\sigma^*(c - x^{2})}{x}.
\end{equation*}
Differentiating $y = c/x$ gives
\begin{equation}\label{6}
    \dot{y} = -\frac{c}{x^{2}} \dot{x} = -\frac{\sigma^* c(c - x^{2})}{x^{3}}.
\end{equation}
From the second Lorenz equation:
\begin{equation}
\label{7}
    \dot{y} = x(\rho - \bar{z}) - \frac{c}{x}.
\end{equation}
Equating Eq.~\eqref{6} and \eqref{7} gives:
\begin{equation}
    -\frac{\sigma^* c(c - x^{2})}{x^{3}} = x(\rho^* - \bar{z}) - \frac{c}{x}. \label{8}
\end{equation}
Multiplying Eq.~\eqref{8} by $x^{3}$ yields:
\[
-\,\sigma^* c(c - x^{2}) = (\rho^* - \bar{z})x^{4} - c x^{2}.
\]
Expanding and rearranging gives:
\[
0 = (\rho^* - \bar{z})x^{4} - c(1+\sigma^*)x^{2} + \sigma^* c^{2}.
\]
Let $u = x^{2} > 0$. Then
\[
(\rho^* - \bar{z})u^{2} - c(1+\sigma^*)u + \sigma^* c^{2} = 0.
\]
This is a constant-coefficient quadratic equation in $u(t)$ that holds for all $t\in I$. Hence $u(t)$ is constant, so $x(t)$ is constant (up to sign), and $y(t) = c/x(t)$ is constant as well. With $x, y, z$ constant, the first Lorenz equation gives $y = x$. The second Lorenz equation then yields $0 = x(\rho^* - \bar{z}) - x$, so $\bar{z} = \rho^* - 1$. From $xy = c = \beta^*\bar{z}$ and $y = x$, we get $x^{2} = \beta^*(\rho^* - 1)$. Thus,
\[
(x,y,z) = \left( \pm\sqrt{\beta^*(\rho^* - 1)},\ \pm\sqrt{\beta^*(\rho^* - 1)},\ \rho^* - 1 \right),
\]
which are exactly the nontrivial Lorenz equilibria, contradicting the assumption.

\noindent \textbf{Case 2: $c = 0$.}  
Then $\beta^*\bar{z} = 0$, so $\bar{z} = 0$ and $xy \equiv 0$.  Since $x$ and $y$ are real analytic functions of time, the identity $xy\equiv 0$ on $I$ implies that either $x\equiv 0$ or $y\equiv 0$ on $I$. If $x\equiv 0$, the first equation gives $y\equiv 0$, yielding the origin (equilibrium).  
If $y\equiv 0$, the second equation gives $x\rho^* = 0$, hence $x\equiv 0$, again the origin.  
Thus the solution is at equilibrium. This contradicts the assumption. 
\end{proof}

%------------------------ Well-posedness and \(C^1\) dependence on \(\tilde{\theta}\) ------- %

\begin{lemma}[Well-posedness and \(C^1\) dependence of the nudged system on \(\tilde{\theta}\)]
\label{lem:wp_c1_fulltheta}
Let $\tilde{\theta}= (\tilde{\sigma},\tilde{\rho},\tilde{\beta})$
and let $\tilde X(t;\tilde{\theta})=(\tilde{x}(t;\tilde{\theta}),
 \tilde{y}(t;\tilde{\theta}),
 \tilde{z}(t;\tilde{\theta}))$
denote the solution of the nudged system \eqref{nudgedsystem}. Fix
\(\tau\geq 0\) and \(T>0\), and set $I_{\tau,T}:=[0,\tau+T].$
Assume that the observed trajectory \(x(t)\) is continuous on \(I_{\tau,T}\),
and that the initial condition for the nudged system is independent of
\(\tilde{\theta}\). Let \(U\) be a compact neighborhood of $\theta^*=(\sigma^*,\rho^*,\beta^*)$ such that $\tilde{\sigma}>0, \tilde{\beta}>0,\tilde{\theta}\in U.$
Assume also that \(\mu\) satisfies the condition
\eqref{eq:mu_condition_mismatch} for all \(\tilde{\theta}\in U\). Then:
\begin{enumerate}
\item[(i)] for each \(\tilde{\theta}\in U\), the nudged system
\eqref{nudgedsystem} admits a unique solution on \(I_{\tau,T}\);
\item[(ii)] the solution depends continuously on the initial data and on
\(\tilde{\theta}\);
\item[(iii)] the solution map $\tilde{\theta}\mapsto \tilde X(\cdot;\tilde{\theta})$
is \(C^1\) from \(U\) into \(C(I_{\tau,T};\mathbb{R}^3)\).
In particular, the partial derivatives
\[
(p_\sigma,q_\sigma,r_\sigma)
:=
\partial_{\tilde{\sigma}}
(\tilde{x},\tilde{y},\tilde{z}),
\]
\[
(p_\rho,q_\rho,r_\rho)
:=
\partial_{\tilde{\rho}}
(\tilde{x},\tilde{y},\tilde{z}),
\qquad
(p_\beta,q_\beta,r_\beta)
:=
\partial_{\tilde{\beta}}
(\tilde{x},\tilde{y},\tilde{z})
\]
exist and are continuous on \(I_{\tau,T}\).
\end{enumerate}
\end{lemma}

\begin{proof}
Set $\tilde X=(\tilde{x},\tilde{y},\tilde{z})^\top.$ The nudged system \eqref{nudgedsystem} can be written as the non-autonomous ordinary differential equation $\dot{\tilde X}=f(t,\tilde X,\tilde{\theta}),$ where
\[
f(t,\tilde X,\tilde{\theta})
=
\begin{pmatrix}
\tilde{\sigma}(\tilde{y}-\tilde{x})+\mu(x(t)-\tilde{x})\\[2pt]
\tilde{x}(\tilde{\rho}-\tilde{z})-\tilde{y}\\[2pt]
\tilde{x}\tilde{y}-\tilde{\beta}\tilde{z}
\end{pmatrix}.
\]
Since \(x(t)\) is continuous on \(I_{\tau,T}\), the map
\(t\mapsto f(t,\tilde X,\tilde{\theta})\) is continuous. Moreover, \(f\) is
polynomial in \(\tilde X\) and smooth in
\(\tilde{\theta}=(\tilde{\sigma},\tilde{\rho},\tilde{\beta})\). Hence, for each
fixed \(\tilde{\theta}\in U\), \(f\) is locally Lipschitz in \(\tilde X\). By the
Picard--Lindel\"of theorem, the system has a unique maximal solution on some
interval \([0,T_{\max})\).

It remains to show that \(T_{\max}\geq \tau+T\). Since \(x(t)\) is continuous on
the compact interval \(I_{\tau,T}\), there exists \(K>0\) such that $|x(t)|\leq K, t\in I_{\tau,T}.$ Define $V(t)
=
\frac12\left(\tilde{x}^2+\tilde{y}^2+\tilde{z}^2\right).$ Differentiating \(V\) along solutions of \eqref{nudgedsystem} gives
\[
\begin{aligned}
\dot V
&=
\tilde{x}\dot{\tilde{x}}
+
\tilde{y}\dot{\tilde{y}}
+
\tilde{z}\dot{\tilde{z}}  \\
&=
-(\tilde{\sigma}+\mu)\tilde{x}^2
-\tilde{y}^2
-\tilde{\beta}\tilde{z}^2
+
(\tilde{\sigma}+\tilde{\rho})\tilde{x}\tilde{y}
+
\mu x(t)\tilde{x}.
\end{aligned}
\]
By Young's inequality,
\[
|(\tilde{\sigma}+\tilde{\rho})\tilde{x}\tilde{y}|
\leq
\bigl(|\tilde{\sigma}+\tilde{\rho}|+K\bigr)^2\tilde{x}^2
+\frac14\tilde{y}^2,
\]
and
\[
\mu |x(t)\tilde{x}|
\leq
\frac{K^2}{\tilde{\beta}}\tilde{x}^2
+
\frac{\mu^2\tilde{\beta}}{4} .
\]
Therefore,
\[
\begin{aligned}
\dot V
&\leq
-\left[
\tilde{\sigma}+\mu
-\bigl(|\tilde{\sigma}+\tilde{\rho}|+K\bigr)^2
-\frac{K^2}{\tilde{\beta}}
\right]\tilde{x}^2
-\frac34\tilde{y}^2
-\tilde{\beta}\tilde{z}^2
+
\frac{\mu^2\tilde{\beta}}{4}.
\end{aligned}
\]
By the assumption \eqref{eq:mu_condition_mismatch},
\[
\mu>
\bigl(|\tilde{\sigma}+\tilde{\rho}|+K\bigr)^2
+\frac{K^2}{\tilde{\beta}}
+\frac14
-\tilde{\sigma},
\]
and hence
\[
\tilde{\sigma}+\mu
-\bigl(|\tilde{\sigma}+\tilde{\rho}|+K\bigr)^2
-\frac{K^2}{\tilde{\beta}}
>
\frac14.
\]
Thus
\[
\dot V
\leq
-\frac14\tilde{x}^2
-\frac34\tilde{y}^2
-\tilde{\beta}\tilde{z}^2
+
\frac{\mu^2\tilde{\beta}}{4}.
\]
Let
\[
c:=\min\left\{\frac14,\frac34,\tilde{\beta}\right\}>0,
\qquad
C:=\frac{\mu^2\tilde{\beta}}{4}.
\]
Then
\[
\dot V
\leq
-c\left(\tilde{x}^2+\tilde{y}^2+\tilde{z}^2\right)+C
=
-2cV+C.
\]
By Gr\"onwall's inequality,
\[
V(t)
\leq
e^{-2ct}V(0)
+
\frac{C}{2c}\left(1-e^{-2ct}\right),
\qquad t\in I_{\tau,T}.
\]
Hence \(\tilde X(t)\) remains bounded on \(I_{\tau,T}\). Therefore the maximal
solution cannot blow up before time \(\tau+T\), and so the solution exists
uniquely on the entire interval \(I_{\tau,T}\).

We now prove continuous dependence. Let
\(\tilde X_1\) and \(\tilde X_2\) be two solutions corresponding to initial data
\(\tilde X_{1,0}\), \(\tilde X_{2,0}\) and trial parameters
\(\tilde{\theta}_1,\tilde{\theta}_2\in U\). The boundedness estimate above shows
that both solutions remain in a bounded subset of \(\mathbb{R}^3\) on
\(I_{\tau,T}\). On this bounded set, \(f\) is uniformly Lipschitz in
\(\tilde X\) and uniformly continuous in \(\tilde{\theta}\). Therefore,
\[
\begin{aligned}
\lVert \tilde X_1(t)-\tilde X_2(t)\rVert
&\leq
\lVert \tilde X_{1,0}-\tilde X_{2,0}\rVert
+
L\int_0^t
\lVert \tilde X_1(s)-\tilde X_2(s)\rVert\,ds  \\
&\quad+
\int_0^t
\lVert f(s,\tilde X_2(s),\tilde{\theta}_1)
-
f(s,\tilde X_2(s),\tilde{\theta}_2)\rVert\,ds .
\end{aligned}
\]
The last integral tends to zero as
\(\tilde{\theta}_1\to\tilde{\theta}_2\), uniformly for \(t\in I_{\tau,T}\).
Applying Gr\"onwall's inequality gives
\[
\sup_{t\in I_{\tau,T}}
\lVert \tilde X_1(t)-\tilde X_2(t)\rVert
\to 0, \quad \text{whenever} \quad \tilde X_{1,0}\to \tilde X_{2,0},
\ \tilde{\theta}_1\to\tilde{\theta}_2.
\]
Thus the solution depends continuously on the initial data and on
\(\tilde{\theta}\).

Finally, since \(f\) is \(C^1\) in \((\tilde X,\tilde{\theta})\), the standard
\(C^1\)-dependence theorem for ordinary differential equations, (see \cite[Theorem~2]{Perko}) implies that the solution map $\tilde{\theta}\mapsto \tilde X(\cdot;\tilde{\theta})$
is \(C^1\) from \(U\) into \(C(I_{\tau,T};\mathbb{R}^3)\). Since the initial
condition of the nudged system is independent of \(\tilde{\theta}\),
differentiating the initial condition gives zero initial data for each partial
derivative. Therefore the partial derivatives with respect to
\(\tilde{\sigma}\), \(\tilde{\rho}\), and \(\tilde{\beta}\) exist and are
continuous on \(I_{\tau,T}\).
\end{proof}

%------------------------ Equations for the partial derivatives ------- %

\subsection{Proof of Proposition \ref{prop:partial_derivative_equations}}
 {\em Proof of Part (i):} \label{A.1}
By Lemma~\ref{lem:wp_c1_fulltheta}, the solution of the nudged system
\eqref{nudgedsystem} is \(C^1\) with respect to
\(\tilde{\theta}=(\tilde{\sigma},\tilde{\rho},\tilde{\beta})\). Therefore the
partial derivatives with respect to each component of \(\tilde{\theta}\) exist
and are obtained by differentiating \eqref{nudgedsystem}. First, differentiating $\dot{\tilde{x}}
=
\tilde{\sigma}(\tilde{y}-\tilde{x})
+
\mu(x-\tilde{x})$
with respect to \(\tilde{\sigma}\) gives
\[
p_\sigma'
=
(\tilde{y}-\tilde{x})
+
\tilde{\sigma}(q_\sigma-p_\sigma)
-
\mu p_\sigma.
\]
Evaluating at \(\tilde{\theta}=\theta^*\) yields
\[
p_\sigma'
=
(\tilde{y}^*-\tilde{x}^*)
-(\sigma^*+\mu)p_\sigma
+\sigma^*q_\sigma.
\]
Next, differentiating $\dot{\tilde{y}}
=
\tilde{x}(\tilde{\rho}-\tilde{z})
-
\tilde{y}$
with respect to \(\tilde{\sigma}\) gives
\[
q_\sigma'
=
p_\sigma(\tilde{\rho}-\tilde{z})
-
\tilde{x}r_\sigma
-
q_\sigma,
\]
and hence, at \(\tilde{\theta}=\theta^*\),
\[
q_\sigma'
=
(\rho^*-\tilde{z}^*)p_\sigma
-
\tilde{x}^*r_\sigma
-
q_\sigma.
\]
Similarly, differentiating $\dot{\tilde{z}}
=
\tilde{x}\tilde{y}
-
\tilde{\beta}\tilde{z}$
with respect to \(\tilde{\sigma}\) gives
\[
r_\sigma'
=
\tilde{y}p_\sigma
+
\tilde{x}q_\sigma
-
\tilde{\beta}r_\sigma,
\]
and therefore, at \(\tilde{\theta}=\theta^*\)
\[
r_\sigma'
=
\tilde{y}^*p_\sigma
+
\tilde{x}^*q_\sigma
-
\beta^*r_\sigma.
\]

The equations for \((p_\rho,q_\rho,r_\rho)\) are obtained in the same way.
Since the first equation of \eqref{nudgedsystem} has no explicit
\(\tilde{\rho}\)-dependence,
\[
p_\rho'
=
\tilde{\sigma}(q_\rho-p_\rho)
-
\mu p_\rho.
\]
Evaluating at \(\tilde{\theta}=\theta^*\) gives
\[
p_\rho'
=
-(\sigma^*+\mu)p_\rho
+
\sigma^*q_\rho.
\]
The second equation contains the explicit factor \(\tilde{\rho}\). Thus
\[
q_\rho'
=
\tilde{x}
+
(\tilde{\rho}-\tilde{z})p_\rho
-
\tilde{x}r_\rho
-
q_\rho,
\]
and hence,  at \(\tilde{\theta}=\theta^*\)
\[
q_\rho'
=
\tilde{x}^*
+
(\rho^*-\tilde{z}^*)p_\rho
-
\tilde{x}^*r_\rho
-
q_\rho.
\]
The third equation gives
\[
r_\rho'
=
\tilde{y}p_\rho
+
\tilde{x}q_\rho
-
\tilde{\beta}r_\rho,
\]
and therefore,  at \(\tilde{\theta}=\theta^*\)
\[
r_\rho'
=
\tilde{y}^*p_\rho
+
\tilde{x}^*q_\rho
-
\beta^*r_\rho.
\]

Finally, differentiating with respect to \(\tilde{\beta}\), the first equation
has no explicit \(\tilde{\beta}\)-dependence, so
\[
p_\beta'
=
\tilde{\sigma}(q_\beta-p_\beta)
-
\mu p_\beta.
\]
At \(\tilde{\theta}=\theta^*\), this becomes
\[
p_\beta'
=
-(\sigma^*+\mu)p_\beta
+
\sigma^*q_\beta.
\]
The second equation gives
\[
q_\beta'
=
(\tilde{\rho}-\tilde{z})p_\beta
-
\tilde{x}r_\beta
-
q_\beta,
\]
and hence,  at \(\tilde{\theta}=\theta^*\)
\[
q_\beta'
=
(\rho^*-\tilde{z}^*)p_\beta
-
\tilde{x}^*r_\beta
-
q_\beta.
\]
For the third equation, the parameter \(\tilde{\beta}\) appears explicitly:
\[
\dot{\tilde{z}}
=
\tilde{x}\tilde{y}
-
\tilde{\beta}\tilde{z}.
\]
Differentiating with respect to \(\tilde{\beta}\) gives
\[
r_\beta'
=
\tilde{y}p_\beta
+
\tilde{x}q_\beta
-
\tilde{\beta}r_\beta
-
\tilde{z}.
\]
Evaluating at \(\tilde{\theta}=\theta^*\) yields
\[
r_\beta'
=
\tilde{y}^*p_\beta
+
\tilde{x}^*q_\beta
-
\beta^*r_\beta
-
\tilde{z}^*.
\]
Since the initial condition of the nudged system is fixed independently of the
trial parameter vector \(\tilde{\theta}\), we have $\tilde X(0;\tilde{\theta})
=
(\tilde{x}_0,\tilde{y}_0,\tilde{z}_0)$ for all \(\tilde{\theta}\in U\). Hence the map
\(\tilde{\theta}\mapsto \tilde X(0;\tilde{\theta})\) is constant. Differentiating
this identity with respect to \(\tilde{\sigma}\), \(\tilde{\rho}\), and
\(\tilde{\beta}\) gives
\[
\partial_{\tilde{\sigma}}\tilde X(0;\tilde{\theta})=0,
\qquad
\partial_{\tilde{\rho}}\tilde X(0;\tilde{\theta})=0,
\qquad
\partial_{\tilde{\beta}}\tilde X(0;\tilde{\theta})=0.
\]
Evaluating at \(\tilde{\theta}=\theta^*\), this yields
\[
(p_j,q_j,r_j)(0)=(0,0,0),
\qquad
j\in\{\sigma,\rho,\beta\}.
\]
This completes the proof.
%\end{proof}

%------------------------ Gram matrix lower bound ------- %
\comments{
 \begin{lemma}[Gram matrix lower bound]
\label{lem:gram_matrix_lower_bound}
Define ${\mathcal F}(\tilde{\theta})(t)
=
x(t)-\tilde{x}(t;\tilde{\theta}),
\ t\in[\tau,\tau+T].$
Let
\[
\langle f,g\rangle_{\tau,T}
:=
\int_{\tau}^{\tau+T} f(t)g(t)\,dt.
\]
Define the Gram matrix
\[
G_{\tau,T}
=
\begin{pmatrix}
\langle p_\sigma,p_\sigma\rangle_{\tau,T} &
\langle p_\sigma,p_\rho\rangle_{\tau,T} &
\langle p_\sigma,p_\beta\rangle_{\tau,T} \\[2mm]
\langle p_\rho,p_\sigma\rangle_{\tau,T} &
\langle p_\rho,p_\rho\rangle_{\tau,T} &
\langle p_\rho,p_\beta\rangle_{\tau,T} \\[2mm]
\langle p_\beta,p_\sigma\rangle_{\tau,T} &
\langle p_\beta,p_\rho\rangle_{\tau,T} &
\langle p_\beta,p_\beta\rangle_{\tau,T}
\end{pmatrix},
\]
where \(p_\sigma,p_\rho,p_\beta\) are the \(x\)-component partial derivatives
from Lemma~\ref{lem:partial_derivative_equations}. Assume that
\(G_{\tau,T}\) is positive definite. Then, for every
\(h=(h_\sigma,h_\rho,h_\beta)\in\mathbb{R}^3\),
\[
\lVert D{\mathcal F}(\theta^*)h\rVert_{L^2(\tau,\tau+T)}
\geq
\sqrt{\lambda_{\min}(G_{\tau,T})}
\lVert h\rVert_{\mathbb{R}^3}.
\]
\end{lemma}
}

{\em Proof of Part (ii):}
%\begin{proof}
By Lemma~\ref{lem:wp_c1_fulltheta}, the map $\tilde{\theta}\mapsto \tilde{x}(\cdot;\tilde{\theta})$
is \(C^1\) from \(U\) into \(C(I_{\tau,T})\). Restricting to
\([\tau,\tau+T]\) and using the continuous embedding
\[
C([\tau,\tau+T])\hookrightarrow L^2(\tau,\tau+T),
\]
it follows that ${\mathcal F}:U\to L^2(\tau,\tau+T)$ is \(C^1\). Since \(x(t)\) is the fixed observed trajectory and does not depend on \(\tilde{\theta}\), we have
\[
D{\mathcal F}(\theta^*)h
=
-D_{\tilde{\theta}}\tilde{x}(t;\theta^*)h.
\]
For \(h=(h_\sigma,h_\rho,h_\beta)\), linearity of the derivative gives
\[
D_{\tilde{\theta}}\tilde{x}(t;\theta^*)h
=
h_\sigma p_\sigma(t)
+
h_\rho p_\rho(t)
+
h_\beta p_\beta(t).
\]
Therefore,
\[
D{\mathcal F}(\theta^*)h
=
-\left(
h_\sigma p_\sigma
+
h_\rho p_\rho
+
h_\beta p_\beta
\right).
\]
The sign does not affect the \(L^2\)-norm, and so
\[
\begin{aligned}
\lVert D{\mathcal F}(\theta^*)h\rVert_{L^2(\tau,\tau+T)}^2
&=
\int_{\tau}^{\tau+T}
\left|
h_\sigma p_\sigma(t)
+
h_\rho p_\rho(t)
+
h_\beta p_\beta(t)
\right|^2\,dt.
\end{aligned}
\]
Expanding the square gives
\[
\begin{aligned}
\lVert D{\mathcal F}(\theta^*)h\rVert_{L^2(\tau,\tau+T)}^2
&=
h_\sigma^2\langle p_\sigma,p_\sigma\rangle_{\tau,T}
+
h_\rho^2\langle p_\rho,p_\rho\rangle_{\tau,T}
+
h_\beta^2\langle p_\beta,p_\beta\rangle_{\tau,T}
\\
&\quad
+
2h_\sigma h_\rho
\langle p_\sigma,p_\rho\rangle_{\tau,T}
+
2h_\sigma h_\beta
\langle p_\sigma,p_\beta\rangle_{\tau,T}
+
2h_\rho h_\beta
\langle p_\rho,p_\beta\rangle_{\tau,T}.
\end{aligned}
\]
By the definition of \(G_{\tau,T}\), the right-hand side is exactly
\[
h^\top G_{\tau,T}h.
\]
Thus
\[
\lVert D{\mathcal F}(\theta^*)h\rVert_{L^2(\tau,\tau+T)}^2
=
h^\top G_{\tau,T}h.
\]
Since \(G_{\tau,T}\) is positive definite, its smallest eigenvalue is positive.
By the Rayleigh quotient characterization,
\[
h^\top G_{\tau,T}h
\geq
\lambda_{\min}(G_{\tau,T})
\lVert h\rVert_{\mathbb{R}^3}^2.
\]
Therefore,
\[
\lVert D{\mathcal F}(\theta^*)h\rVert_{L^2(\tau,\tau+T)}^2
\geq
\lambda_{\min}(G_{\tau,T})
\lVert h\rVert_{\mathbb{R}^3}^2.
\]
Taking square roots yields
\[
\lVert D{\mathcal F}(\theta^*)h\rVert_{L^2(\tau,\tau+T)}
\geq
\sqrt{\lambda_{\min}(G_{\tau,T})}
\lVert h\rVert_{\mathbb{R}^3}.
\]
This completes the proof.
%\end{proof}

%===========================================================================

\newpage 

\bibliographystyle{unsrturl}
\bibliography{references}

 \end{document}